\documentclass[11pt]{article}
\usepackage[margin=1in]{geometry}
\usepackage{amsmath,empheq,bm}%
\usepackage{amsfonts}%
\usepackage{amssymb}%
\usepackage{amsthm}
\usepackage{graphicx}
\usepackage{wrapfig}
\usepackage{subfig}
\usepackage{color}
\usepackage{bbm}

\usepackage{enumitem}
\usepackage{algpseudocode}
\usepackage{algorithmicx}
\usepackage{algorithm}
\usepackage{bbm}
\usepackage{xspace}

\newcommand{\jvyb}[1]{\textcolor{black}{#1}}
\newtheorem{theorem}{Theorem}

\newtheorem{corollary}{Corollary}

\newtheorem{definition}{Definition}

\newtheorem{lemma}{Lemma}

\newtheorem{proposition}{Proposition}
\newtheorem{remark}{Remark}

\numberwithin{equation}{section}

\newcommand{\calA}{\ensuremath{\mathcal{A}}}
\newcommand{\calB}{\ensuremath{\mathcal{B}}}

\newcommand{\calH}{\ensuremath{\mathcal{H}}}

\newcommand{\calG}{\ensuremath{\mathcal{G}}}

\newcommand{\calP}{\ensuremath{\mathcal{P}}}

\newcommand{\calS}{\ensuremath{\mathcal{S}}}

\newcommand{\calN}{\ensuremath{\mathcal{N}}}

\newcommand{\norm}[1]{\|{#1}\|}
\newcommand{\abs}[1]{|{#1}|}

\newcommand{\set}[1]{\left\{{#1}\right\}}
\newcommand{\dotprod}[2]{\langle#1,#2\rangle}
\newcommand{\est}[1]{\widehat{#1}}
\newcommand{\expec}{\ensuremath{\mathbb{E}}}

\newcommand{\matR}{\ensuremath{\mathbb{R}}}
\newcommand{\matZ}{\ensuremath{\mathbb{Z}}}
\newcommand{\matN}{\ensuremath{\mathbb{N}}}

\newcommand{\prob}{\ensuremath{\mathbb{P}}}

\newcommand{\indic}{\ensuremath{\mathbbm{1}}} 

\newcommand{\vecx}{\mathbf x}
\newcommand{\vece}{\mathbf e}

\newcommand{\veci}{\mathbf i}
\newcommand{\vecj}{\mathbf j}
\newcommand{\vecu}{\mathbf u}

\newcommand{\vecy}{\mathbf y}

\newcommand{\vecz}{\mathbf z}
\newcommand{\veca}{\mathbf a}

\newcommand{\matA}{\ensuremath{\mathbf{A}}}
\newcommand{\matB}{\ensuremath{\mathbf{B}}}

\newcommand{\ftil}{\ensuremath{\tilde{f}}}

\newcommand{\sampsgnvec}{\ensuremath{\bm{\beta}}}
\newcommand{\sampsgn}{\ensuremath{\beta}}		
\newcommand{\sampsgnbar}{\ensuremath{\bar{\sampsgn}}}		

\newcommand{\betavect}{\ensuremath{\sampsgnvec^{(2)}}}
\newcommand{\betat}{\ensuremath{\sampsgn^{(2)}}}
\newcommand{\betavecp}{\ensuremath{\sampsgnvec^{(p)}}}
\newcommand{\arbnoisebd}{\ensuremath{\nu}}
\newcommand{\riptt}{\ensuremath{\delta}} 
\newcommand{\riptoub}{\ensuremath{\gamma^{\mathrm{ub}}}} 
\newcommand{\riptolb}{\ensuremath{\gamma^{\mathrm{lb}}}} 
\newcommand{\dchp}{\ensuremath{{D}}} 

\newcommand{\vech}{\mathbf h}
\newcommand{\noisebd}{\ensuremath{\triangle}} 

\newcommand{\xposvec}{\ensuremath{\vecx^{+}}}
\newcommand{\xnegvec}{\ensuremath{\vecx^{-}}}
\newcommand{\xpos}{\ensuremath{x^{+}}}
\newcommand{\xneg}{\ensuremath{x^{-}}}

\newcommand{\origsparsvec}{\ensuremath{\vecz^{*}}}
\newcommand{\origspars}{\ensuremath{z^{*}}}

\newcommand{\csalg}{\ensuremath{\texttt{SPARSE-REC}}}

\newcommand{\setproj}{\ensuremath{\Pi}}
\newcommand{\numbasepts}{\ensuremath{m}}

\newcommand{\noise}{\ensuremath{\eta}}
\newcommand{\noisevec}{\ensuremath{\bm{\eta}}}

\newcommand{\totsparsity}{\ensuremath{k}}

\newcommand{\univsupp}{\ensuremath{\calS_1}} 
\newcommand{\bivsupp}{\ensuremath{\calS_2}} 
\newcommand{\bivsuppvar}{\ensuremath{\calS_2^{\rm{var}}}} 
\newcommand{\totsupp}{\ensuremath{\calS}}

\newcommand{\order}{\ensuremath{r_0}}
\newcommand{\betavecord}{\ensuremath{\sampsgnvec^{(\order)}}}
\newcommand{\betaord}{\ensuremath{\sampsgn^{(\order)}}}
\newcommand{\hashfam}{\ensuremath{\calH}}

\newcommand{\modmean}{\ensuremath{\mu}} 
\newcommand{\dimn}{\ensuremath{d}}

\newcommand{\idenconst}{\ensuremath{D}}
\newcommand{\baseset}{\ensuremath{\chi}}
\newcommand{\twohashfam}{\ensuremath{\calH_{2}^{d}}}
\newcommand{\thashfam}{\ensuremath{\calH_{t}^{d}}}

\newcommand{\hashfn}{\ensuremath{h}}
\newcommand{\canvec}{\ensuremath{\mathbf{e}}}

\newcommand{\digit}{\ensuremath{\mathtt{digit}}}
\newcommand{\randvar}{Rademacher } 

\title{Learning \jvyb{general} sparse additive models\\ from point queries in high dimensions}

\author{Hemant Tyagi\thanks{INRIA Lille-Nord Europe, France. This work was done by the author while affiliated to the Alan Turing Institute, London, United Kingdom, 
and the School of Mathematics, University of Edinburgh, United Kingdom; this author was supported by EPSRC grant EP/N510129/1 at the Alan Turing Institute.}\\ \texttt{hemant.tyagi@inria.fr} \and 
Jan Vybiral\thanks{Dept. of Mathematics FNSPE, Czech Technical University in Prague, Trojanova 13, 12000 Prague, Czech Republic;
this author was supported by the grant P201/18/00580S of the Grant Agency of the Czech Republic,
by the grant 8X17028 of the Czech Ministry of Education, and by the Neuron Fund for Support of Science.}\\
\texttt{jan.vybiral@fjfi.cvut.cz}} 

\begin{document}
\maketitle

\begin{abstract}
We consider the problem of learning a $d$-variate function $f$ defined on the cube $[-1,1]^d\subset \matR^d$,
where the algorithm is assumed to have black box access to samples of $f$ within this domain.
Denote $\totsupp_r \subset {[d] \choose r}; r=1,\dots,\order$ 
to be sets consisting of unknown $r$-wise interactions amongst the coordinate variables. 
We then focus on the setting where $f$ has an additive structure, i.e., 
it can be represented as 
$$f = \sum_{\vecj \in \totsupp_1} \phi_{\vecj} + \sum_{\vecj \in \totsupp_2} \phi_{\vecj} + \dots 
+ \sum_{\vecj \in \totsupp_{\order}} \phi_{\vecj},$$
where each $\phi_{\vecj}$; $\vecj \in \totsupp_r$ is at most $r$-variate for $1 \leq r \leq \order$.  
We derive randomized algorithms that query $f$ at carefully constructed set of points, and exactly 
recover each $\totsupp_r$ with high probability. In contrary to the previous work, our analysis does not rely
on numerical approximation of derivatives by finite order differences.
\end{abstract}

\section{Introduction} \label{sec:intro}
%
%
%

Approximating a function from its samples is a fundamental problem with rich theory developed in areas such as numerical analysis 
and statistics, and which also has numerous practical applications such as in systems biology \cite{Goel08}, solving PDEs \cite{CohenPDE2010}, 
control systems \cite{ZhuControl16}, optimization \cite{Shan2010} etc. 
Concretely, for an unknown $d$-variate function $f: \calG \rightarrow \matR$, one is given information about $f$ in the form of 
samples $(\vecx_i,f(\vecx_i))_{i=1}^n$. Here, the $\vecx_i$'s belong to a compact subset $\calG \subset \matR^d$. 
The goal is to construct a smooth estimate $\est{f} : \calG \rightarrow \matR$ such that the error between $\est{f}$ and $f$ is small. 
In this paper we focus on the high dimensional setting where $d$ is large.
We will consider the 
scenario where the algorithm has black box access to the function, and can query it at any point within $\calG$.
This setting appears for instance in materials science \cite{Luca2015}, where $\vecx$ represents some material and $f(\vecx)$
some of its properties of interest (like thermal or electric conductivity). The local-density approximations in density functional theory can be used to
compute to high accuracy such properties of a given material. The sampling then corresponds to running a costly numerical PDE-solver.
Since such simulations are typically 
expensive to run, one would like to minimize the number of queries made. This setting is different from the regression setting typically considered
in statistics wherein the $\vecx_i$'s are generated apriori from some unknown distribution over $\calG$.

\paragraph{Curse of dimensionality.} It is well known that provided we only make smoothness assumptions on 
$f$ (such as differentiability or Lipschitz continuity), then 
the problem is intractable, i.e.,  has exponential complexity (in the worst case) with respect to the dimension $d$.
For instance if $f \in C^{s}(\calG)$, then any algorithm 
needs in the worst case $n = \Omega(\delta^{-d/s})$ samples\footnote{\jvyb{This means that there exists a constant $c > 0$ such that 
$n \geq c \delta^{-d/s}$ when $d$ is sufficiently large. See Section \ref{sec:problem} for a formal definition.}} 
to uniformly approximate $f$ with error $\delta \in (0,1)$, cf. \cite{Novak2006,Vybiral2007}.
Furthermore, the constants behind the $\Omega$-notation may also depend on $d$.
A detailed study of the dependence on $d$ was performed in the field of \emph{Information Based Complexity}
for $f \in C^{\infty}(\calG)$ in a more recent work \cite{Novak2009}. The authors show that even here,
$n = \Omega(2^{\lfloor d/2 \rfloor})$ samples are needed in the worst case for uniform approximation within an error $\delta \in (0,1)$
(with no additional dependence on $d$ hidden behind the $\Omega$-notation).
This exponential dependence on $d$ is commonly referred to as the \emph{curse of dimensionality}. 
The above results suggest that in order to get tractable algorithms in the high dimensional regime, one needs to make additional assumptions 
on $f$. To this end, a growing line of work over the past decade has focused on the setting where $f$ possesses an intrinsic, albeit unknown, low dimensional structure, 
with much smaller intrinsic dimension than the ambient dimension $d$. The motivation is that one could now hope to design algorithms with complexity
at most exponential in the intrinsic dimension, but with mild dependence on $d$. 

\subsection{Sparse additive models (SPAMs)} 
A popular class of functions with an intrinsic low dimensional structure are the so-called sparse additive 
models
(SPAMs). These are functions that are decomposable as the sum of a small number of lower dimensional functions.
\jvyb{To give the formal definition, we denote $[d]=\{1,\dots,d\}$ and by ${[d]\choose r}$ we mean the collection
of all ordered $r$-tuples from $[d]$. Then,}
for $\totsupp_r \subset {[d] \choose r}; r=1,\dots,\order$, the function $f:\calG \rightarrow \matR$ is of the form
\begin{equation} \label{eq:gen_spam_eq}
f = \sum_{j \in \totsupp_1} \phi_{j} (x_j) + \sum_{(j_1,j_2) \in \totsupp_2} \phi_{(j_1,j_2)} (x_{j_1},x_{j_2}) + \dots 
+ \sum_{(j_1,\dots,j_{\order}) \in \totsupp_{\order}} \phi_{(j_1,\dots,j_{\order})} (x_{j_1},\dots,x_{j_{\order}})
\end{equation}
with each $\abs{\totsupp_r} \ll {d \choose r}$, and $\order \ll d$. We can interpret the tuples in $\totsupp_r$ as 
$r^{th}$ order interactions terms. Let us remark, that usually the terminology Sparse additive models is used for
the case $\order = 1$, but we prefer to use it here in the general sense of \eqref{eq:gen_spam_eq}.

These models appear in optimization under the name 
\emph{partially separable} models (cf., \cite{Griewank1982}). They also arise in 
electronic structure computations in physics (cf., \cite{Blochl90}), and problems involving
multiagent systems represented as decentralized partially observable Markov decision processes (cf., \cite{Dibangoye14}). 
There exists a rich line of work that mostly study special cases of the model \eqref{eq:gen_spam_eq}. We  review them briefly below, 
leaving a detailed comparison with our results to Section \ref{sec:disc_conc_rems}.

\paragraph{The case $\order = 1$.} In this setting, \eqref{eq:gen_spam_eq} reduces to a sparse sum of univariate functions. This model has 
been studied extensively in the non parametric statistics literature with a range of results on 
estimation of $f$ (cf., \cite{Huang2010,Koltch08,Koltch2010,Meier2009,Raskutti2012,Ravi2009}) 
and also on variable selection, i.e., identifying the support $\totsupp_1$ (cf., \cite{Huang2010,Ravi2009,Wahl15}). 
The basic idea behind these approaches is to approximately represent each $\phi_j$ in a suitable basis of finite size (for eg., splines or wavelets)
and then to find the coefficients in the basis expansion by solving a least squares problem with 
smoothness and sparsity penalty constraints. Koltchinskii et al. \cite{Koltch2010} and Raskutti et al. \cite{Raskutti2012} 
proposed a convex program for estimating $f$ in the Reproducing kernel 
Hilbert space (RKHS) setting, and showed that $f$ lying in a Sobolev space with 
smoothness parameter $\alpha > 1/2$ can be estimated at the $L_2$ rate $\frac{\totsparsity \log d}{n} + \totsparsity n^{-\frac{2\alpha}{2\alpha+1}}$. 
This rate was shown to be optimal in \cite{Raskutti2012}. There also exist results for variable selection, i.e., identifying the 
support $\totsupp_1$. These results in non parametric statistics are typically asymptotic in the limit of large $n$, also referred to 
as sparsistency \cite{Huang2010,Ravi2009,Wahl15}. Recently, Tyagi et al. \cite{Tyagi14_nips} derived algorithms that query $f$, along with 
non-asymptotic sampling bounds for identifying $\totsupp_1$. They essentially estimate the (sparse) gradient of $f$ using results from 
compressed sensing (CS), at few carefully chosen locations in $\calG$.

\paragraph{The case $\order = 2$.} This setup has received relatively less attention than the aforementioned setting. 
Radchenko et al. \cite{Rad2010} proposed an algorithm VANISH, and showed that it 
is sparsistent, i.e., recovers $\univsupp,\bivsupp$ in the limit of large $n$. 
The ACOSSO algorithm \cite{Storlie2011} can handle this setting, with theoretical guarantees (sparsistency, convergence rates) 
shown when $\order = 1$. Recently, Tyagi et al. \cite{Tyagi_aistats16,Tyagi_spamint_long16} derived algorithms that query $f$, 
and derived non-asymptotic sampling bounds for recovering $\univsupp,\bivsupp$. Their approach for recovering $\bivsupp$ was based on estimating the 
(sparse) Hessian of $f$ using results from CS, at carefully chosen points in $\calG$. 
The special case where $f$ is \emph{multilinear} has been studied considerably; there exist algorithms that recover $\univsupp,\bivsupp$, 
along with convergence rates for estimating $f$ in the limit of large $n$ \cite{Choi2010,Rad2010,Bien2013}. There also exist 
non-asymptotic sampling bounds for identifying $\univsupp,\bivsupp$ in the noiseless setting (cf., \cite{Nazer2010,Kekatos11}); 
these works essentially make use of the CS framework.

\paragraph{The general case.} Much less is known about the general setup where $\order \geq 2$ is possible. 
Lin et al. \cite{Lin2006} were the first to introduce learning SPAMs of the form \eqref{eq:gen_spam_eq}, and 
proposed the COSSO algorithm. Recently, Dalalyan et al. \cite{Dala2014} and Yang et al. \cite{Yang2015} 
studied \eqref{eq:gen_spam_eq} in the regression setting and derived non-asymptotic error rates for estimating $f$. 
In particular, Dalalyan et al. studied this in the Gaussian white noise 
model,
while Yang et al. considered the Bayesian setup wherein a Gaussian process (GP) prior is placed on $f$. When $f$ is multilinear, 
the work of Nazer et al. \cite{Nazer2010}, which is in the CS framework, gives non-asymptotic sampling bounds for recovering $\totsupp_r$, $r=1,\dots,\order$. 

\subsection{Our contributions and main idea}
Before proceeding, we will briefly mention our problem setup to put our results in the context; it is described more formally later on in Section \ref{sec:problem}. 
We consider $f:[-1,1]^d \rightarrow \matR$ of the form \eqref{eq:gen_spam_eq} and denote by ${\mathcal S}_j^{(1)}$
the variables occurring in ${\mathcal S}_j$. We assume, that ${\mathcal S}^{(1)}_j$ are disjoint\footnote{For $\order = 2$,
this represents no additional assumption. See discussion after Proposition \ref{prop:mod_unique_bivar}.}
for $1\le j\le\order.$
Each component $\phi$ is assumed to be H\"older smooth, and is also assumed to be ``sufficiently large'' at some point within its domain. 
Our goal is to query $f$ at few locations in $\calG = [-1,1]^d$, and recover the underlying sets of interactions $\totsupp_r$, for each $r=1,\dots,\order$.

\paragraph{Our results.}
To our knowledge, we provide the first non-asymptotic sampling bounds for exact identification 
of $\totsupp_r$, for each $r=1,\dots,\order$, for SPAMs of the form \eqref{eq:main_sam_compl_res}. 
In particular, we derive a randomized algorithm that with high probability recovers 
each $\totsupp_r$, $r=1,\dots,\order$ with
\begin{equation} \label{eq:main_sam_compl_res}
\Omega\Biggl(\sum_{i=3}^{\order}\Bigl[\underbrace{c_i^i \jvyb{i^{2}} \abs{\totsupp_i}^2 \log^2 d}_{\text{Identifying } \totsupp_i} \Bigr] 
+ \underbrace{c_2 \abs{\bivsupp} \log\left(\frac{d^2}{\abs{\bivsupp}}\right) \log d}_{\text{Identifying } \totsupp_2}
+  \underbrace{c_1 \abs{\univsupp} \log \left(\frac{d}{\abs{\univsupp}}\right)}_{\text{Identifying } \totsupp_1}\Biggr)
\end{equation}
noiseless queries of $f$ within $[-1,1]^d$. The same bound holds when the queries are corrupted with
arbitrary bounded noise provided the noise magnitude is sufficiently 
small (see Theorem \ref{thm:main_multiv_arbnoise} \jvyb{and Remark \ref{rem:samp_compl_fin_multiv}}). 
\jvyb{Here, the $c_i$'s depend on the smoothness parameters of the $\phi_i$'s and scale as $\sqrt{i}$ with $i$.
In the setting of i.i.d. Gaussian noise, which we handle by resampling each query sufficiently many times, and averaging, 
we obtain a similar sample complexity as \eqref{eq:main_sam_compl_res} with additional factors 
depending on the variance of the noise (see Theorem \ref{thm:main_multiv_gauss_noise})}.  

We improve on the recent work of Tyagi et al. \cite{Tyagi_aistats16,Tyagi_spamint_long16}, wherein SPAMs with $\order = 2$ were considered, 
by being able to handle general $\order \geq 1$.  Moreover, we only require $f$ to be H\"older smooth while the algorithms in 
\cite{Tyagi_aistats16,Tyagi_spamint_long16} necessarily require $f$ to be continuously differentiable.
Finally, our bounds improve upon those in \cite{Tyagi_aistats16,Tyagi_spamint_long16} when the noise is i.i.d. Gaussian. 
In this scenario, our bounds are linear in the sparsity $\abs{\totsupp_2} + \abs{\totsupp_1}$ while those 
in \cite{Tyagi_aistats16,Tyagi_spamint_long16} are polynomial in the sparsity.  

The sampling scheme that we employ to achieve these bounds is novel, and is specifically tailored to the additive nature of $f$. 
We believe this scheme to be of independent interest for other problems involving additive models, such as in optimization of 
high dimensional functions with partially separable structure.

\paragraph{Main idea.} We identify each set $\totsupp_i$ in a sequential ``top down'' manner by first 
identifying $\totsupp_{\order}$. Once we find $\totsupp_{\order}$, the same procedure is repeated on the 
remaining set of variables (excluding those found in $\totsupp_{\order}$) to identify $\totsupp_{\order-1}$, and consequently, each 
remaining $\totsupp_i$. We essentially perform the following steps for recovering $\totsupp_{\order}$.
Consider some given partition of $[d]$ into $\order$ disjoint subsets $\calA = (\calA_1,\dots,\calA_{\order})$, 
a \randvar vector $\sampsgnvec \in \set{-1,1}^d$ and some given $\vecx \in [-1,1]^d$. 
We generate $2^{\order}$ query points $(\vecx_i)_{i=1}^{2^{\order}}$, where each $\vecx_i$ is constructed using $\sampsgnvec,\vecx$ and $\calA$. 
Then, for some fixed sequence of signs $s_1,s_2,\dots,s_{2^{\order}} \in \set{-1,1}$ (depending only on $\order$), we 
show for the anchored-ANOVA representation of $f$ (see Section \ref{sec:problem} and Lemma \ref{lem:samp_motiv_multiv}) that
\begin{align} \label{eq:main_mult_samp_obs}
\sum_{i=1}^{2^{\order}} s_i f(\vecx_i)  = 
\sum_{(j_1,\dots,j_{\order}) \in \calA \cap \totsupp_{\order}} \beta_{j_1}\dots \beta_{j_{\order}} \phi_{(j_1,\dots,j_{\order})} (x_{j_1},\dots,x_{j_{\order}}).
\end{align}
Observe, that \eqref{eq:main_mult_samp_obs} corresponds to a \emph{multilinear} measurement of a \emph{sparse} vector 
with entries $\phi_{(j_1,\dots,j_{\order})} (x_{j_1},\dots,x_{j_{\order}})$, indexed by the tuple $(j_1,\dots,j_{\order})$. 
Indeed, this vector is $\abs{\totsupp_{\order}}$ sparse. This suggests that by repeating the above process at sufficiently many 
random $\sampsgnvec$'s, we can recover an estimate of this $\abs{\totsupp_{\order}}$ vector by using known results from CS. 
Thereafter, we repeat the above process for each $\calA$ corresponding to a family of perfect hash functions (see Definition \ref{def:thash_fam}). 
The size of this set is importantly at most exponential in $\order$, and only logarithmic in $d$. The $\vecx$'s are then chosen to be points 
on a uniform $\order$ dimensional grid constructed using $\calA$. This essentially enables us to guarantee that we are able to sample 
each $\phi_{(j_1,\dots,j_{\order})}$ sufficiently fine within its domain, and thus identify $(j_1,\dots,j_{\order})$ by thresholding.

\paragraph{Organization of paper.} The rest of the paper is organized as follows. In Section \ref{sec:problem}, we set up the notation 
and also define the problem formally. In Section \ref{sec:spam_univ}, we begin with the case $\order = 1$ as warm up, and describe 
the sampling scheme, along with the algorithm for this setting. Section \ref{sec:spam_biv} considers the bivariate case $\order = 2$, 
while Section \ref{sec:spam_multiv} consists of the most general setting wherein $\order \geq 2$ is possible. Section \ref{sec:sparse_rec_results} 
contains (mostly) known results from compressed sensing for estimating sparse multilinear functions from random samples. Section \ref{sec:put_togeth_final} 
then puts together the content from the earlier sections, wherein we derive our final theorems. Section \ref{sec:disc_conc_rems} consists of a 
comparison of our results with closely related work, along with some directions for future work. 

\section{Notation and problem setup} \label{sec:problem}
\paragraph{Notation.} Scalars will be usually denoted by plain letters (e.g. $\dimn$), 
vectors by lowercase boldface letters (e.g., ${\vecx}$), matrices by uppercase boldface
letters (e.g., ${\matA}$) and sets by uppercase calligraphic letters (e.g., $\calS$),
with the exception of $[n]$, which denotes the index set $\set{1, \ldots, n}$ for any natural number $n\in\matN$.
For a (column) vector $\vecx = (x_1\dots x_d)^{T}$ and an \jvyb{ ordered $r$-tuple $\vecj = (j_1, j_2, \dots, j_r) \in {[d] \choose r}$
with $1\le j_1<j_2<\dots<j_r\le d$,} 
we denote $\vecx_{\vecj} = (x_{j_1} \dots x_{{j_r}})^{T} \in \matR^{r}$ to be the restriction of $\vecx$ on $\vecj$. 
For any finite set $\calA$, $\abs{\calA}$ denotes the cardinality of $\calA$. Moreover, if $\calA \subseteq [d]$, 
then $\setproj_{\calA}(\vecx)$ denotes the projection of $\vecx$ on $\calA$ where
\begin{equation} \label{eq:notation_proj}
(\setproj_{\calA}(\vecx))_i = \left\{
\begin{array}{rl}
x_i \ ; & i \in \calA, \\
0 \ ; & i \notin \calA,
\end{array} \right. \quad i \in [d].
\end{equation} 
The $\ell_p$ norm of a vector $\vecx \in \matR^{\dimn}$ is defined as $\|\vecx\|_p
:= \left ( \sum_{i=1}^\dimn \abs{x_i}^p \right )^{1/p}$.
A random variable $\sampsgn$ is called \randvar variable if $\sampsgn=+1$ with 
probability $1/2$ and $\sampsgn=-1$ with probability $1/2.$
A vector $\sampsgnvec\in\{-1,+1\}^d$ of independent \randvar variables is called \randvar vector. 
Similarly, a matrix $\matB\in\{-1,+1\}^{n\times d}$
is called \randvar matrix, if all its entries are independent \randvar variables.
\jvyb{For non-negative functions $f,g$ we write $f(x) = \Theta(g(x))$ if there exist constants $c_1,c_2,x_0 > 0$ 
such that $c_1 g(x) \leq f(x) \leq c_2 g(x)$ for all $x \geq x_0$. Similarly, if there exist constants $c,x_0 > 0$ such 
that} 
\begin{itemize}
\item \jvyb{$f(x) \leq cg(x)$ for all $x \geq x_0$, then we write $f(x) = O(g(x))$;}
\item \jvyb{$f(x) \geq c g(x)$ for all $x \geq x_0$, then we write $f(x) = \Omega(g(x))$.}
\end{itemize}

\paragraph{Sparse Additive Models.} For an unknown $f:\matR^d \rightarrow \matR$, our aim will be to approximate $f$ uniformly from point queries within a compact 
domain $\calG \subset \matR^d$. From now on, we will assume $\calG = [-1,1]^d$. 
The sets $\totsupp_r \subset {[d] \choose r}; r=1,\dots,\order$, 
will represent the interactions amongst the coordinates, with $\totsupp_r$ consisting of 
$r$-wise interactions. Our interest will be in the setting where each $\totsupp_r$ is sparse, i.e., 
$\abs{\totsupp_r} \ll d^r$. Given this setting, we assume to have the following structure

\begin{equation} \label{eq:gspam_form}
f = \sum_{\vecj \in \totsupp_1} \phi_{\vecj} + \sum_{\vecj \in \totsupp_2} \phi_{\vecj} + \dots 
+ \sum_{\vecj \in \totsupp_{\order}} \phi_{\vecj}.
\end{equation}

It is important to note here that the components in $\totsupp_r$ will be assumed to be truly $r$-variate, 
in the sense that they cannot be written as the sum of lower dimensional functions. For example, we assume that the components in 
$\bivsupp$ cannot be expressed as the sum of univariate functions. 


\paragraph{Model Uniqueness and ANOVA-decompositions.} We note now that the representation of $f$ in \eqref{eq:gspam_form} is 
not necessarily unique and some additional assumptions are needed to ensure uniqueness. For instance, one could add 
constants to each $\phi$ that sum up to zero, thereby giving the same $f$. Moreover, if $\bivsupp$ contains overlapping pairs 
of variables, then for each such variable -- call it $p$ -- one could add/subtract functions of the variable $x_p$ to each corresponding 
$\phi_{\vecj}$ such that $f$ remains unaltered. To obtain unique representation of $f$,
we will work with the so-called Anchored ANOVA-decomposition of $f$. We recall its notation and results in the form needed later
and refer to \cite{holtz2010sparse} for more details.

The usual notation of an ANOVA-decomposition works with functions indexed by subsets of $[d]$,
instead of tuples from $[d]$. As there is an obvious one-to-one correspondence between
\jvyb{ordered} $r$-tuples and subsets of $[d]$ with $r$ elements, we prefer to give the ANOVA-decomposition
in its usual form.

Let $\mu_j,j=1,\dots,d$ be measures defined on all Borel subsets of $[-1,1]$
and let ${U}\subseteq [d]$. We let $d\mu_U(\vecx_U)=\prod_{j\in U} d\mu_j(x_j)$
be the product measure. We define
$$
P_{U}f(\vecx_{U})=\int_{[-1,1]^{d-|{U}|}} f(\vecx)d\mu_{[d]\setminus{U}}(\vecx_{[d]\setminus U}).
$$
The ANOVA-decomposition of $f$ is then given as
\begin{align*}
f(\vecx)&=f_\emptyset+\sum_{i=1}^d f_i(x_i)+\sum_{i=1}^{d-1}\sum_{j=i+1}^d f_{i,j}(x_i,x_j)+\dots+f_{1,\dots,d}(x_1,\dots,x_d)=\sum_{U\subseteq[d]}f_{U}(\vecx_{U}),
\end{align*}
where
\begin{equation} \label{eq:anova_comp_exp}
f_{U}(\vecx_{U})=\sum_{V\subseteq U}(-1)^{|U|-|V|}P_{V}f(\vecx_{V}).
\end{equation}
In the case of $d\mu_j(x_j)=\delta(x_j)dx_j$, where $\delta$ is the Dirac distribution,
we obtain the Anchored-ANOVA decomposition
\begin{align*}
f(\vecx)&=\sum_{U\subseteq[d]}f_{U}(\vecx_{U}),
\end{align*}
where $f_{\emptyset}=f(0)$ and $f_{U}(\vecx_{U})=0$ if $x_j=0$ for some $j\in{U}.$

The standard theory of ANOVA decompositions is usually based on Hilbert space theory.
As we prefer to work with continuous functions, we give
the following representation theorem. The proof can be found in the Appendix.
%
\begin{proposition}\label{lem:anova} Let $f\in C([-1,1]^d)$. 
\jvyb{Then the collection of $(f_U)_{U\subseteq[d]}$ defined in \eqref{eq:anova_comp_exp} is the unique 
system such that the following holds.}
\begin{enumerate}
\item[a)] $f_U\in C([-1,1]^{|U|})$;
\item[b)] $f$ can be represented as
\begin{equation}\label{eq:anova1}
f(\vecx)=\sum_{U\subseteq[d]}f_U(\vecx_U),\quad \vecx\in[-1,1]^d,
\end{equation}
where $\vecx_U\in[-1,1]^{|U|}$ is the restriction of $\vecx$ onto indices included in $U$;
\item[c)] $f_{U}(\vecx_{U})=0$ if $x_j=0$ for some $j\in{U}$.
\end{enumerate}
\end{proposition}
The Anchored ANOVA-decomposition \eqref{eq:anova1} can be used to ensure uniqueness of representation
of $f$ of the form \eqref{eq:gspam_form}. For the clarity of presentation, we will later distinguish between
three settings. The first one is univariate with $r_0=1$, the second one with $r_0=2$ allows also for bivariate
interactions between the variables. Finally, in the multivariate case $r_0>2$, arbitrary
higher-order interactions can occur. We will present a detailed proposition about the corresponding ANOVA-decomposition
in each of the sections separately.

\paragraph{Assumptions.} We will specify the assumptions in each of the settings discussed later in more detail.
But, in general, we will work with two groups of conditions.
\begin{enumerate}
\item \textit{Smoothness.} We will assume throughout the paper that
the components of the ANOVA-decomposition are H\"older smooth
with exponent $\alpha\in (0,1]$ and constant $L>0$, i.e.,
$$
|\phi(\vecx)-\phi(\vecy)|\le L\norm{\vecx-\vecy}_2^\alpha
$$
for all admissible $\vecx,\vecy$.
\item \textit{Identifiability.} Furthermore, our aim is the identification of the possible interactions between the variables.
We are therefore not only interested in the approximation of $f$ but also on the identification
of the sets $\univsupp,\bivsupp,\dots,\totsupp_{r_0}$. Naturally, this is only possible if
the non-zero functions in the Anchored-ANOVA decomposition are significantly large at some point.
We will therefore assume that
$$
\|\phi\|_\infty=\sup_{\vecx}|\phi(\vecx)|>D
$$
for some $D>0$.
\end{enumerate}

\paragraph{Problem parameters and goal.} 
Based on the above setup, we will consider our problem specific parameters 
to be
\begin{itemize}
\item[(a)] smoothness parameters: $L > 0$, $\alpha \in (0,1]$,
\item[(b)] identifiability parameters: $D$,
\item[(c)] intrinsic/extrinsic dimensions: $d, r_0, \abs{\univsupp},\dots, \abs{\totsupp_{r_0}}$.
\end{itemize}
These parameters will be assumed to be known 
by the algorithm. The goal of the algorithm will then be to query $f$ within $[-1,1]^d$, and to identify the sets
$\univsupp$, \dots, $\totsupp_{r_0}$ exactly. Using standard methods of approximation theory
and sampling along canonical subspaces, one may recover also the components in \eqref{eq:gspam_form}.
We give some more details on this issue in Section \ref{sec:disc_conc_rems}.
\section{The univariate case} \label{sec:spam_univ}
As a warm up, we begin with the relatively simple setting where $r_0=1$, 
meaning that $f$ is a sum of only univariate components. It means that
$f$ admits the representation
\begin{equation} \label{eq:unique_anova_univ_spam}
f = \modmean + \sum_{p \in \univsupp} \phi_{p}(x_{p}).
\end{equation}
To ensure the uniqueness of this decomposition, we set $\mu = f(0)$ and assume that
$\phi_p(0)=0$ for all $p \in \univsupp$.

\paragraph{Assumptions.} We will make the following assumptions on the model \eqref{eq:unique_anova_univ_spam}.
\begin{enumerate}
\item \textit{Smoothness.} The terms in \eqref{eq:unique_anova_univ_spam} are H\"older continuous with parameters
$L > 0, \alpha \in (0,1]$, i.e.,
\begin{align*}
\abs{\phi_{p}(x) - \phi_{p}(y)} &\leq L |x - y|^{\alpha}\quad \text{for all}\quad p\in\univsupp\quad\text{and all}\quad x,y\in[-1,1].
\end{align*} 
%
\item \textit{Identifiability.} For every $p\in\univsupp$ there is an $x_p^*\in[-1,1]$, such that $|\phi_p(x_p^*)|>D_1$.
\end{enumerate}


%
\paragraph{Sampling scheme.} Our sampling scheme is motivated by 
the following simple observation. For any fixed $\vecx \in [-1,1]^d$, and 
some $\sampsgnvec \in \set{-1,+1}^d$, consider the points $\xposvec, \xnegvec \in [-1,1]^d$ 
defined as

\begin{equation} \label{eq:univsamp_query_pts}
\xpos_i = \left\{
\begin{array}{rl}
x_i \ ; & \sampsgn_i = +1, \\
0 \ ; & \sampsgn_i = -1
\end{array} \right. \quad \text{and}\quad
\xneg_i = \left\{
\begin{array}{rl}
0 \ ; & \sampsgn_i = +1, \\
x_i \ ; & \sampsgn_i = -1,
\end{array} \right. \ i \in[d].
\end{equation}
%
Upon querying $f$ at $\xposvec, \xnegvec$, we obtain the noisy samples
$$\ftil(\xposvec) = f(\xposvec) + \noise^{+}, \quad \ftil(\xnegvec) = f(\xnegvec) + \noise^{-},$$  
where $\noise^{+}, \noise^{-} \in \matR$ denotes the noise. One can then easily verify that the following identity holds on account of the 
structure of $f$
\begin{align} 
\ftil(\xposvec) - \ftil(\xnegvec) 
&= \sum_{i\in\univsupp} \sampsgn_i\underbrace{\phi_i(x_i)}_{\origspars_i(x_i)} + \noise^{+} - \noise^{-} \label{eq:univspam_samp_ident}  
= \dotprod{\sampsgnvec}{\origsparsvec(\vecx)} + \noise^{+} - \noise^{-}. 
\end{align}
Note that $\origsparsvec(\vecx) = (\origspars_1(x_1) \dots \origspars_d(x_d))^T$ is $\abs{\univsupp}$ sparse, 
and $\ftil(\xposvec) - \ftil(\xnegvec)$ corresponds to a noisy linear 
measurement of $\origsparsvec(\vecx)$, with $\sampsgnvec$. From standard compressive sensing results, 
we know that a sparse vector can be recovered \emph{stably}, from only a few noisy linear measurements with random vectors, drawn from a suitable distribution. In particular, it is well established that random \randvar measurements satisfy this criteria. We discuss this separately later on, for now it suffices to assume that we have at hand an appropriate sparse 
recovery algorithm: $\csalg$. 

We thus generate independent \randvar vectors $\sampsgnvec_1, \sampsgnvec_2, \dots, \sampsgnvec_n\in\{-1,+1\}^d$. 
For each $\sampsgnvec_i$, we create $\xposvec_i,\xnegvec_i$ as described in \eqref{eq:univsamp_query_pts} (for some fixed $\vecx$), 
and obtain $\ftil(\xposvec_i), \ftil(\xnegvec_i)$. Then, \eqref{eq:univspam_samp_ident} gives us the linear system
\begin{equation} \label{eq:lin_sys_univspam}
\underbrace{\begin{pmatrix}
  \ftil(\xposvec_1) - \ftil(\xnegvec_1)  \\ \vdots \\ \vdots \\ \ftil(\xposvec_n) - \ftil(\xnegvec_n)
 \end{pmatrix}}_{\vecy}
= 
\underbrace{\begin{pmatrix}
  \sampsgnvec_1^T  \\ \vdots \\ \vdots \\ \sampsgnvec_n^T
 \end{pmatrix}}_{\matB} \origsparsvec(\vecx) 
+ \underbrace{\begin{pmatrix}
  \noise_1^{+} - \noise_1^{-}  \\ \vdots \\ \vdots \\ \noise_n^{+} - \noise_n^{-}
 \end{pmatrix}}_{\noisevec}.
\end{equation}
$\csalg$ will take as input $\vecy, \matB$, and will output an estimate $\est{\origsparsvec}(\vecx)$ to 
$\origsparsvec(\vecx)$. 
\jvyb{As will be shown formally in Section \ref{sec:sparse_rec_results}, one can choose 
$\csalg$ as a $\ell_1$ minimization (convex) program for which it is well known from the  
compressive sensing literature that if $n$ is sufficiently large, then 
we will have for some $\epsilon \geq 0$ depending on $\norm{\noisevec}_{\infty}$ that 
$\norm{\est{\origsparsvec}(\vecx) - \origsparsvec(\vecx)}_{\infty} \leq \epsilon$ holds.} 
In such a case, we will refer to $\csalg$ as being ``$\epsilon$-accurate'' at $\vecx$. 
\jvyb{Also, we remark that the choice for $\csalg$ that we consider in Section \ref{sec:sparse_rec_results} 
will need an upper bound estimate of the noise level $\noisevec$ (in a suitable norm).}

Given the above, we now describe how to choose $\vecx \in [-1,1]^d$. To this end, we adopt the 
approach of \cite{Tyagi14_nips}, where the following grid on the diagonal of $[-1,1]^d$ was considered
\begin{align} \label{eq:univ_spam_base_pts}
\baseset := \bigg\{\vecx = (x \ x \ \cdots \ x)^T \in \matR^{\dimn}: x &\in \Bigl\{-1,-\frac{\numbasepts-1}{\numbasepts},\dots,\frac{\numbasepts-1}{\numbasepts},1\Bigr\}\bigg\}.
\end{align}
Our aim will be to obtain the estimate $\est{\origsparsvec}(\vecx)$ at each $\vecx \in \baseset$. 
Note that this gives us estimates to $\phi_p(x_p)$ for $p = 1,\dots,d$, with $x_p$ 
lying on a uniform one dimensional grid in $[-1,1]$. Thus we can see, at least intuitively, that 
provided $\epsilon$ is small enough, and the grid is fine enough (so that we are close to $\phi_p(x_p^{*})$ 
for each $p \in \univsupp$), we will be able to detect each $p \in \univsupp$ by thresholding. 

\paragraph{Algorithm outline and guarantees}
The discussion above is outlined formally in the form of Algorithm \ref{algo:univ_spam}. 
Lemma \ref{lemma:univ_spam_algo} below provides formal guarantees for exact recovery of 
support $\univsupp$.
\begin{algorithm*}[!ht]
\caption{Algorithm for estimating $\univsupp$} \label{algo:univ_spam} 
\begin{algorithmic}[1] 
\State \textbf{Input:} $d$, $\abs{\univsupp}$, $m$, $n$, $\epsilon$.
\State \textbf{Initialization:} $\est{\univsupp} = \emptyset$.
\State \textbf{Output:} $ \est{\univsupp}$. \\
\hrulefill

\State Construct $\baseset$ as defined in \eqref{eq:univ_spam_base_pts} with $\abs{\baseset} = 2m+1$. \label{algounivspam:base_constr}

\State Generate \randvar vectors $\sampsgnvec_1, \sampsgnvec_2, \dots, \sampsgnvec_n\in\{-1,+1\}^d$. 

\State Form $\matB \in \matR^{n \times d}$ as in \eqref{eq:lin_sys_univspam}. \label{algounivspam:gen_betas}

\For{$\vecx \in \baseset$}

	\State Generate $\xposvec_i, \xnegvec_i \in [-1,1]^d$, as in \eqref{eq:univsamp_query_pts}, using $\vecx, \sampsgnvec_i$ for each $i\in[n]$. 
	\label{algounivspam:gen_xposneg}
	
	\State Using the samples $(\ftil(\xposvec_i), \ftil(\xnegvec_i))_{i=1}^{n}$, form $\vecy$ as in \eqref{eq:lin_sys_univspam}. \label{algounivspam:query_f}

	\State Obtain $\est{\origsparsvec}(\vecx) = \csalg(\vecy, \matB)$. \label{algounivspam:sparse_rec}
			
	\State Update $\est{\univsupp} = \est{\univsupp} \cup \bigl\{p \in [d]: \abs{(\est{\origsparsvec}(\vecx))_p} > \epsilon\bigr\}$. \label{algounivspam:update_supp}

\EndFor

\end{algorithmic}
\end{algorithm*}

\begin{lemma} \label{lemma:univ_spam_algo}
Let $\csalg$ be $\epsilon$-accurate for each $\vecx \in \baseset$ with $\epsilon < \idenconst_1/3$,
which uses $n$ linear measurements. Then for $m \geq (3L/\idenconst_1)^{1/\alpha}$, Algorithm \ref{algo:univ_spam} 
recovers $\univsupp$ exactly, i.e., $\est{\univsupp} = \univsupp$. Moreover, the total number of queries 
of $f$ is $2(2m+1)n$.
\end{lemma}
\begin{proof}
Recall that we denote $\origsparsvec(\vecx) = (\phi_1(x_1)\dots\phi_d(x_d))^T$, 
and $\origspars_i(x_i) = \phi_i(x_i)$. 
For any given $p \in \univsupp$, we know that there exists $x^{*}_p \in [-1,1]$ such that 
$\abs{\phi_p(x^{*}_p)} > \idenconst_1$. Also, on account of the construction of $\baseset$, 
there exists $\vecx = (x\dots x)^{T} \in \baseset$ such that $\abs{x-x^{*}_p} \leq 1/m$. Then starting 
with the fact that $\csalg$ is $\epsilon$ accurate at $\vecx$, we obtain
\begin{align*}
\abs{\est{\origspars_p}(x)} 
&\geq \abs{\phi_p(x)} - \epsilon \geq \abs{\phi_p(x^{*}_p)}
      - \abs{\phi_p(x^{*}_p) - \phi_p(x)} - \epsilon\\
&\geq \idenconst_1 - \frac{L}{m^{\alpha}} - \epsilon
\geq \frac{2\idenconst_1}{3} - \epsilon.
\end{align*}
We used the reverse triangle inequality and the identifiability and smoothness assumptions 
on $\phi_p$. On the other hand, since $\csalg$ is $\epsilon$ accurate at each point in $\baseset$, therefore 
for every $q \notin \univsupp$ and $(c \ c \ \dots c) \in \baseset$, 
we know that $\abs{\est{\origspars_q}(c)} \leq \epsilon$.  
It then follows readily for the stated choice of $m$, $\epsilon$ that $\est{\univsupp}$ contains each variable 
in $\univsupp$, and none from $\univsupp^{c}$. 
\end{proof}

\section{The bivariate case} \label{sec:spam_biv}
Next, we consider the scenario where $r_0 = 2$, i.e.,  $f$ 
can be written as a sum of univariate and bivariate functions. 
We denote by $\bivsuppvar\jvyb{={\mathcal S}_2^{(1)}}$ the set of variables which are part 
of a $2$-tuple in $\bivsupp$. Inserting this restriction into Proposition \ref{lem:anova}, we derive the following
uniqueness result (its proof is postponed to the Appendix).
%
\begin{proposition}\label{prop:mod_unique_bivar}
Let $f\in C([-1,1]^d)$ be of the form
\begin{equation} \label{eq:mod_unique_bivar}
f = \modmean + \sum_{p \in \univsupp} \phi_{p}(x_{p}) + \sum_{\vecj \in \bivsupp} \phi_{\vecj}(x_{\vecj}) + \sum_{l \in \bivsuppvar} \phi_l(x_l),
\end{equation}
where $\univsupp \cap \bivsuppvar = \emptyset$. Moreover, let 
\begin{enumerate}
\item[a)] $\modmean=f(0)$,
\item[b)] $\phi_j(0)=0$ for all $j\in{\mathcal S}_1\cup \bivsuppvar$,
\item[c)] $\phi_{\vecj}(x_\vecj)=0$ if $\vecj=(j_1,j_2)\in{\mathcal S}_2$ and $x_{j_1}=0$ or $x_{j_2}=0.$
\end{enumerate}
Then the representation \eqref{eq:mod_unique_bivar} of $f$ is unique in the sense that each component in \eqref{eq:mod_unique_bivar} is uniquely 
identifiable.
\end{proposition}
\begin{remark}
In \eqref{eq:mod_unique_bivar}, we could have ``collapsed'' the terms corresponding to 
variables $l$ in $\sum_{l \in \bivsuppvar} \phi_l(x_l)$ -- for $l$ occurring exactly once in $\bivsupp$ --   
uniquely into the corresponding component $\phi_{\vecj}(x_{\vecj})$. A similar approach was adopted in \cite{Tyagi_spamint_long16}, 
and the resulting model was shown to be uniquely identifiable. Yet here, we choose to represent $f$ in the 
form \eqref{eq:mod_unique_bivar} for convenience, and clarity of notation. 
This also leads to a less cumbersome expression, when we work with general interaction terms later.
\end{remark}
\paragraph{Assumptions.} We now make the following assumptions on the model \eqref{eq:mod_unique_bivar}.

\begin{enumerate}

\item \textit{Smoothness.} We assume each term in \eqref{eq:mod_unique_bivar} to be H\"older continuous with parameters  
$L > 0, \alpha \in (0,1]$, i.e., 
\begin{align*}
\abs{\phi_p(x) - \phi_p(y)} &\leq L \abs{x - y}^{\alpha} \quad \text{for all}\ p \in \univsupp \cup \bivsuppvar \text{ and for all } x,y \in [-1,1], \\
\abs{\phi_{\vecj}(\vecx) - \phi_{\vecj}(\vecy)} &\leq L \norm{\vecx - \vecy}_2^{\alpha} \quad \text{for all}\ \vecj \in \bivsupp \text{ and for all } \vecx,\vecy \in [-1,1]^2.
\end{align*} 

\item \textit{Identifiability of $\univsupp, \bivsupp$.}
We assume that for each $p \in \univsupp$, there exists $x_p^{*} \in [-1,1]$ so that $\abs{\phi_p(x_p^{*})} > \idenconst_1$ for some constant 
$\idenconst_1 > 0$.
Furthermore, we assume that for each $\vecj \in \bivsupp$
there exists $\vecx^{*}_{\vecj} \in [-1,1]^2$ such that $\abs{\phi_{\vecj}(\vecx^{*}_{\vecj})} > \idenconst_2$. 
\end{enumerate}

%

%
%
\begin{remark}
\jvyb{We consider the same $\alpha, L$ for all components in $\univsupp, \bivsupp$ for the ease of 
exposition only. One can also consider the parameters $\alpha_i, L_i$ for the components in $\mathcal{S}_i$. 
This also applies to the general multivariate setting in Section \ref{sec:spam_multiv}.}
\end{remark}
Before describing our sampling scheme, we need some additional notation. For any 
$\beta \in \set{-1,1}$, we denote $\sampsgnbar = (-\sampsgn)$. Moreover, $\indic_{\sampsgn}$ denotes the indicator variable of
$\sampsgn$, i.e., $\indic_{\sampsgn} = 1$ if $\sampsgn = 1$, and $\indic_{\sampsgn} = 0$ if $\beta=-1$.
Overall, our scheme proceeds in two stages. We first identify $\bivsupp$, and only then identify $\univsupp$.
\paragraph{Sampling lemma for identifying $\bivsupp$.} 
We begin by providing the motivation behind our sampling 
scheme for identifying $\bivsupp$.
Consider some fixed mapping $\hashfn: [d] \rightarrow \set{1,2}$ that partitions 
$[d]$ into $\calA_1 = \set{i \in [d]: \hashfn(i) = 1}$ and $\calA_2 = \set{i \in [d]: \hashfn(i) = 2}$. 
Then for a given \randvar vector $\sampsgnvec \in \set{-1,1}^d$ and $\vecx = (x_1\dots x_d)^T \in [-1,1]^d$, consider the 
points $\vecx_1,\vecx_2,\vecx_3,\vecx_4 \in [-1,1]^d$ defined as
\begin{equation}
\begin{aligned}
x_{1,i} = \left\{
\begin{array}{rl}
\indic_{\sampsgn_i} x_i \ ; & i \in \calA_1, \\
\indic_{\sampsgn_i} x_i \ ; & i \in \calA_2,
\end{array} \right. \quad 
x_{2,i} = \left\{
\begin{array}{rl}
\indic_{\sampsgnbar_i} x_i \ ; & i \in \calA_1, \\
\indic_{\sampsgn_i} x_i \ ; & i \in \calA_2,
\end{array} \right. \label{eq:bivsamp_query_pts_1} \\
x_{3,i} = \left\{
\begin{array}{rl}
\indic_{\sampsgn_i} x_i \ ; & i \in \calA_1, \\
\indic_{\sampsgnbar_i} x_i \ ; & i \in \calA_2,
\end{array} \right. \quad 
x_{4,i} = \left\{
\begin{array}{rl}
\indic_{\sampsgnbar_i} x_i \ ; & i \in \calA_1, \\
\indic_{\sampsgnbar_i} x_i \ ; & i \in \calA_2,
\end{array} \right. 
\end{aligned}
\qquad i\in[d].
\end{equation}
%
%
The following lemma is the key motivation behind our sampling scheme.
\begin{lemma} \label{lem:samp_motiv_biv}
Denote $\calA = \set{\vecj \in {[d] \choose 2}: \vecj \in \set{\calA_1 \times \calA_2} \cup \set{\calA_2 \times \calA_1}}$. 
Then for functions $f$ of the form \eqref{eq:mod_unique_bivar}, we have that 
\begin{align} 
 f(\vecx_1) - f(\vecx_2) - f(\vecx_3) + f(\vecx_4) 
&= \sum_{\vecj \in \bivsupp: j_1 \in \calA_1, j_2 \in \calA_2} \beta_{j_1} \beta_{j_2} \phi_{\vecj} (\vecx_{\vecj})
+ \sum_{\vecj \in \bivsupp: j_1 \in \calA_2, j_2 \in \calA_1} \beta_{j_1} \beta_{j_2} \phi_{\vecj} (\vecx_{\vecj}) \label{eq:multilin_ident_biv_1} \\
&= \sum_{\vecj \in \calA \cap \bivsupp} \beta_{j_1} \beta_{j_2} \phi_{\vecj} (\vecx_{\vecj}). \notag
\end{align}
\end{lemma}
\begin{proof}
For any $\vecj \in \bivsupp$, let us first consider the case where $j_1, j_2$ lie in different sets. For example,
let $j_1 \in \calA_1$ and $j_2 \in \calA_2$. Then the contribution of $\phi_{\vecj}$ to the left-hand side of \eqref{eq:multilin_ident_biv_1} 
turns out to be for all possible values of $\beta_{j_1},\beta_{j_2}\in\{-1,+1\}$ equal to
\begin{align*}
&\phi_{\vecj} (\vecx_{1,\vecj}) - \phi_{\vecj}(\vecx_{2,\vecj}) - \phi_{\vecj}(\vecx_{3,\vecj}) + \phi_{\vecj}(\vecx_{4,\vecj}) \\
&= \phi_{\vecj} (\indic_{\sampsgn_{j_1}} x_{j_1}, \indic_{\sampsgn_{j_2}} x_{j_2}) - \phi_{\vecj}(\indic_{\sampsgnbar_{j_1}} x_{j_1}, \indic_{\sampsgn_{j_2}} x_{j_2}) 
- \phi_{\vecj}(\indic_{\sampsgn_{j_1}} x_{j_1}, \indic_{\sampsgnbar_{j_2}} x_{j_2}) + \phi_{\vecj}(\indic_{\sampsgnbar_{j_1}} x_{j_1}, \indic_{\sampsgnbar_{j_2}} x_{j_2}) \\
&= \beta_{j_1}\beta_{j_2}(\phi_{\vecj} (x_{j_1},x_{j_2}) - \phi_{\vecj}(x_{j_1},0) - \phi_{\vecj}(0,x_{j_2}) + \phi_{\vecj}(0,0))
= \beta_{j_1}\beta_{j_2}\phi_{\vecj} (\vecx_{\vecj}).
\end{align*} 
In case $j_1 \in \calA_2$ and $j_2 \in \calA_1$, then the contribution of $\phi_{\vecj}$ turns out to be the same as above. 
Since $f$ is additive over $\vecj \in \bivsupp$, thus the total contribution of $\bivsupp$ is given by the right-hand side of \eqref{eq:multilin_ident_biv_1}.

Now, for all $\vecj \in \bivsupp$ with $j_1,j_2$ lying in the same set, the contribution of $\phi_{\vecj}$ turns out to be zero. 
Indeed, if $j_1, j_2 \in \calA_1$, the contribution of $\phi_{\vecj}$ is 
\begin{align*}
&\phi_{\vecj} (\vecx_{1,\vecj}) - \phi_{\vecj}(\vecx_{2,\vecj}) - \phi_{\vecj}(\vecx_{3,\vecj}) + \phi_{\vecj}(\vecx_{4,\vecj}) \\
&= \phi_{\vecj} (\indic_{\sampsgn_{j_1}} x_{j_1}, \indic_{\sampsgn_{j_2}} x_{j_2}) - \phi_{\vecj}(\indic_{\sampsgnbar_{j_1}} x_{j_1}, \indic_{\sampsgnbar_{j_2}} x_{j_2}) 
- \phi_{\vecj}(\indic_{\sampsgn_{j_1}} x_{j_1}, \indic_{\sampsgn_{j_2}} x_{j_2}) + \phi_{\vecj}(\indic_{\sampsgnbar_{j_1}} x_{j_1}, \indic_{\sampsgnbar_{j_2}} x_{j_2})= 0.
\end{align*}
The same is easily verified if $j_1, j_2 \in \calA_2$. 
Lastly, let us verify that the contribution of $\phi_p$ for each $p \in \univsupp \cup \bivsuppvar$ is zero. 
Indeed, when $p \in \calA_1$, we get
\begin{align*}
&\phi_{p} (x_{1,p}) - \phi_{p}(x_{2,p}) - \phi_{p}(x_{3,p}) + \phi_{p}(x_{4,p}) = \phi_{p} (\indic_{\sampsgn_{p}} x_{p}) - \phi_{p} (\indic_{\sampsgnbar_{p}} x_{p}) 
- \phi_{p} (\indic_{\sampsgn_{p}} x_{p}) + \phi_{p} (\indic_{\sampsgnbar_{p}} x_{p})= 0
\end{align*}
and the same is true also for $p \in \calA_2$. This completes the proof.
\end{proof}
Denoting $\origspars_{\vecj}(\vecx_{\vecj}) = \phi_{\vecj}(\vecx_{\vecj})$ if $\vecj \in \bivsupp$ and $0$ otherwise, 
let $\origsparsvec(\vecx) \in \matR^{{[d] \choose 2}}$ be the corresponding ($\abs{\bivsupp}$ sparse) vector.
For $\calA\subseteq {[d] \choose 2}$ we denote $\origsparsvec(\vecx;\calA) \in \matR^{{[d] \choose 2}}$ to be the \emph{projection} of $\origsparsvec(\vecx)$ onto $\calA$. 
Clearly $\origsparsvec(\vecx;\calA)$ is at most $\abs{\bivsupp}$ sparse too -- it is in fact $\abs{\bivsupp \cap \calA}$ sparse. 
For a \randvar vector $\sampsgnvec \in \set{-1,+1}^d$, 
let $\betavect \in \set{-1,+1}^{[d] \choose 2}$, where $\betat_{\vecj} = \sampsgn_{j_1}\sampsgn_{j_2}$ for each $\vecj = (j_1,j_2)$. 
Hence we see that \eqref{eq:multilin_ident_biv_1} corresponds to a linear measurement of $\origsparsvec(\vecx;\calA)$ with the 
\randvar vector $\betavect$.

\paragraph{Sampling scheme for identifying $\bivsupp$.} 
We first generate independent \randvar vectors 
$\sampsgnvec_1, \sampsgnvec_2, \allowbreak\dots, \allowbreak \sampsgnvec_n \in \set{-1,1}^d$. 
Then for some fixed $\vecx \in [-1,1]^d$ and a mapping $\hashfn : [d] \rightarrow \set{1,2}$ -- the choice of both to be 
made clear later -- we obtain the samples $\ftil(\vecx_{i,p}) = f(\vecx_{i,p}) + \noise_{i,p}$, $i\in[n]$ and $p\in\{1,2,3,4\}$. 
Here, $\vecx_{i,1}, \vecx_{i,2}, \vecx_{i,3}, \vecx_{i,4}$ are generated using $\vecx, \sampsgnvec_i, \hashfn$ as 
outlined in \eqref{eq:bivsamp_query_pts_1}.
As a direct implication of Lemma \ref{lem:samp_motiv_biv}, we obtain the linear system
\begin{equation} \label{eq:lin_sys_bivspam}
\underbrace{\begin{pmatrix}
  \ftil(\vecx_{1,1}) - \ftil(\vecx_{1,2}) - \ftil(\vecx_{1,3}) + \ftil(\vecx_{1,4})  \\ \vdots \\ \vdots 
	\\ \ftil(\vecx_{n,1}) - \ftil(\vecx_{n,2}) - \ftil(\vecx_{n,3}) + \ftil(\vecx_{n,4})
 \end{pmatrix}}_{\vecy \in \matR^n}
= 
\underbrace{\begin{pmatrix}
  {\betavect_{1}}^T  \\ \vdots \\ \vdots \\ {\betavect_{n}}^T
 \end{pmatrix}}_{\matB \in \matR^{n \times {d \choose 2}}} \origsparsvec(\vecx;\calA) 
+ \underbrace{\begin{pmatrix}
  \noise_{1,1} - \noise_{1,2} - \noise_{1,3} + \noise_{1,4}  \\ \vdots \\ \vdots \\ \noise_{n,1} - \noise_{n,2} - \noise_{n,3} + \noise_{n,4}
 \end{pmatrix}}_{\noisevec \in \matR^n}.
\end{equation}
By feeding $\vecy, \matB$ as input to $\csalg$, we then obtain the estimate $\est{\origsparsvec}(\vecx;\calA)$ to 
$\origsparsvec(\vecx;\calA)$. Assuming $\csalg$ to be $\epsilon$-accurate at $\vecx$, we will have that 
$\norm{\est{\origsparsvec}(\vecx;\calA) - \origsparsvec(\vecx;\calA)}_{\infty} \leq \epsilon$ holds.
Let us mention that $\calA$ from Lemma \ref{lem:samp_motiv_biv} is completely determined by $\hashfn$ but we avoid
denoting this explicitly for clarity of notation.

At this point, it is natural to ask, how one should choose $\vecx$ and the mapping $\hashfn$. 
To this end, we borrow the approach of \cite{Devore2011}, which involves choosing $\hashfn$ 
from a family of hash functions, and creating for each $\hashfn$ in the family a uniform grid.
To begin with, we introduce the following definition of a family of hash functions.
\begin{definition} \label{def:thash_fam}
For some $t \in \mathbb{N}$ and $j=1,2,\dots$, let $h_j : [d] \rightarrow \set{1,2,\dots,t}$.
We call the \jvyb{family of hash functions $\thashfam = (\hashfn_1,\hashfn_2,\dots)$} a $(\dimn,t)$-hash family 
if for any distinct $i_1,i_2,\dots,i_t \in [d]$, there exists $\hashfn \in \thashfam$ such that $h$ is an 
injection when restricted to $i_1,i_2,\dots,i_t$.
\end{definition}
Hash functions are commonly used in theoretical computer science and are widely used in finding juntas \cite{Mossel03}.
One can construct $\thashfam$ of size $O(t e^t \log \dimn)$ using a standard probabilistic argument. 
The reader is for instance referred to Section $5$ in \cite{Devore2011}, where for any constant $C_1 > 1$
the probabilistic construction yields $\thashfam$ of size $\abs{\thashfam} \leq (C_1 + 1)t e^t \log \dimn$ 
with probability at least $1 - \dimn^{-C_1 t}$, in time linear in the output size.

Focusing on the setting $t = 2$ now, say we have at hand a family $\twohashfam$ of size $O(\log d)$.
Then for any $(i,j) \in {[d] \choose 2}$, there exists $\hashfn \in \twohashfam$ so that 
$\hashfn(i) \neq \hashfn(j)$. 
For each $\hashfn \in \twohashfam$, let us define $\canvec_1(\hashfn), \canvec_2(\hashfn) \in \matR^{\dimn}$, where
\begin{equation*}
(\canvec_i(\hashfn))_q := \left\{
\begin{array}{rl}
1 \ ; & \hashfn(q) = i, \\
0 \ ; & \text{otherwise}
\end{array} \right .  \quad \text{for} \ i=1,2 \ \text{and} \ q\in[\dimn].
\end{equation*}
Then we create a two dimensional grid with respect to $\hashfn$
\begin{equation} \label{eq:baseset_hash}
\baseset(\hashfn) := \biggr\{\vecx \in [-1,1]^{\dimn}: \vecx =  c_1 \canvec_1(\hashfn)+c_2 \canvec_2(\hashfn); c_1,c_2 \in
\Bigl\{-1,-\frac{\numbasepts-1}{\numbasepts},\dots,\frac{\numbasepts-1}{\numbasepts},1\Bigr\}\biggr\}.
\end{equation}
Equipped with $\baseset(\hashfn)$ for each $\hashfn \in \twohashfam$, we now possess the following approximation 
property. For any $\vecj \in {[d] \choose 2}$ and any $(x^{*}_{j_1},x^{*}_{j_2}) \in [-1,1]^2$, there exists 
$\hashfn \in \twohashfam$ with $h(j_1)\not=h(j_2)$ and a corresponding $\vecx \in \baseset(\hashfn)$ so that 
$\abs{x^{*}_{j_1} - x_{j_1}}, \abs{x^{*}_{j_2} - x_{j_2}} \leq 1/m$. 

Informally speaking, our idea is the following. Assume that $\csalg$ is $\epsilon$-accurate for each 
$\hashfn \in \twohashfam$, $\vecx \in \baseset(\hashfn)$. Also, say $\numbasepts$, $\epsilon$ are sufficiently large 
and small respectively. Hence, if we estimate $\origsparsvec(\vecx;\calA)$ at each 
$\hashfn \in \twohashfam$ and $\vecx \in \baseset(\hashfn)$, 
then for every $\vecj \in \bivsupp$, we are guaranteed to have a point $\vecx$ at which the 
estimate $\abs{\est{\origspars_{\vecj}}(\vecx_{\vecj})}$ is sufficiently large. Moreover, 
for every $\vecj \not\in \bivsupp$, we would always (i.e., for each $\hashfn \in \twohashfam$ and $\vecx \in \baseset(\hashfn)$) 
have $\abs{\est{\origspars_{\vecj}}(\vecx_{\vecj})}$ sufficiently small;  
more precisely, $\abs{\est{\origspars_{\vecj}}(\vecx_{\vecj})} \leq \epsilon$ 
since $\phi_{\vecj} \equiv 0$. Consequently, we will be able to identify $\bivsupp$ by thresholding, via a 
suitable threshold.

\paragraph{Sampling scheme for identifying $\univsupp$.} 
Assuming $\bivsupp$ is identified, the model \eqref{eq:mod_unique_bivar} reduces 
to the univariate case on the reduced set $\calP := [d] \setminus \bivsuppvar$ with $\univsupp \subset \calP$. 
We can therefore apply Algorithm \ref{algo:univ_spam} on 
$\calP$ by setting the coordinates in $\calP^{c}=\est{\bivsuppvar}$ to zero. 
Indeed, we first construct for some $m \in \matN$ the following set
%
%
\begin{equation} \label{eq:samp_restr_univ_base_pts}
\baseset = \Bigl\{(c \ c \ \dots \ c)^T \in \matR^{\calP}: 
c \in \biggl\{-1,-\frac{\numbasepts-1}{\numbasepts},\dots,\frac{\numbasepts-1}{\numbasepts},1\Bigr\}\biggr\} \subset [-1,1]^{\calP}.
\end{equation}
%
Then, for any given $\sampsgnvec \in \set{-1,1}^{\calP}$, and $\vecx \in \baseset$, 
we construct $\xposvec$, $\xnegvec \in \matR^d$ using $\sampsgnvec,\vecx$ as follows
%
\begin{equation} \label{eq:univsamp_restr_query_pts}
\xpos_i = \left\{
\begin{array}{rl}
x_i \ ; & \ \sampsgn_i = +1 \ \text{and} \ i \in \calP, \\
0 \ ; &  \text{otherwise},
\end{array} \right. \quad 
\xneg_i = \left\{
\begin{array}{rl}
x_i \ ; & \ \sampsgn_i = -1 \ \text{and} \ i \in \calP, \\
0 \ ; & \text{otherwise},
\end{array} \right.  \quad i \in[d].
\end{equation}
%
Note that $\xpos_i, \xneg_i = 0$ for $i \notin \calP$. 
Then, similarly to \eqref{eq:univspam_samp_ident}, 
we have that 
\begin{align} 
\ftil(\xposvec) - \ftil(\xnegvec) 
= \sum_{i \in \calP} \sampsgn_i\underbrace{\phi_i(x_i)}_{\origspars_i(x_i)} + \noise^{+} - \noise^{-}  
= \dotprod{\sampsgnvec}{\origsparsvec_{\calP}(\vecx)} + \noise^{+} - \noise^{-}, \label{eq:bivuniv_ident} 
\end{align} 
where $\origsparsvec_{\calP}(\vecx) \in \matR^{\calP}$ is the restriction of $\origsparsvec(\vecx)$ onto $\calP$, 
and is $\abs{\univsupp}$ sparse. 
Thereafter, we proceed as in Algorithm \ref{algo:univ_spam} by forming a linear system as in \eqref{eq:lin_sys_univspam} 
(where now $\matB \in {\matR^{n \times \abs{\calP}}}$) at each $\vecx \in \baseset$, and employing an $\epsilon$-accurate  
$\csalg$ to estimate $\origsparsvec_{\calP}(\vecx)$.
\paragraph{Algorithm outline and guarantees.} Our scheme for identifying $\bivsupp$ is 
outlined formally as the first part of Algorithm \ref{algo:biv_spam}. The second part 
involves the estimation of $\univsupp$.
%
\begin{algorithm*}[!ht]
\caption{Algorithm for estimating $\bivsupp,\univsupp$} \label{algo:biv_spam} 
\begin{algorithmic}[1] 
\State \textbf{Input:} $d$, $\abs{\bivsupp}$, $m_2$, $n_2$, $\epsilon_2$. 
\hfill \textsc{// Estimation of } $\bivsupp$ 
\State \textbf{Initialization:} $\est{\bivsupp} = \emptyset$.
\State \textbf{Output:}  $\est{\bivsupp}$.\\
\hrulefill

\State Generate independent \randvar vectors $\sampsgnvec_1, \sampsgnvec_2, \dots, \sampsgnvec_{n_2}\in\{-1,+1\}^d$.  \label{algobivspam:gen_betas}

\State Form $\matB \in \matR^{n_2 \times {d \choose 2}}$ as in \eqref{eq:lin_sys_bivspam}. \label{algobivspam:gen_B}

\State Construct a $(d,2)$ hash family: $\twohashfam$.

\For{$\hashfn \in \twohashfam$}
\State Construct $\baseset(\hashfn)$ as defined in \eqref{eq:baseset_hash} with $\abs{\baseset(\hashfn)} = (2m_2+1)^2$. \label{algobivspam:base_constr}

\For{$\vecx \in \baseset(\hashfn)$}

	\State Generate $\vecx_{i,1}, \vecx_{i,2}, \vecx_{i,3}, \vecx_{i,4} \in [-1,1]^d$, 
	as in \eqref{eq:bivsamp_query_pts_1}, using $\vecx, \sampsgnvec_i$ for each $i \in [n_2]$. 
	\label{algobivspam:gen_xts}
	
	\State Using the samples $(\ftil(\vecx_{i,1}), \ftil(\vecx_{i,2}), \ftil(\vecx_{i,3}), \ftil(\vecx_{i,4}))_{i=1}^{n_2}$, 
	form $\vecy$ as in \eqref{eq:lin_sys_bivspam}. \label{algobivspam:query_f}

	\State Obtain $\est{\origsparsvec}(\vecx;\calA) = \csalg_2(\vecy, \matB)$. \label{algobivspam:sparse_rec}
			
	\State Update $\est{\bivsupp} = \est{\bivsupp} \cup \set{\vecj \in \calA: \abs{\est{\origspars_{\vecj}}(\vecx_{\vecj})} > \epsilon_2}$. \label{algobivpam:update_bivsupp}

\EndFor
\EndFor \\

\hrulefill

\State \textbf{Input:} $d$, $\abs{\univsupp}$, $\est{\bivsuppvar}$, $m_1$, $n_1$, $\epsilon_1$. \hfill \textsc{// Estimation of} $\univsupp$
\State \textbf{Initialization:} $\est{\univsupp} = \emptyset$, $\calP = [d] \setminus \est{\bivsuppvar}$.
\State \textbf{Output:} $ \est{\univsupp}$. 

\State Construct $\baseset \subset [-1,1]^{\calP}$ with $\abs{\baseset} = 2m_1+1$, as in \eqref{eq:samp_restr_univ_base_pts}.  

\State Generate independent \randvar vectors $\sampsgnvec_1, \sampsgnvec_2, \dots, \sampsgnvec_{n_1}\in\{-1,+1\}^{\calP}$. 

\State Form $\matB \in \matR^{n_1 \times \abs{\calP}}$ as in \eqref{eq:lin_sys_univspam}.  

\For{$\vecx \in \baseset$}

	\State Generate $\xposvec_i, \xnegvec_i \in [-1,1]^d$, as in \eqref{eq:univsamp_restr_query_pts}, using $\vecx, \sampsgnvec_i$ for each $i \in [n_1]$. 
	 	
	\State Using the samples $(\ftil(\xposvec_i), \ftil(\xnegvec_i))_{i=1}^{n_1}$, form $\vecy$ as in \eqref{eq:lin_sys_univspam}.  
	
	\State Obtain $\est{\origsparsvec}(\vecx) = \csalg_1(\vecy, \matB)$ where $\est{\origsparsvec}(\vecx) \in \matR^{\calP}$. 
			
	\State Update $\est{\univsupp} = \est{\univsupp} \cup \set{p \in \calP: \abs{(\est{\origsparsvec}(\vecx))_p} > \epsilon_1}$.  
\EndFor

\end{algorithmic}
\end{algorithm*}
Lemma \ref{lemma:biv_spam_algo} provides exact recovery guarantees for $\bivsupp$ and $\univsupp$ by Algorithm \ref{algo:biv_spam}.
\begin{lemma} \label{lemma:biv_spam_algo}
Let $\twohashfam$ be a $(d,2)$ hash family, and let $\csalg_2$ be $\epsilon_2$-accurate for 
each $\hashfn \in \twohashfam$, $\vecx \in \baseset(\hashfn)$ with $\epsilon_2 < \idenconst_2/3$, which uses $n_2$ linear measurements.
If $m_2 \geq \sqrt{2}\bigl(\frac{3L}{\idenconst_2}\bigr)^{1/\alpha}$, then 
Algorithm \ref{algo:biv_spam} recovers $\bivsupp$ exactly, i.e., $\est{\bivsupp} = \bivsupp$. 
Moreover, assuming $\est{\bivsupp} = \bivsupp$ holds, and $\csalg_1$ is $\epsilon_1$-accurate (using $n_1$ measurements),
then if $m_1,n_1,\epsilon_1$ satisfy the conditions of Lemma \ref{lemma:univ_spam_algo}, we have $\est{\univsupp} = \univsupp$. 
Lastly, the total number of queries of $f$ made is $4(2m_2 + 1)^2 n_2 \abs{\twohashfam} + 2(2m_1 + 1) n_1$.
\end{lemma}
\begin{proof}
For any given $\vecj \in \bivsupp$ there exists 
$\vecx^{*}_{\vecj} \in [-1,1]^2$ with $\abs{\phi_{\vecj}(\vecx^{*}_{\vecj})} \geq \idenconst_2$. 
Moreover, since $\twohashfam$ is a $(d,2)$ hash family, there exists $\hashfn \in \twohashfam$ that is an injection on $\vecj$. 
Consequently, there exists $\vecx \in \baseset(\hashfn)$ such that $\norm{\vecx_{\vecj} - \vecx^{*}_{\vecj}}_2 \leq  \frac{\sqrt{2}}{m_2}$.
This in turn implies by H\"older continuity of $\phi_{\vecj}$ that
\begin{align}
\abs{\phi_{\vecj}(\vecx_{\vecj}) - \phi_{\vecj}(\vecx^{*}_{\vecj})} \leq L \frac{2^{\alpha/2}}{m_2^{\alpha}}. \label{eq:biv_proof_1}
\end{align}
Since $\csalg_2$ is $\epsilon_2$-accurate for each $\hashfn \in \twohashfam$, $\vecx \in \baseset(\hashfn)$, we know that 
at the aforementioned $\vecx$, the following holds via reverse triangle inequality
\begin{equation}
\abs{\est{\origspars_{\vecj}}(\vecx_{\vecj})} \geq \abs{\phi_{\vecj}(\vecx_{\vecj})} - \epsilon_2.  \label{eq:biv_proof_2}
\end{equation}
Using \eqref{eq:biv_proof_1}, \eqref{eq:biv_proof_2} and the reverse triangle inequality, we get by the choice of $\epsilon_2$ and $m_2$
\begin{align*}
\abs{\est{\origspars_{\vecj}}(\vecx_{\vecj})} 
&\geq \abs{\phi_{\vecj}(\vecx^{*}_{\vecj})} - L \frac{2^{\alpha/2}}{m_2^{\alpha}} - \epsilon_2 
\geq \idenconst_2 - L \frac{2^{\alpha/2}}{m_2^{\alpha}} - \epsilon_2 
\geq \frac{2\idenconst_2}{3} - L \frac{2^{\alpha/2}}{m_2^{\alpha}}
\geq \frac{\idenconst_2}{3}.
\end{align*}
Also, for any $\vecj \notin \bivsupp$, we have for all $\hashfn \in \twohashfam$, $\vecx \in \baseset(\hashfn)$ that  
$\abs{\est{\origspars_{\vecj}}(\vecx_{\vecj})} \leq \epsilon_2 < \idenconst_2/3$ (since $\phi_{\vecj} \equiv 0$). 
Hence, the stated choice of $\epsilon_2$ guarantees identification of each $\vecj \in \bivsupp$, 
and none from ${[d] \choose 2} \setminus \bivsupp$. 
The proof for recovery of $\univsupp$ is identical to Lemma \ref{lemma:univ_spam_algo}, and hence omitted.
\end{proof}
%
\begin{remark}
On a top level, Algorithm \ref{algo:biv_spam} is similar to \cite[Algorithms 3,4]{Tyagi_spamint_long16} in the sense 
that they all involve solving $\ell_1$ minimization problems at base points lying in $\baseset(\hashfn)$ defined 
in \eqref{eq:baseset_hash} (for identification of $\bivsupp$), and $\baseset$ defined in \eqref{eq:samp_restr_univ_base_pts} 
(for identification of $\univsupp$). The difference however lies in the nature of the sampling schemes. The scheme in 
\cite[Algorithms 3,4]{Tyagi_spamint_long16} relies on estimating sparse Hessians, gradients of $f$ via 
their linear measurements, through random samples in the neighborhood of the base point. In contrast, 
the sampling scheme in Algorithm \ref{algo:biv_spam} is not local; for instance during the identification of $\bivsupp$, 
at each base point $\vecx \in \baseset(\hashfn)$, 
the points $\vecx_{i,1}, \vecx_{i,2}, \vecx_{i,3}, \vecx_{i,4}$ for any given $i \in [n_2]$ can be arbitrarily far from each other.
The same is true during the identification of $\univsupp$.
\end{remark}
%
%
%
%
\section{The multivariate case} \label{sec:spam_multiv}
Finally, we treat also the general case where $f$ consists of at most $\order$-variate components, where $r_0 > 2$ is possible.
To begin with, let $\totsupp_1, \totsupp_2,\dots,\totsupp_{\order}$ be such that $\totsupp_{r} \subset {[d] \choose r}$
for $r \in [\order]$. Here $\totsupp_r$ represents the $r$ wise interaction terms.
We now need some additional notation. 
\begin{enumerate}
\item For $r \ge 1$, let $\totsupp_r^{(1)}$ denote the set of variables occurring in $\totsupp_r$\jvyb{,}
with $\totsupp_1^{(1)}=\totsupp_1$. \jvyb{Hence, $\totsupp_r^{(1)}=\bivsuppvar$ in the bivariate case $r=2$.}

\item For each $1\le i < r \leq \order$, denote $\totsupp_r^{(i)} = {\totsupp_r^{(1)} \choose i}$ 
to be the sets of $i^{th}$ order tuples induced by $\totsupp_r$. 
\end{enumerate}

The multivariate analogue of Proposition \ref{prop:mod_unique_bivar} is provided by the following result.
\begin{proposition}\label{prop:mod_unique_multivar}
Let $1 \le \order \le d$ and let $f\in C([-1,1]^d)$ be of the form
\begin{align} \label{eq:mod_unique_multivar}
f(\vecx) = \modmean &+ \sum_{j \in \bigcup_{r=2}^{\order}\totsupp_{r}^{(1)} \cup \univsupp} \phi_{j}(x_{j}) 
+ \sum_{\vecj \in \bigcup_{r=3}^{\order}\totsupp_{r}^{(2)} \cup \bivsupp} \phi_{\vecj}(x_{\vecj}) \nonumber \\
&\qquad
+ \dots 
+ \sum_{\vecj \in \totsupp_{\order}^{(\order-1)} \cup \totsupp_{\order-1}} \phi_{\vecj}(x_{\vecj}) 
+ \sum_{\vecj \in \totsupp_{\order}} \phi_{\vecj}(x_{\vecj}),
\end{align}
where all the functions $\phi_{j}$ are not identically zero. Moreover, let 
\begin{enumerate}
\item[(a)] $\modmean=f(0)$.

\item[(b)] For each $1 \leq l \leq \order-1$, $\phi_{\vecj}(x_\vecj)=0$ if $\vecj=(j_1,\dots,j_l) \in \bigcup_{r=l+1}^{\order}\totsupp_{r}^{(l)} \cup \totsupp_l$, 
and $x_{j_i}=0$ for some $i \in [l].$

\item[(c)] $\phi_j(0)=0$ if $\vecj = (j_1,\dots,j_{\order}) \in \totsupp_{\order}$, and $x_{j_i} = 0$ for 
some $i \in [\order]$.
\end{enumerate}
Then the representation \eqref{eq:mod_unique_multivar} of $f$ is unique in the sense that 
each component in \eqref{eq:mod_unique_multivar} is uniquely identifiable.
\end{proposition}
The proof of this result is similar to the proof of Proposition \ref{prop:mod_unique_bivar}, and we leave it to the reader.
\begin{remark}
Let us note that for the special cases $\order\in\{1,2\}$, the statement of Proposition \ref{prop:mod_unique_multivar} reduces to 
that of Proposition \ref{prop:mod_unique_bivar} for univariate/bivariate SPAMs.
\end{remark}
%
%
%
%
We now generalize the sampling scheme given before for bivariate components to the setting of multivariate components.
\jvyb{For this sake, we denote by $\digit(a,b)\in\{0,1\}$ the $b^{th}$ digit of the dyadic decomposition of $a$ for $a,b\in\matN_0$.
and put $\digit(a)$ to be the sum of digits of $a\in\matN_0$, i.e.
$$
a=\sum_{i=0}^\infty \digit(a,i)\,\cdot\,2^i,\qquad \digit(a)=\sum_{i=0}^\infty \digit(a,i).
$$
Let us fix some mapping $\hashfn: [d] \rightarrow [\order]$ that partitions 
$[d]$ into $\calA_1 = \set{i \in [d]: \hashfn(i) = 1}$, $\calA_2 = \set{i \in [d]: \hashfn(i) = 2},\dots,\calA_{\order}=\set{i\in[d]:\hashfn(i)=\order}$.
Let us fix a \randvar vector $\sampsgnvec \in \set{-1,1}^d$ and $\vecx = (x_1\dots x_d)^T \in [-1,1]^d$.
For $z\in[2^{\order}]$ and $i\in[d]$, we define
\begin{equation} \label{eq:multiv_xsamp_pts}
(\vecx_z)_i=x_{z,i}=\begin{cases} x_i\quad\text{if }\beta_i=(-1)^{\digit(z-1,h(i)-1)},\\
0\quad\text{otherwise}.
\end{cases}
\end{equation}}
%
%
%
%

\begin{remark} It $r_0=1$, it is easily verified, that the points $(\vecx_z)_{z=1}^2$ in \eqref{eq:multiv_xsamp_pts}
coincide with the points $\xposvec,\xnegvec$ defined in \eqref{eq:univsamp_query_pts} for univariate SPAMs.
Similarly, for $\order = 2$, the points $(\vecx_z)_{z=1}^4$ from \eqref{eq:multiv_xsamp_pts} agree with
those defined in \eqref{eq:bivsamp_query_pts_1} for bivariate SPAMs.
In the same way, the following lemma is a generalization of \eqref{eq:univspam_samp_ident}
and Lemma \ref{lem:samp_motiv_biv} for \eqref{eq:mod_unique_multivar}.
Indeed, if $r_0=1$, there is only one mapping $h : [d] \rightarrow \{1\}$, and so $\calA = [d]$.
\end{remark}

%
%
\begin{lemma} \label{lem:samp_motiv_multiv}
Denote $\calA = \set{\vecj \in {[d] \choose \order}: \hashfn \text{ is injective on } \set{j_1,\dots,j_{\order}} }$.
Then for functions $f$ of the form \eqref{eq:mod_unique_multivar}, we have that 
\begin{align}
\sum_{z=1}^{2^{\order}} (-1)^{\digit(z-1)}f(\vecx_z)
= \sum_{\vecj \in \calA \cap \totsupp_{\order}} \beta_{j_1}\dots \beta_{j_{\order}} \phi_{\vecj} (\vecx_{\vecj}). \label{eq:multilin_ident_multiv_2}
\end{align}
\end{lemma}
%
%
\begin{proof}
We plug \eqref{eq:mod_unique_multivar} into the left-hand side of \eqref{eq:multilin_ident_multiv_2} and obtain
\begin{align*}
&\sum_{z=1}^{2^{\order}} (-1)^{\digit(z-1)}f(\vecx_z) \\
&=\sum_{z=1}^{2^{\order}} (-1)^{\digit(z-1)}
\Bigl[\modmean + \sum_{j \in \bigcup_{r=2}^{\order}\totsupp_{r}^{(1)} \cup \univsupp} \phi_{j}(x_{z,j}) + \sum_{\vecj \in \bigcup_{r=3}^{\order}\totsupp_{r}^{(2)} \cup \bivsupp} \phi_{\vecj}((\vecx_z)_{\vecj})+ \dots + \sum_{\vecj \in \totsupp_{\order}} \phi_{\vecj}((\vecx_z)_{\vecj})\Bigr]\\
&=\mu \sum_{z=1}^{2^{\order}} (-1)^{\digit(z-1)}+\sum_{j \in \bigcup_{r=2}^{\order}\totsupp_{r}^{(1)} \cup \univsupp} \sum_{z=1}^{2^{\order}} (-1)^{\digit(z-1)}\phi_{j}(x_{z,j}) + \dots
+ \sum_{\vecj \in \totsupp_{\order}} \sum_{z=1}^{2^{\order}} (-1)^{\digit(z-1)}\phi_{\vecj}((\vecx_z)_{\vecj}) \\
&=I_0+I_1+\dots+I_{\order}.
\end{align*}
We show first that $I_0=I_1=\dots=I_{{\order}-1}=0.$ Indeed,
\begin{align*}
I_0=\mu \sum_{z=1}^{2^{\order}} (-1)^{\digit(z-1)}=\mu \sum_{z=1}^{2^{{\order}-1}} \Bigl((-1)^{\digit(2z-2)}+(-1)^{\digit(2z-1)}\Bigr)
\end{align*}
and the last expression vanishes as $\digit(2z-1)=\digit(2z-2)+1$ for every $z\in[2^{{\order}-1}].$

If $j \in \bigcup_{r=2}^{\order}\totsupp_{r}^{(1)} \cup \univsupp$, we define the set $U_j=\set{z\in[2^{\order}]: \beta_j=(-1)^{\digit(z-1,h(j)-1)}}$ and write
\begin{align*}
I_1&= \sum_{j \in \bigcup_{r=2}^{\order}\totsupp_{r}^{(1)} \cup \univsupp}\sum_{z=1}^{2^{\order}} (-1)^{\digit(z-1)}\phi_{j}(x_{z,j})=
\sum_{j \in \bigcup_{r=2}^{\order}\totsupp_{r}^{(1)} \cup \univsupp} \phi_{j}(x_{j})\sum_{\substack{z\in[2^{\order}]\\ \beta_j=(-1)^{\digit(z-1,h(j)-1)}}} (-1)^{\digit(z-1)}\\
&=\sum_{j \in \bigcup_{r=2}^{\order}\totsupp_{r}^{(1)} \cup \univsupp} \phi_{j}(x_{j})\sum_{z\in U_j} (-1)^{\digit(z-1)}.
\end{align*}
\jvyb{The definition of $U_j$ fixes one digit of $z-1$ (namely the one at position $h(j)-1$).
The sums over $U_j$ contain $2^{{\order}-1}$ number of summands. Looking at their digit
on a position different from $h(j)-1$, we see that half of the summands is equal to 1 and the other half to $-1.$ Therefore also $I_1=0.$}

Similarly, if $\vecj=(j_1,j_2)\in\bigcup_{r=3}^{\order}\totsupp_{r}^{(2)} \cup \bivsupp$, we set
$$
U_{\vecj}=\set{z\in[2^{\order}]: \beta_{j_1}=(-1)^{\digit(z-1,h(j_1)-1)}\ \text{and}\ \beta_{j_2}=(-1)^{\digit(z-1,h(j_2)-1)}}
$$
and obtain
\begin{align*}
I_2=\sum_{\vecj \in \bigcup_{r=3}^{\order}\totsupp_{r}^{(2)} \cup \bivsupp} \phi_{\vecj}((\vecx_z)_{\vecj})\sum_{z\in U_{\vecj}}(-1)^{\digit(z-1)}.
\end{align*}
\jvyb{If now $h(j_1)=h(j_2)$ and $\beta_{j_1}\not=\beta_{j_2}$, then $U_{\vecj}$ is empty and the sum over $U_{\vecj}$ is zero.
If $h(j_1)=h(j_2)$ and $\beta_{j_1}=\beta_{j_2}$, then $U_{\vecj}=\{z\in[r_0]:\beta_{j_1}=(-1)^{\digit(z-1,h(j_1)-1)}\}=U_{j_1}$
contains $2^{r_0-1}$ elements and the sum over $U_{\vecj}$ is again zero by the same argument as above. Finally, if $h(j_1)\not=h(j_2)$,
the definition of $U_{\vecj}$ fixes two digits of $z-1$. Therefore, $U_{\vecj}$ has $2^{r_0-2}$ elements.
Then we consider an index $l\in\{0,1,\dots,r_0-1\}$ different from $h(j_1)$ and $h(j_2)$, and observe that $z\in U_{\vecj}$ can have dyadic digit on $l$
equal to zero or one. Hence the sum over $U_{\vecj}$ is again equal to zero and $I_2=0.$
The same argument can be applied as long as $\{h(j_1),\dots,h(j_r)\}$ is a proper subset of $[r_0]$
leading to $I_0=I_1=\dots=I_{{\order}-1}=0.$}

Finally, if $\vecj=(j_1,\dots,j_{\order})\in\totsupp_{\order}$, we define
$$
U_{\vecj}=\set{z\in[2^{\order}]: \beta_{j_i}=(-1)^{\digit(z-1,h(j_i)-1)}\ \text{for all}\ i\in[{\order}]}.
$$
If $\hashfn$ is an injection on $\{j_1,\dots,j_{\order}\}$, we get that $\{h(j_1),\dots,h(j_{\order})\}=[{\order}]$ and $U_\vecj=\set{z^\vecj}$ is a singleton with
$$
\sum_{z\in U_{\vecj}}(-1)^{\digit(z-1)}=(-1)^{\digit(z^\vecj-1)}=\prod_{i=1}^{\order}(-1)^{\digit(z^\vecj-1,i-1)}
=\prod_{i=1}^{\order}(-1)^{\digit(z^\vecj-1,h(j_i)-1)}=\prod_{i=1}^{\order} \beta_{j_i}.
$$
If, on the other hand, $\hashfn$ is no injection on $\set{j_1,\dots,j_{\order}}$, $U_\vecj$ has even number of elements and using the same argument as above we
obtain
$$
\sum_{z\in U_{\vecj}}(-1)^{\digit(z-1)}=0.
$$
We conclude that
\begin{align*}
\sum_{z=1}^{2^{\order}} (-1)^{\digit(z-1)}f(\vecx_z)&=I_0+I_1+\dots+I_{\order}
=\sum_{\vecj\in\totsupp_{\order}}\phi_{\vecj}(\vecx_{\vecj})\sum_{z\in U_\vecj} (-1)^{\digit(z-1)}\\
&=\sum_{\vecj\in\calA\cap\totsupp_{\order}} \beta_{j_1}\dots\beta_{j_{\order}}\phi_{\vecj}(\vecx_{\vecj}).\qedhere
\end{align*}
\end{proof}

We denote again $\origspars_{\vecj}(\vecx_{\vecj}) = \phi_{\vecj}(\vecx_{\vecj})$ if $\vecj \in \totsupp_{\order}$ and $0$ otherwise.
Similarly, $\origsparsvec(\vecx) \in \matR^{{[d] \choose \order}}$ stands for the corresponding $\abs{\totsupp_{\order}}$-sparse vector
and $\origsparsvec(\vecx;\calA) \in \matR^{{[d] \choose \order}}$ for the projection of $\origsparsvec(\vecx)$ onto $\calA$. 
Again, $\origsparsvec(\vecx;\calA)$ is $\abs{\totsupp_{\order} \cap \calA}$-sparse. 
Finally, for a \randvar vector $\sampsgnvec \in \set{-1,+1}^d$, 
let $\betavecord \in \set{-1,+1}^{[d] \choose \order}$ where $\betaord_{\vecj} = \sampsgn_{j_1}\sampsgn_{j_2}\dots\sampsgn_{j_{\order}}$ 
for each $\vecj = (j_1,j_2,\dots,j_{\order})$. 
Hence, \eqref{eq:multilin_ident_multiv_2} corresponds to a linear measurement of $\origsparsvec(\vecx;\calA)$ with $\betavecord$. 

\paragraph{Assumptions.} We will make the following assumptions on the model \eqref{eq:mod_unique_multivar}.
\begin{enumerate}
\item \textit{Smoothness.} Each term in \eqref{eq:mod_unique_multivar} is H\"older continuous with parameters  
$L > 0, \alpha \in (0,1]$, i.e., for each $i \in [\order]$, 
\begin{align}
\abs{\phi_{\vecj}(\vecx) - \phi_{\vecj}(\vecy)} &\leq L \norm{\vecx - \vecy}_2^{\alpha} \quad 
\text{ for all } \vecj \in \totsupp_i \cup\bigcup_{l=i+1}^{\order} \totsupp_l^{(i)}\ \text{and for all }\vecx,\vecy \in [-1,1]^{i}.
\end{align} 
%
\item \textit{Identifiability of $\totsupp_i$, $i \in [\order]$.} 
We assume that for each $i \in [\order]$ there exists a constant $\idenconst_i > 0$, such that 
for every $\vecj \in \totsupp_i$ there exists $\vecx_{\vecj}^{*} \in [-1,1]^i$ with $\abs{\phi_{\vecj}(\vecx^{*}_{\vecj})} > \idenconst_i$.
%
%
\item \jvyb{\textit{Disjointness.} We assume that $\totsupp_p^{(1)} \cap \totsupp_q^{(1)} = \emptyset$
for all $p \neq q \in [\order]$. If $\order = 2$, we observed earlier in Section \ref{sec:spam_biv} 
that the assumption $\univsupp\cap \bivsupp^{(1)}=\emptyset$ can be made without loss
of generality. However, for $\order > 2$, this is an additional assumption. It will allow to structure 
the recovery algorithm into recursive steps. For eg., if $\order = 3$, then the following configuration 
does not satisfy the disjointness assumption}
\begin{equation*}
\jvyb{\univsupp = \set{1,2,3}, \quad \bivsupp = \set{(4,5), (5,6), (6,7)},\quad  \totsupp_3 = \set{(6,8,9)}.}
\end{equation*}
\jvyb{In this case, ${\mathcal S}_2^{(1)}=\{4,5,6,7\}$ and ${\mathcal S}_3^{(1)}=\{6,8,9\}$ are not disjoint.}
\end{enumerate}

\paragraph{Sampling scheme for identifying $\totsupp_{\order}$.}
Similarly to the sampling scheme for identifying $\bivsupp$ in the bivariate case,
we generate independent \randvar vectors $\sampsgnvec_1, \sampsgnvec_2, \dots, \sampsgnvec_n \in \set{-1,1}^d$.
For fixed $\vecx \in [-1,1]^d$ and $\hashfn : [d] \rightarrow [\order]$,
we obtain the samples $\ftil(\vecx_{i,z}) = f(\vecx_{i,z}) + \noise_{i,z}$, where $i\in[n]$ and $z \in[2^{\order}]$.
Here, $\vecx_{i,z}$ are generated using $\vecx_i, \sampsgnvec_i, \hashfn$ as outlined in \eqref{eq:multiv_xsamp_pts}.

As a direct implication of Lemma \ref{lem:samp_motiv_multiv}, we obtain the linear system
\begin{equation} \label{eq:lin_sys_multiv_spam}
\underbrace{\begin{pmatrix}
  \sum_{z=1}^{2^{\order}} (-1)^{\digit(z-1)} \ftil(\vecx_{1,z})  \\ \vdots \\ \vdots \\ \sum_{z=1}^{2^{\order}} (-1)^{\digit(z-1)} \ftil(\vecx_{n,z})
 \end{pmatrix}}_{\vecy \in \matR^n}
= 
\underbrace{\begin{pmatrix}
  {\betavecord_{1}}^T  \\ \vdots \\ \vdots \\ {\betavecord_{n}}^T
 \end{pmatrix}}_{\matB \in \matR^{n \times {d \choose \order}}} \origsparsvec(\vecx;\calA) 
+ \underbrace{\begin{pmatrix}
  \sum_{z=1}^{2^{\order}} (-1)^{\digit(z-1)} \noise_{1,z} \\ \vdots \\ \vdots \\ \sum_{z=1}^{2^{\order}} (-1)^{\digit(z-1)} \noise_{n,z}
 \end{pmatrix}}_{\noisevec \in \matR^n}.
\end{equation}

By feeding $\vecy, \matB$ as input to $\csalg$, we will obtain the estimate $\est{\origsparsvec}(\vecx;\calA)$ to 
$\origsparsvec(\vecx;\calA)$. Assuming $\csalg$ to be $\epsilon$-accurate at $\vecx$, we will have that 
$\norm{\est{\origsparsvec}(\vecx;\calA) - \origsparsvec(\vecx;\calA)}_{\infty} \leq \epsilon$ holds. 

The choice of $\vecx,\hashfn$ is along similar lines as in the previous section. Indeed 
we first construct a $(d,\order)$ hash family $\hashfam_{\order}^d$ so that
for any $\vecj = (j_1,\dots,j_{\order}) \in {[d] \choose \order}$, there exists $\hashfn \in \hashfam_{\order}^d$ which is injective on  
$[\vecj]$. For each $\hashfn \in \hashfam_{\order}^d$, let us define 
$\canvec_1(\hashfn), \canvec_2(\hashfn), \dots, \canvec_{{\order}}(\hashfn) \in \matR^{\dimn}$, where
\begin{equation*}
(\canvec_i(\hashfn))_q := \left\{
\begin{array}{rl}
1 \ ; & \hashfn(q) = i, \\
0 \ ; & \text{otherwise}
\end{array} \right .  \quad \text{for} \ i\in[\order] \ \text{and} \ q\in[\dimn].
\end{equation*}
We then create the following $\order$ dimensional grid with respect to $\hashfn$.
\begin{equation} \label{eq:baseset_hash_gen}
\baseset(\hashfn) := \biggl\{\vecx \in [-1,1]^{\dimn}: \vecx = \sum_{i=1}^{\order} c_i \canvec_i(\hashfn); c_1,c_2,\dots,c_{\order} \in
\Bigl\{-1,-\frac{\numbasepts-1}{\numbasepts},\dots,\frac{\numbasepts-1}{\numbasepts},1\Bigr\}\biggr\}.
\end{equation}
Equipped with $\baseset(\hashfn)$ for each $\hashfn \in \hashfam_{\order}^d$, we now possess the following approximation 
property. For any $\vecj \in {[d] \choose \order}$ and any $(x^{*}_{j_1},x^{*}_{j_2},\dots,x^{*}_{j_{\order}}) \in [-1,1]^{\order}$, 
there exists $\hashfn \in \hashfam_{\order}^d$ and a corresponding $\vecx \in \baseset(\hashfn)$ so that 
$\abs{x^{*}_{j_1} - x_{j_1}}, \abs{x^{*}_{j_2} - x_{j_2}}, \dots, \abs{x^{*}_{j_{\order}} - x_{j_{\order}}} \leq 1/m$. 

Here on, our idea for estimating $\totsupp_{\order}$ is based on the same principle that we followed in the preceding section. 
Assume that $\csalg$ is $\epsilon$ accurate for each $\hashfn \in \hashfam_{\order}^d$, $\vecx \in \baseset(\hashfn)$, 
and that $\numbasepts$, $\epsilon$ are sufficiently large 
and small respectively. Hence, if we estimate $\origsparsvec(\vecx;\calA)$ at each 
$\hashfn \in \hashfam_{\order}^d$ and $\vecx \in \baseset(\hashfn)$, 
then for every $\vecj \in \totsupp_{\order}$ we are guaranteed to have a point $\vecx$ at which the 
estimate $\abs{\est{\origspars_{\vecj}}(\vecx_{\vecj})}$ is sufficiently large. Moreover, 
for every $\vecj \not\in \totsupp_{\order}$, we would always (i.e., for each $\hashfn \in \hashfam_{\order}^d$, $\vecx \in \baseset(\hashfn)$) 
have that $\abs{\est{\origspars_{\vecj}}(\vecx_{\vecj})}$ is sufficiently small;  
more precisely, $\abs{\est{\origspars_{\vecj}}(\vecx_{\vecj})} \leq \epsilon$ 
since $\phi_{\vecj} \equiv 0$. Consequently, we will be able to identify $\totsupp_{\order}$ by thresholding, via a 
suitable threshold.

\paragraph{Sampling scheme for identifying $\totsupp_{\order-1}$.}
Say we have an estimate for $\totsupp_{\order}$, lets call it $\est{\totsupp_{\order}}$, and assume 
$\est{\totsupp_{\order}}$ was identified correctly, so $\est{\totsupp_{\order}} = \totsupp_{\order}$. 
Then, we now have a SPAM of order $\order-1$ on the reduced set of variables $\calP = [d]\setminus\est{\totsupp_{\order}^{(1)}}$.
Therefore, in order to estimate $\totsupp_{\order-1}$, we simply repeat the above procedure on the 
reduced set $\calP$ by freezing the variables in $\est{\totsupp_{\order}^{(1)}}$ to $0$. 
More precisely, we have the following steps.
\begin{itemize}
\item We will construct a $(\calP, \order-1)$ hash family $\hashfam_{\order-1}^{\calP}$, hence each 
$\hashfn \in \hashfam_{\order-1}^{\calP}$ is a mapping $\hashfn:\calP \rightarrow [\order-1]$.

\item For each $\hashfn \in \hashfam_{\order-1}^{\calP}$, define
$\canvec_1(\hashfn), \canvec_2(\hashfn), \dots, \canvec_{{\order}-1}(\hashfn) \in \matR^{\calP}$, where
\begin{equation*}
(\canvec_i(\hashfn))_q := \left\{
\begin{array}{rl}
1 \ ; & \hashfn(q) = i \ \text{and} \ q \in \calP, \\
0 \ ; & \text{otherwise},
\end{array} \right .  \quad \text{for} \ i\in[\order-1] \ \text{and} \ q \in \calP,
\end{equation*}
and use $(\canvec_i(\hashfn))_{i=1}^{\order-1}$ to create a $\order-1$ dimensional grid $\baseset(\hashfn) \subset [-1,1]^{\calP}$ 
in the same manner as in \eqref{eq:baseset_hash_gen}. 

\item For $\hashfn \in \hashfam_{\order-1}^{\calP}$, a \randvar vector $\sampsgnvec \in \set{-1,1}^{\calP}$ and 
$\vecx \in [-1,1]^{\calP}$, we define $\vecx_z \in \matR^d$ in \eqref{eq:multiv_xsamp_pts} as follows
\begin{equation} \label{eq:multiv_xsamp_pts_1}
(\vecx_z)_i=x_{z,i}=
\begin{cases} 
x_i\ ; \quad \text{if }\beta_i=(-1)^{\digit(z-1,h(i)-1)} \text{ and} \ i \in \calP,\\
0\ ; \quad \text{otherwise},
\end{cases} \ \text{for}\  i\in[d]\quad\text{and}\quad  z \in[2^{\order-1}].
\end{equation}
 Hence denoting 
$\calA = \set{\vecj \in {\calP \choose \order-1}: \hashfn \text{ is injective on } \set{j_1,\dots,j_{\order-1}} }$, 
since $\est{\totsupp_{\order}} = \totsupp_{\order}$, we obtain as a result of Lemma \ref{lem:samp_motiv_multiv} that 
\begin{align}
\sum_{z=1}^{2^{\order-1}} (-1)^{\digit(z-1)}f(\vecx_z)
= \sum_{\vecj \in \calA \cap \totsupp_{\order-1}} \beta_{j_1}\dots \beta_{j_{\order-1}} \phi_{\vecj} (\vecx_{\vecj}). \label{eq:multilin_ident_multiv_3}
\end{align}
Consequently, in the linear system in \eqref{eq:lin_sys_multiv_spam}, we have 
$\matB \in \matR^{n \times {\abs{\calP} \choose \order-1}}$ where the $i^{th}$ row of $\matB$ 
is $\sampsgnvec^{(\order-1)} \in \set{-1,+1}^{\calP \choose \order-1}$ with 
$\sampsgn^{(\order-1)}_{\vecj} = \sampsgn_{j_1}\sampsgn_{j_2}\dots\sampsgn_{j_{\order-1}}$ 
for each $\vecj = (j_1,j_2,\dots,j_{\order-1})$. Note that $\origsparsvec(\vecx;\calA) \in \matR^{{\calP \choose \order-1}}$ is the  
$\abs{\totsupp_{\order-1}}$ sparse vector to be estimated.

\item Finally, we will estimate $\origsparsvec(\vecx;\calA)$ at each 
$\hashfn \in \hashfam_{\order-1}^{\calP}$ and $\vecx \in \baseset(\hashfn)$.  
If $\csalg$ is $\epsilon$ accurate, with $\epsilon$ sufficiently small, then by choosing 
the number of points $m$ to be sufficiently large, we will be able to identify $\totsupp_{\order-1}$ 
via thresholding.
\end{itemize} 
%

By repeating the above steps for all $i=\order, \order-1,\dots,1$, we arrive at a procedure for estimating the supports 
$\totsupp_i, i \in [\order]$; this is outlined formally in the form of the Algorithm \ref{algo:multiv_spam} below.
%
%
\begin{algorithm*}[!ht]
\caption{Algorithm for estimating $\totsupp_1,\totsupp_2,\dots,\totsupp_{\order}$} \label{algo:multiv_spam} 
\begin{algorithmic}[1] 
\State \textbf{Input:} $d$, $\abs{\totsupp_i}$, $(m_i,n_i,\epsilon_i)$ for $i = 1,\dots,\order$. 
\State \textbf{Initialization:} $\est{\totsupp_i} = \emptyset$ for $i = 1,\dots,\order$. $\calP = [d]$.
\State \textbf{Output:} $\est{\totsupp_i}$ for $i = 1,\dots,\order$.\\
\hrulefill

\For{$i = \order,\order-1,\dots,1$}\qquad \textsc{// Estimation of } $\totsupp_i$ 

\State Generate \randvar random vectors $\sampsgnvec_1, \sampsgnvec_2, \dots, \sampsgnvec_{n_i} \in \set{-1,1}^{\calP}$.  \label{algomultivspam:gen_betas}

\State Form $\matB \in \matR^{n_i \times \jvyb{{\abs{\calP} \choose i}}}$ as in \eqref{eq:lin_sys_multiv_spam}. \label{algomultivspam:gen_B}

\State Construct a $(\calP,i)$ hash family $\hashfam_{i}^{\calP}$.

\For{$\hashfn \in \hashfam_{i}^{\calP}$}

\State Construct $\baseset(\hashfn) \subset [-1,1]^{\calP}$ in the same manner as in \eqref{eq:baseset_hash_gen} 
       with $\abs{\baseset(\hashfn)} = (2m_i+1)^{i}$. \label{algomultivspam:base_constr}

\For{$\vecx \in \baseset(\hashfn)$}

	\State Generate $\vecx_{z} \in [-1,1]^d$, with $z \in[2^{i}]$ 
	as in \eqref{eq:multiv_xsamp_pts_1}, using $\vecx, \sampsgnvec_u$ for each $u \in [n_i]$. 
	\label{algomultivspam:gen_xts}
	
	\State Using the samples $(\ftil(\vecx_{z}))_{z=1}^{2^{i}}$, form $\vecy \in \matR^{n_i}$ as in \eqref{eq:lin_sys_multiv_spam}. \label{algomultivspam:query_f}

	\State Obtain $\est{\origsparsvec}(\vecx;\calA) = \csalg_i(\vecy, \matB)$. \label{algomultivspam:sparse_rec}
			
	\State Update $\est{\totsupp_i} = \est\totsupp_i \cup \set{\vecj \in \calA: \abs{\est{\origspars_{\vecj}}(\vecx_{\vecj})} > \epsilon_i}$. \label{algomultivpam:update_bivsupp}

\EndFor
\EndFor

\State Update $\calP = \calP \setminus \est\totsupp^{(1)}_i$.

\EndFor
\end{algorithmic}
\end{algorithm*}
%
%
Lemma \ref{lemma:multiv_spam_algo} below provides sufficient conditions on the sampling parameters in Algorithm \ref{algo:multiv_spam} 
for exact recovery of all $\totsupp_i$'s. 
\begin{lemma} \label{lemma:multiv_spam_algo}
For each $i\in[\order]$ assume that the following hold:
\begin{enumerate}
\item $m_i \geq \sqrt{i}\bigl(\frac{3L}{\idenconst_i}\bigr)^{1/\alpha}$.

\item $\csalg_i$ is $\epsilon_i$ accurate with $\epsilon_i < \idenconst_i/3$  for 
all $\hashfn \in \hashfam_i^{\calP_i}$, $\vecx \in \baseset(\hashfn)$, where 
$\calP_i$ denotes the set $\calP$ at the beginning of iteration $i$ (so $\calP_{\order} = [d]$).
The number of measurements used by $\csalg_i$ is denoted by $n_i$.

\item $\hashfam_i^{\calP_i}$ is a $(\calP_i,i)$ hash family.
\end{enumerate}
Then $\est{\totsupp}_i = \totsupp_i$ for all $i=\order,\order-1,\dots,1$ in Algorithm \ref{algo:multiv_spam}. 
Moreover, the total number of queries of $f$ made is 
$$\sum_{i=1}^{\order} 2^i (2m_i + 1)^i n_i \abs{\hashfam_i^{\calP_i}}.$$
\end{lemma}
\begin{proof}
The proof outline builds on what we have seen in the preceding sections. 
Say we are at the beginning of iteration $i\in[\order]$ with $\est{\totsupp}_l = \totsupp_l$ holding true for each $l > i$. 
Hence, the model has reduced to an order $i$ sparse additive model on the set $\calP_i \subset [d]$, with 
$\totsupp_i^{(1)},\totsupp_{i-1}^{(1)},\dots,\totsupp_1 \subset \calP_i$. 

By identifiability assumption, we know that for any given $\vecj \in \totsupp_i$, there exists $\vecx^{*}_{\vecj} \in [-1,1]^i$
such that $\abs{\phi_{\vecj}(\vecx^{*}_{\vecj})} \geq \idenconst_i$ holds. Moreover, since $\hashfam_i^{\calP_i}$ is a $(\calP_i,i)$ hash family,  
there exists a $\hashfn \in \hashfam_i^{\calP_i}$ that is an injection on $\vecj$.  
Consequently, there exists $\vecx \in \baseset(\hashfn)$ such that 
$\norm{\vecx_{\vecj} - \vecx^{*}_{\vecj}}_2 \leq (\sum_{p=1}^{i} \frac{1}{m_i^2})^{1/2} = \sqrt{i}/{m_{i}}$.
By H\"older continuity of $\phi_{\vecj}$, this means
\begin{align}
\abs{\phi_{\vecj}(\vecx_{\vecj}) - \phi_{\vecj}(\vecx^{*}_{\vecj})} \leq L \frac{i^{\alpha/2}}{m_i^{\alpha}}. \label{eq:multiv_lemma_proof_1}
\end{align}
Since $\csalg_i$ is $\epsilon_i$ accurate for each $\hashfn \in \hashfam_{i}^{\calP_i}$, $\vecx \in \baseset(\hashfn)$, we know that 
at the aforementioned $\vecx$, the following holds via reverse triangle inequality
\begin{equation}
\abs{\est{\origspars_{\vecj}}(\vecx_{\vecj})} \geq \abs{\phi_{\vecj}(\vecx_{\vecj})} - \epsilon_i.  \label{eq:multiv_lemma_proof_2}
\end{equation}
Using \eqref{eq:multiv_lemma_proof_1}, \eqref{eq:multiv_lemma_proof_2}, reverse triangle inequality and the choice of $\epsilon_i$ and $m_i$, we obtain
\begin{align*}
\abs{\est{\origspars_{\vecj}}(\vecx_{\vecj})} 
&\geq \abs{\phi_{\vecj}(\vecx^{*}_{\vecj})} - L \frac{i^{\alpha/2}}{m_i^{\alpha}} - \epsilon_i
\geq \idenconst_i - L \frac{i^{\alpha/2}}{m_i^{\alpha}} - \epsilon_i 
\geq \frac{2\idenconst_i}{3} - L \frac{i^{\alpha/2}}{m_i^{\alpha}} 
\geq \frac{\idenconst_i}{3}. 
\end{align*}
For any $\vecj \notin \totsupp_i$, we have for all $\hashfn \in \hashfam_{i}^{\calP_i}$, $\vecx \in \baseset(\hashfn)$ that  
$\abs{\est{\origspars_{\vecj}}(\vecx_{\vecj})} \leq \epsilon_i < \idenconst_i/3$ (since $\phi_{\vecj} \equiv 0$). 
Hence clearly, the stated choice of $\epsilon_i$ guarantees identification of each $\vecj \in \totsupp_i$, 
and none from ${\calP_i \choose i} \setminus \totsupp_i$.
This means that we will recover $\totsupp_i$ exactly. As this is true for each $i \in [\order]$, it also completes the proof 
for exact recovery of $\totsupp_i$ for each $i \in [\order]$. 

The expression for the total number of queries made follows from a simple calculation where we note that at iteration $i$, 
and corresponding to each $\vecx \in \cup_{\hashfn \in \hashfam_{i}^{\calP_i}} \baseset(\hashfn)$, we make $2^i n_i$ queries of $f$. 
\end{proof}
\section{Estimating sparse multilinear functions from few samples}\label{sec:sparse_rec_results}
In this section, we provide results from the sparse recovery literature for estimating 
sparse multilinear forms from random samples. In particular, these results cover 
arbitrary bounded noise and i.i.d. Gaussian noise models.

\subsection{Sparse linear functions} \label{subsec:sparse_lin_funcs}
Consider a linear function $g: \matR^d \rightarrow \matR$, where $g(\sampsgnvec) = \sampsgnvec^T\veca$. 
Our interest is in recovering the unknown coefficient vector $\veca$ from 
$n$ noisy samples $y_i = g(\sampsgnvec_i) + \noise_i$, $i\in[n]$, where $\noise_i$ refers to the 
noise in the $i^{\text{th}}$ sample. Arranging the samples together, we arrive at the 
linear system $\vecy = \matB \veca + \noisevec$, where
\begin{equation} \label{eq:sparse_lin_sys}
\underbrace{\begin{pmatrix}
  y_1  \\ \vdots \\ \vdots \\ y_n
 \end{pmatrix}}_{\vecy \in \matR^{n}}
= 
\underbrace{\begin{pmatrix}
  {\sampsgnvec_1}^{T}  \\ \vdots \\ \vdots \\ {\sampsgnvec_n}^{T}
 \end{pmatrix}}_{\matB \in \matR^{n \times {d}}} \veca 
+ \underbrace{\begin{pmatrix}
  \noise_1  \\ \vdots \\ \vdots \\ \noise_n
 \end{pmatrix}}_{\noisevec}.
\end{equation}
Denoting by $\totsupp:= \set{j \in [d] : a_{j} \neq 0}$ the support of $\veca$, 
%
%
our interest is in the setting where $\veca$ is sparse, i.e., $\abs{\totsupp} = \totsparsity \ll d$, and consequently to 
estimate $\veca$ from a small number of samples $n$.
To begin with, we will require $\matB$ in \eqref{eq:sparse_lin_sys} to satisfy the so called $\ell_2/\ell_2$ 
RIP, defined below. 
\begin{definition} \label{def:rip_l2l2}
A matrix $\matA\in\matR^{n\times d}$ is said to satisfy the $\ell_2/\ell_2$ Restricted Isometry Property (RIP) of order $\totsparsity$ 
with constant $\riptt_{\totsparsity} \in (0,1)$ if
%
\begin{equation*}
(1-\riptt_{\totsparsity}) \norm{\vecx}_2^2 \ \leq \ \frac{1}{n}\norm{\matA \vecx}_2^2 \ \leq \ (1 + \riptt_{\totsparsity}) \norm{\vecx}_2^2
\end{equation*}
%
holds for all $\totsparsity$-sparse $\vecx$.
\end{definition}
%
%
\paragraph{Bounded noise model.} 
Let us consider the scenario where the noise is bounded in the $\ell_2$ norm, i.e., $\|{\noisevec}\|_2 \leq \arbnoisebd$. 
We will recover an estimate $\est{\veca}$ to $\veca$ as a solution of the following 
quadratically constrained $\ell_1$ minimization program \cite{Candes08restr}
\begin{equation} \label{eq:l1min_quadconstr}
\text{(P1)} \quad  \min_{\vecz \in \matR^{d}} \norm{\vecz}_1 \quad \text{s.t} \quad \norm{\vecy - \matB \vecz}_2 \leq \arbnoisebd.
\end{equation}
The following result provides a bound on the estimation error $\norm{\est{\veca} - \veca}_2$ for (P1). 
\begin{theorem} \label{thm:arbnoise_lin_rec}
Consider the sampling model in \eqref{eq:sparse_lin_sys}, where $\matB \in\{-1,+1\}^{n\times d}$ is a \randvar matrix.
Then the following hold.
\begin{enumerate}
\item (\cite{Baraniuk2008_simple}) 
For any constant $\riptt \in (0,1)$, there exist constants $c_1, c_2 > 0$ depending on $\riptt$ such that if 
$$n \geq c_1  \totsparsity \log (d/\totsparsity),$$ then with probability at least $1-2\exp(-c_2 n)$, 
the matrix $\matB$ satisfies $\ell_2/\ell_2$ RIP of order $\totsparsity$, with $\riptt_{\totsparsity} \leq \riptt$.

\item (\cite[Theorem 1.2]{Candes08restr}) Let $\matB$ satisfy the $\ell_2/\ell_2$ RIP with $\riptt_{2\totsparsity} < \sqrt{2}-1$.
Then there exist constants $C_1,C_2 > 0$ such that, simultaneously for all vectors $\veca \in \matR^{d}$,
any solution $\est{\veca}$ to (P1) satisfies
\begin{equation*}
\norm{\est{\veca} - \veca}_2 \leq C_1\frac{\norm{\veca-\veca_k}_1}{\sqrt{k}} + C_2\frac{\arbnoisebd}{\sqrt{n}}.
\end{equation*}
Here, $\veca_{\totsparsity}$ denotes the best $\totsparsity$-term approximation of $\veca$.
\end{enumerate}
\end{theorem}

\paragraph{Gaussian noise model.} 
We now consider the scenario where the noise samples are i.i.d. Gaussian with variance $\sigma^2$, i.e, 
$\noise_i \sim \calN(0,\sigma^2)$ i.i.d. for all $i\in[n]$. Using standard concentration inequalities 
for sub-exponential random variables \jvyb{(see Proposition \ref{prop:gaussian_norm_bd}(\ref{prop:gauss_l2norm_bd}) in Appendix \ref{app:std_conc_results})}, 
one can show that $\norm{\noisevec}_2 = \Theta(\sigma\sqrt{n})$ with high probability. 
This leads to the following straightforward corollary of Theorem \ref{thm:arbnoise_lin_rec}.
\begin{corollary} \label{cor:gauss_noise_lin_rec}
Consider the sampling model in \eqref{eq:sparse_lin_sys} for some given vector $\veca \in \matR^{n}$, and 
let $\noise_i \sim \calN(0,\sigma^2)$ i.i.d. for all $i\in[n]$. 
Say $\matB$ satisfies $\ell_2/\ell_2$ RIP with $\riptt_{2\totsparsity} < \sqrt{2}-1$. 
For some $\varepsilon \in (0,1)$, let $\est{\veca}$ be a solution to (P1) with $\arbnoisebd = (1+\varepsilon)\sigma\sqrt{n}$. 
Then there exists a constant $c_3 > 0$ so that any solution $\est{\veca}$ to (P1) satisfies
\begin{equation*}
\norm{\est{\veca} - \veca}_2 \leq C_1\frac{\norm{\veca-\veca_k}_1}{\sqrt{k}} + C_2 (1+\varepsilon)\sigma
\end{equation*}  
with probability at least $1 -2\exp(-c_3\varepsilon^2 n)$. Here $C_1, C_2$ are the constants 
from Theorem \ref{thm:arbnoise_lin_rec}.
\end{corollary}
\begin{proof}
Use \jvyb{Proposition \ref{prop:gaussian_norm_bd}(\ref{prop:gauss_l2norm_bd}) from Appendix \ref{app:std_conc_results}} with Theorem \ref{thm:arbnoise_lin_rec}.
\end{proof}
\subsection{Sparse bilinear functions} \label{subsec:sparse_bilin_funcs}
Let $\veca \in \matR^{{d \choose 2}}$ be a vector of length ${d \choose 2}$ with entries indexed 
by ${[d] \choose 2}$ (sorted in lexicographic order). 
The entry of $\veca$ at index $\vecj = (j_1,j_2) \in {[d] \choose 2}$ will be denoted by $a_{\vecj}$.
We now consider the setting where $g: \matR^d \rightarrow \matR$ is a second order multilinear function, i.e., 
$g(\sampsgnvec) = \dotprod{\betavect}{\veca}$ with $\betavect, \veca \in \matR^{{d \choose 2}}$, 
and $\betat_{(j_1, j_2)} = \sampsgn_{j_1} \sampsgn_{j_2}$.
As before, our goal is to recover $\veca$ from $n$ noisy samples $y_i = g(\sampsgnvec_i) + \noise_i$, $i\in[n]$ resulting 
in the linear system
\begin{equation} \label{eq:sparse_bilin_sys}
\underbrace{\begin{pmatrix}
  y_1  \\ \vdots \\ \vdots \\ y_n
 \end{pmatrix}}_{\vecy \in \matR^{n}}
= 
\underbrace{\begin{pmatrix}
  {\betavect_1}^{T}  \\ \vdots \\ \vdots \\ {\betavect_n}^{T}
 \end{pmatrix}}_{\matB \in \matR^{n \times {d \choose 2}}} \veca 
+ \underbrace{\begin{pmatrix}
  \noise_1  \\ \vdots \\ \vdots \\ \noise_n
 \end{pmatrix}}_{\noisevec}.
\end{equation}
Observe that
\begin{equation} \label{eq:sparsemat_quad_meas_rel}
\dotprod{\betavect}{\veca} = \sampsgnvec^T \matA \sampsgnvec = \dotprod{\sampsgnvec\sampsgnvec^T}{\matA},
\end{equation}
where $\matA \in \matR^{d \times d}$ is a symmetric matrix with zero on the diagonal and $A_{i,j} = a_{(i,j)}/2$ if $i \neq j$. 
This simple observation allows us to rewrite \eqref{eq:sparse_bilin_sys} as
\begin{equation} \label{eq:quad_sketch}
y_i = \dotprod{\sampsgnvec_i\sampsgnvec_i^T}{\matA} + \noise_i,\quad i\in[n].
\end{equation}
Chen et al. \cite{Chen15} recently showed that a sparse symmetric matrix -- not necessarily being 
zero on the diagonal -- can be recovered from 
its measurements of the form \eqref{eq:quad_sketch}, provided that $\sampsgn_j$ is 
sampled in an i.i.d. manner from a distribution satisfying
\begin{equation} \label{eq:quadmeas_momemt_conds}
\expec[\sampsgn_i] = 0,\quad \expec[\sampsgn_i^2] = 1 \quad \text{and} \quad \expec[\sampsgn_i^4] > 1. 
\end{equation}
Their recovery program\footnote{Note that one can remove the positive 
semi-definite constraint in their program and replace it with a symmetry enforcing constraint, the result remains 
unchanged (cf. remark in \cite[Section E]{Chen15})} is essentially constrained $\ell_1$ minimization (\cite[Eq.(4)]{Chen15}), 
the guarantees for which rely on the $\ell_2/\ell_1$ RIP for sparse symmetric matrices (\cite[Def. 2]{Chen15}).

Since our matrix $\matA$ in \eqref{eq:quad_sketch} is sparse and symmetric, it is natural for us 
to use their scheme. We do this, albeit with some technical changes:
\begin{itemize}
\item We will see that if $\matA$ is known to be zero on the diagonal (which is the case here), 
the fourth order moment condition in \eqref{eq:quadmeas_momemt_conds} is not needed. 
Hence one could, for instance, sample $\sampsgn$ from the 
symmetric \randvar distribution. 

\item Instead of optimizing over the set of symmetric matrices with zeros on the diagonal, we 
will perform $\ell_1$ minimization over the upper triangular entries of $\matA$, represented by $\veca$. 
These approaches are equivalent, but the latter has the computational advantage of having fewer constraints.
\end{itemize}
The analysis is based on the notion of $\ell_2/\ell_1$ RIP of the matrix $\matB$ in \eqref{eq:sparse_bilin_sys}, which is defined as follows.
%
\begin{definition} \label{def:rip_l2l1}
A matrix $\matB \in \matR^{n \times N}$ is said to satisfy the $\ell_2/\ell_1$ Restricted Isometry Property (RIP) of order $\totsparsity$ 
with constants $\riptolb_{\totsparsity} \in (0,1)$ and $\riptoub_{\totsparsity} > 0$ if
\begin{equation*}
(1-\riptolb_{\totsparsity}) \norm{\vecx}_2 \ \leq \ \frac{1}{n}\norm{\matB \vecx}_1 \ \leq \ (1 + \riptoub_{\totsparsity}) \norm{\vecx}_2
\end{equation*}
holds for all $\totsparsity$-sparse $\vecx \in \matR^{N}$.
\end{definition}
The above definition is analogous to the one in \cite[Def. 2]{Chen15} for sparse symmetric matrices. 
\paragraph{Bounded noise model.} 
Let us first consider the setting where the noise is bounded in the $\ell_1$ norm, i.e., 
$\|\noisevec\|_1 \leq \arbnoisebd$. We recover the estimate $\est{\veca}$ as a solution to 
the following program
\begin{equation} \label{eq:sparserec_prog_ordtwo}
\text{(P2)} \quad \min_{\vecz \in \matR^{{d \choose 2}}} \norm{\vecz}_1 \quad \text{s.t} \quad \norm{\vecy - \matB \vecz}_1 \leq \arbnoisebd.
\end{equation}
The next result shows that $\matB$ satisfies $\ell_2/\ell_1$ RIP with high probability
if the rows of $\matB$ are formed by independent \randvar vectors.
Consequently, the above program stably recovers $\veca$. 
%
\begin{theorem} \label{thm:arbnoise_bilin_rec}
Consider the sampling model in \eqref{eq:sparse_bilin_sys}, where 
the rows of $\matB$ are formed by independent \randvar vectors.
Then the following hold.
\begin{enumerate}
\item
There exist absolute constants $c_1,c_2,c_3>0$, such that the following is true. Let $\veca \in \matR^{{d \choose 2}}$. Then
\begin{equation}\label{eq:bilin_new1}
c_1 \norm{\veca}_2 \leq  \frac{1}{n} \norm{\matB \veca}_1 \leq c_2 \norm{\veca}_2
\end{equation}
holds with probability at least $1-\exp(-c_3 n)$.
%

\item With constants $c_1,c_2$ from \eqref{eq:bilin_new1}, there exist constants $c'_1,c'_2,c'_3 > 0$ such that if 
$n > c'_3 \totsparsity \log (d^2/\totsparsity)$, then 
$\matB$ satisfies $\ell_2/\ell_1$ RIP of order $\totsparsity$ with probability at least $1-c'_1 \exp(-c'_2 n)$ with constants
$\riptolb_{\totsparsity}$ and $\riptoub_{\totsparsity}$, which fulfill
\begin{equation}\label{eq:gammas}
1 - \riptolb_{\totsparsity} \geq \frac{c_1}{2}\quad\text{and} \quad 1 + \riptoub_{\totsparsity} \leq 2 c_2.
\end{equation}

\item If there exists a number $K>2k$ such that $\matB$ satisfies
$$
\frac{1-\riptolb_{k+K}}{\sqrt{2}}-(1+\riptoub_{K})\sqrt{\frac{k}{K}}\ge \beta>0
$$
for some $\beta>0$, then the solution $\est{\veca}$ to (P2) satisfies
\begin{equation*}
\norm{\est{\veca} - \veca}_2 \leq \Bigl(\frac{\tilde C_1}{\beta}+\tilde C_3\Bigr) \frac{\norm{\veca - \veca_{\totsparsity}}_1}{\sqrt{K}} + \frac{\tilde C_2}{\beta} \frac{\arbnoisebd}{n},
\end{equation*}
where $\veca_{\totsparsity}$ denotes the best $\totsparsity$-term approximation of $\veca$ and $\tilde C_1, \tilde C_2, \tilde C_3$ are universal positive constants.

\item There exist absolute constants $\tilde c_3, \tilde c_4, C_3, C_4 > 0$ such that if $n\ge \tilde c_3 k\log(d^2/k)$, 
then the solution $\est{\veca}$ to (P2) satisfies
\begin{equation*}
\norm{\est{\veca} - \veca}_2 \leq C_3 \frac{\norm{\veca - \veca_{\totsparsity}}_1}{\sqrt{\totsparsity}} + C_4 \frac{\arbnoisebd}{n},
\end{equation*}
simultaneously for all $\veca \in \matR^{{d \choose 2}}$ with probability at least $1-\exp(-\tilde c_4 n)$.
\end{enumerate}
\end{theorem}
\jvyb{We sketch the proof of Theorem \ref{thm:arbnoise_bilin_rec}, which is essentially based on \cite{Chen15},
in Appendix \ref{app:proof_thm_sparse_bilin}.}

%
%
\paragraph{Gaussian noise model.} 
We now consider the scenario where the noise samples are i.i.d. Gaussian with variance $\sigma^2$, i.e, 
$\noise_i \sim \calN(0,\sigma^2)$ i.i.d. for all $i\in[n]$. Using standard concentration inequalities 
for sub-Gaussian random variables (see \jvyb{Proposition \ref{prop:gaussian_norm_bd}(\ref{prop:gauss_l1norm_bd}) in Appendix \ref{app:std_conc_results})}, 
one can show that $\norm{\noisevec}_1 = \Theta(\sigma n)$ with high probability. 
This leads to the following straightforward corollary of Theorem \ref{thm:arbnoise_bilin_rec}.
\begin{corollary} \label{cor:gauss_noise_bilin_rec}
For constants $\tilde c_3, \tilde c_4, C_3, C_4 > 0$ defined in Theorem \ref{thm:arbnoise_bilin_rec}, the following is true.
Consider the sampling model in \eqref{eq:sparse_bilin_sys} for a given $\veca \in \matR^{{d \choose 2}}$,
where 
the rows of $\matB$ are formed by independent \randvar vectors,
and  $\noise_i \sim \calN(0,\sigma^2)$ i.i.d. for all $i\in[n]$ with $n \ge  \tilde c_3 k\log(d^2/k)$.
For some $\varepsilon \in (0,1)$, let $\est{\veca}$ be a solution to (P2) with $\arbnoisebd = (1+\varepsilon)\sigma n$.
Then
\begin{equation*}
\norm{\est{\veca} - \veca}_2 \leq C_3 \frac{\norm{\veca - \veca_{\totsparsity}}_1}{\sqrt{\totsparsity}} + C_4 \sqrt{\frac{2}{\pi}}(1+\varepsilon) \sigma
\end{equation*}  
with probability at least $1 - \exp(- \tilde c_4 n) - e \cdot \exp\left(-\tilde c_5 \varepsilon^2 n\right)$, for some 
constant $\tilde c_5 > 0$. Here, $\veca_{\totsparsity}$ denotes the best $\totsparsity$-term approximation of $\veca$.
\end{corollary}
\begin{proof}
Use \jvyb{Proposition \ref{prop:gaussian_norm_bd}(\ref{prop:gauss_l1norm_bd}) from Appendix \ref{app:std_conc_results}} with Theorem \ref{thm:arbnoise_bilin_rec}.
\end{proof}
%
%
\begin{remark} \label{rem:comp_cost_sparse_rec}
\jvyb{Both (P1) and (P2) are convex programs where (P1) can be cast as a second order cone program (SOCP) (eg., \cite{l1magic}) 
and (P2) can be easily written as a linear program (LP). Hence they can be solved to arbitrary accuracy in 
time polynomial in $n$ and $d$, using for instance interior point algorithms (eg., \cite{NestNemi94,socp_goldfarb03}). In practice, 
one often considers non-convex alternatives such as Iterative Hard Thresholding (IHT) (eg. \cite{BLUMENSATH2009,Blumensath10}) 
which typically have low computational cost.}
\end{remark}

\subsection{Sparse multilinear functions} \label{subsec:sparse_multilin_funcs}
For $p,d \in \mathbb{N}$ with $p \leq d$, let $\veca \in \matR^{{d \choose p}}$ be a vector 
with entries indexed by ${[d] \choose p}$ (sorted in lexicographic order). 
We denote again the entry of $\veca$ at index $\vecj = (j_1,\dots,j_p) \in {[d] \choose p}$ by $a_{\vecj}$
and consider a multilinear function $g: \matR^d \rightarrow \matR$ in $d$ variables 
$\sampsgnvec = (\sampsgn_1 \dots \sampsgn_d)^T$ such that  
\begin{equation} \label{eq:multlin_def}
g(\sampsgnvec) = \sum_{\vecj = (j_1,\dots,j_p) \in {[d] \choose p}} \sampsgn_{j_1} \sampsgn_{j_2} \cdots \sampsgn_{j_p} a_{\vecj}. 
\end{equation}
We will refer to \eqref{eq:multlin_def} as a multilinear 
function of order $p$. For clarity of notation, we will write \eqref{eq:multlin_def} as 
$g(\sampsgnvec) = \dotprod{\betavecp}{\veca}$, where $\betavecp \in \matR^{{d \choose p}}$, 
and the entry of $\betavecp$ at index $\vecj = (j_1,\dots,j_p) \in {[d] \choose p}$ 
being $\sampsgn_{j_1} \sampsgn_{j_2} \cdots \sampsgn_{j_p}$.

We are again interested in recovering the unknown coefficient vector $\veca$ from 
$n$ noisy samples $y_i = g(\sampsgnvec_i) + \noise_i, i\in[n]$, where $\noise_i$ refers to the
noise in the $i^{\text{th}}$ sample. Arranging the samples together, we arrive at the 
linear system $\vecy = \matB \veca + \noisevec$, where
\begin{equation} \label{eq:sparse_multlin_sys}
\underbrace{\begin{pmatrix}
  y_1  \\ \vdots \\ \vdots \\ y_n
 \end{pmatrix}}_{\vecy \in \matR^{n}}
= 
\underbrace{\begin{pmatrix}
  {\betavecp_1}^{T}  \\ \vdots \\ \vdots \\ {\betavecp_n}^{T}
 \end{pmatrix}}_{\matB \in \matR^{n \times {d \choose p}}} \veca 
+ \underbrace{\begin{pmatrix}
  \noise_1  \\ \vdots \\ \vdots \\ \noise_n
 \end{pmatrix}}_{\noisevec}.
\end{equation}
We denote by $\totsupp=\Bigl\{(j_1,j_2,\dots,j_p) \in {[d] \choose p} : a_{(j_1,j_2,\dots,j_p)} \neq 0\Bigr\}$ the support of $\veca$ 
%
%
and we are especially interested in the setting where $\veca$ is sparse, i.e., $\abs{\totsupp} = \totsparsity \ll {d \choose p}$.
Our aim is to estimate $\veca$ with a small number of samples $n$. 

\paragraph{Bounded noise model.} 
Let us consider the scenario where the noise is bounded in the $\ell_2$ norm, i.e., $\|\noisevec\|_2 \leq \arbnoisebd$. 
We will recover an estimate $\est{\veca}$ to $\veca$ as a solution of (P1).
The following result provides a bound on the estimation error $\norm{\est{\veca} - \veca}_2$. 
\begin{theorem}(\cite[Theorem 4; Lemmas 4, 5]{Nazer2010}) \label{thm:arbnoise_multlin_rec}
Let $\dchp = {d \choose p}$ 
and let $\matB$ be defined as in \eqref{eq:sparse_multlin_sys} with rows formed by independent \randvar vectors.
Then, for any $\riptt \in (0,1)$, there exist constants $c_6, c_7 > 0$ depending on $\riptt$ such that 
the matrix $\matB$ satisfies $\ell_2/\ell_2$ RIP of order $\totsparsity$, with $\riptt_{\totsparsity} \leq \riptt$, 
\begin{enumerate} 
\item[(a)] with probability at least $1-\exp(-c_7  n/\totsparsity^2)$ if $n \geq c_6 \totsparsity^2 \log D$ \label{itm:gerschgorin_rip}
\item[(b)] and with probability at least $1-\exp(-c_7 \min\set{n/3^{2p}, n/\totsparsity})$ if 
$$n \geq c_6  \max\set{3^{2p} \totsparsity \log(D/\totsparsity), \totsparsity^2 \log(D/\totsparsity)}.
$$  
\end{enumerate}
\end{theorem}
The bounds on $n$ in the theorem are obtained via the application of two very different methodologies.
The bound in part (a) 
is a consequence of bounding the eigenvalues of the $\totsparsity \times \totsparsity$ Gram 
matrices $\frac{1}{n}\matB_{\totsupp}^T \matB_{\totsupp}$ for all $\totsupp$ (here $\matB_{\totsupp}$ is the submatrix 
of $\matB$ with column indices in $\totsupp$) using Gershgorin's disk theorem,  
along with standard concentration inequalities \cite[Theorem 4]{Nazer2010}. 
The bound in part (b) 
involves the usage of tail estimates 
for Rademacher chaos variables and follows from \cite[Lemmas 4, 5]{Nazer2010}. 
\jvyb{We also note that in Theorem \ref{thm:arbnoise_multlin_rec}, the number of measurements $n$ scales 
quadratically with the sparsity parameter $\totsparsity$. However when $p=2$, we can see that Theorem \ref{thm:arbnoise_bilin_rec} 
is stronger since $n$ therein scales linearly with $\totsparsity$.}
\begin{remark}
The analysis in \cite[Section A]{Nazer2010} derives RIP bounds in terms of the so-called combinatorial dimension.
However in our opinion, this analysis has several inaccuracies because of which we are not sure if the corresponding bounds are correct.
Hence we do not state those bounds here. 
\end{remark}
%
%

%
%

\paragraph{Gaussian noise model.} 
In the scenario where the noise in the samples is i.i.d. Gaussian, we arrive at a 
statement similar to Corollary \ref{cor:gauss_noise_lin_rec}, hence we do not discuss this further.

\section{Putting it together: final theoretical guarantees} \label{sec:put_togeth_final}
We are now in a position to combine our efforts from the preceding sections and to state the final results 
for recovery of the support sets. As before, we state this separately for the univariate, bivariate, and general 
multivariate cases.
%
\subsection{Univariate case} \label{subsec:final_res_univ}
We begin with the univariate case considered in Section \ref{sec:spam_univ}. Recall Algorithm \ref{algo:univ_spam} 
for recovering the support $\univsupp$. The ensuing Lemma \ref{lemma:univ_spam_algo} gave sufficient conditions 
for exact recovery provided $\csalg$ is $\epsilon$-accurate at each $\vecx \in \baseset$, for small enough $\epsilon$. Instantiating 
$\csalg$ with (P1) in \eqref{eq:l1min_quadconstr} gives the final results below. Let us start with the bounded noise model.
%
\paragraph{Bounded noise model.} In this noise model, querying $f$ at $\vecx$ returns $f(\vecx) + \noise$, where 
$\abs{\noise} \leq \noisebd$. For the linear system \eqref{eq:lin_sys_univspam}, this means that 
$\abs{\noise_i^{+}}, \abs{\noise_i^{-}} \leq \noisebd$, $i\in[n]$, and 
hence $\|\noisevec\|_{\infty} \leq 2\noisebd$. The following theorem shows that if $\noisebd$ is sufficiently small, then 
Algorithm \ref{algo:univ_spam} recovers $\univsupp$ exactly provided the parameters $m,n,\epsilon$ are chosen in a suitable way. 
%
%
\begin{theorem} \label{thm:main_univ_arbnoise}
For the bounded noise model with the noise uniformly bounded by $\noisebd$, 
consider Algorithm \ref{algo:univ_spam} with $\csalg$ instantiated with $(P1)$ in \eqref{eq:l1min_quadconstr}. 
If $\arbnoisebd = 2\noisebd \sqrt{n}$ in $(P1)$,
%
\begin{equation*}
 \noisebd < \frac{\idenconst_1}{6 C_2}, \quad m \geq (3L/\idenconst_1)^{1/\alpha}, \quad \text{and} \quad  
 n \geq \tilde{c}_1  \abs{\univsupp} \log (d/\abs{\univsupp})
\end{equation*}
%
are satisfied, it follows for the choice $\epsilon = 2 C_2 \noisebd$ that $\est{\univsupp} = \univsupp$, 
with probability at least $1 - 2\exp(-\tilde{c}_2 n)$. 
The total number of queries made is  
$2(2m+1)n = \Omega\left(\jvyb{\bar{c}_1}\abs{\univsupp} \log \left(\frac{d}{\abs{\univsupp}}\right)\right)$ \jvyb{where 
$\bar{c}_1 > 1$ depends on $L,\idenconst_1,\alpha$}.
\end{theorem}
\begin{proof}
The proof follows by combining Lemma \ref{lemma:univ_spam_algo} with Theorem \ref{thm:arbnoise_lin_rec}. 
Since $\abs{\noise_i^{+}}, \abs{\noise_i^{-}} \leq \noisebd$ in \eqref{eq:lin_sys_univspam} for $i\in[n]$, we obtain
$\|\noisevec\|_2 \leq \sqrt{n} \|\noisevec\|_{\infty} \leq 2\noisebd\sqrt{n} $. Therefore 
we set $\csalg = (P1)$ with $\arbnoisebd = 2\noisebd\sqrt{n}$.

As a consequence of Theorem \ref{thm:arbnoise_lin_rec}, there exist constants 
$\tilde{c}_1, \tilde{c}_2 > 0$ (depending only on ${c}_1, c_2 > 0$ defined therein) so that,
for $n \geq \tilde{c}_1  \abs{\univsupp} \log (d/\abs{\univsupp})$, $\matB$ satisfies 
$\ell_2/\ell_2$ RIP with $\riptt_{2\abs{\univsupp}} < \sqrt{2}-1$ with probability at least 
$1-2\exp(-\tilde{c}_2 n)$. Conditioning on this event, it follows from Theorem \ref{thm:arbnoise_lin_rec} 
that
\begin{equation*}
  \norm{\est{\origsparsvec}(\vecx) - \origsparsvec(\vecx)}_{\infty} 
\ \leq\  \norm{\est{\origsparsvec}(\vecx) - \origsparsvec(\vecx)}_{2} \ \leq \,\jvyb{2C_2 \noisebd} \quad \text{for all}\quad \vecx \in \baseset. 
\end{equation*}
%
Hence, we see that with probability at least $1-2\exp(-\tilde{c}_2 n)$, 
$\csalg$ is $\epsilon = 2 C_2 \noisebd$ accurate at each $\vecx \in \baseset$. Now invoking 
Lemma \ref{lemma:univ_spam_algo}, it follows that if 
\begin{equation*}
 2 C_2 \noisebd < \frac{\idenconst_1}{3} \Leftrightarrow \noisebd < \frac{\idenconst_1}{6C_2}
\end{equation*}
holds, then the stated choice of $\epsilon$ and $m$ ensures exact recovery. This completes the proof.
\end{proof}

\paragraph{Gaussian noise model.} We now move to the Gaussian noise model, wherein 
querying $f$ at $\vecx$ returns $f(\vecx) + \noise$; $\noise \sim \calN(0,\sigma^2)$. Moreover, the noise samples 
are independent across the queries. For the linear system \eqref{eq:lin_sys_univspam}, this means that
$\noise_i\sim\calN(0,2\sigma^2), i\in[n]$ are i.i.d. random variables.
The following theorem essentially shows that if the noise variance $\sigma^2$ is sufficiently small, then 
Algorithm \ref{algo:univ_spam} recovers $\univsupp$ exactly provided the parameters $m,n,\epsilon$ are chosen properly. 
The reduction in the variance is handled via re-sampling each query $N$ times and averaging the values. Essentially, 
this leads to the same sampling model with i.i.d. $\noise_i \sim \calN(0,2\sigma^2/N), i\in[n]$.
\begin{theorem} \label{thm:main_univ_gauss_noise} 
For the Gaussian noise model with i.i.d. noise samples with variance $\sigma^2$ and $N\in\matN$, we resample each query 
$N$ times, and average the values. 
Consider Algorithm \ref{algo:univ_spam} with $\csalg$ instantiated with $(P1)$ in \eqref{eq:l1min_quadconstr}, wherein 
$\arbnoisebd = \sqrt{2}(1+\varepsilon)\sigma\sqrt{n/N}$ for some $\varepsilon \in (0,1)$, and with 
\begin{equation}\label{eq:univ_gauss_1}
N \geq \left\lfloor \frac{18C_2^2 (1+\varepsilon)^2 \sigma^2}{\idenconst_1^2} \right\rfloor + 1, \ m \geq \left(\frac{3L}{\idenconst_1}\right)^{1/\alpha}, \ \  
 n \geq \max\set{\tilde{c}_1  \abs{\univsupp} \log \left(\frac{d}{\abs{\univsupp}}\right), \frac{2\log(2m+1)}{c_3 \varepsilon^2}}
\end{equation}
being satisfied. 
Then $\est{\univsupp} = \univsupp$ with probability at least 
$1 - 2\exp(-\tilde{c}_2 n) - 2\exp(-\frac{c_3 \varepsilon^2 n}{2})$. The constants $\tilde{c}_1,\tilde{c}_2 > 0$ are as 
in Theorem \ref{thm:main_univ_arbnoise}; $C_2, c_3 > 0$ come from Theorem \ref{thm:arbnoise_lin_rec} and 
Corollary \ref{cor:gauss_noise_lin_rec}, respectively.
The total number of queries made is 
$2(2m+1) n N = \Omega\left(\jvyb{\bar{c}_1^{\prime}}\abs{\univsupp} \log \left(\frac{d}{\abs{\univsupp}}\right)\right)$ 
\jvyb{where $\bar{c}_1^{\prime} > 1$ depends on $L,\idenconst_1,\alpha,\sigma$}.
\end{theorem}
%
\begin{proof}
First note that in \eqref{eq:lin_sys_univspam}, as a consequence of resampling $N$ times and averaging, we have 
$\noise_i^{+}, \noise_i^{-} \sim \calN(0,\sigma^2/N)$, and $\noise_i = \noise_i^{+} - \noise_i^{-} \sim \calN(0,2\sigma^2/N), i \in[n]$.
Hence we set $\csalg = (P1)$ with $\arbnoisebd = \sqrt{2}(1+\varepsilon)\sigma\sqrt{n/N}$.

From Theorem \ref{thm:arbnoise_lin_rec}, there exist constants 
$\tilde{c}_1, \tilde{c}_2 > 0$ (depending only on ${c}_1, c_2 > 0$ defined therein) so that for the choice  
$n \geq \tilde{c}_1  \abs{\univsupp} \log (d/\abs{\univsupp})$, $\matB$ satisfies
$\ell_2/\ell_2$ RIP with $\riptt_{2\abs{\univsupp}} < \sqrt{2}-1$ with probability at least 
$1-2\exp(-\tilde{c}_2 n)$. Conditioning on this event, and invoking Corollary \ref{cor:gauss_noise_lin_rec} for the 
stated choice of $\arbnoisebd$, 
we have for any given $\vecx \in \baseset$ that with probability at least $1 - 2\exp(-c_3 \varepsilon^2 n)$, the following holds
\begin{equation} \label{eq:thm_gauss_uinv_spam_1}
 \norm{\est{\origsparsvec}(\vecx) - \origsparsvec(\vecx)}_{\infty} 
\leq \norm{\est{\origsparsvec}(\vecx) - \origsparsvec(\vecx)}_{2} 
\leq 	\sqrt{2}C_2 (1+\varepsilon)\sigma/\sqrt{N}.
\end{equation}
By the union bound over the $2m+1$ elements of $\baseset$, it follows that \eqref{eq:thm_gauss_uinv_spam_1} holds for all 
$\vecx \in \baseset$ with probability at least 
\begin{align*}
1 - 2\abs{\baseset}\exp(-c_3 \varepsilon^2 n) 
&= 1 - 2\exp(\log(2m+1)-c_3 \varepsilon^2 n) 
\geq 1 - 2\exp\Bigl(-\frac{c_3 \varepsilon^2 n}{2}\Bigr)
\end{align*}
if $n \geq \frac{2\log(2m+1)}{c_3 \varepsilon^2}$ holds.
By \eqref{eq:univ_gauss_1}, this condition is indeed satisfied and furthermore, we see
that $\csalg$ is $\epsilon = \sqrt{2}C_2 (1+\varepsilon)\sigma/\sqrt{N}$
accurate for all $\vecx \in \baseset$ with probability at least $1 - 2\exp(-\tilde{c}_2 n)- 2\exp(-\frac{c_3 \varepsilon^2 n}{2})$.
Now invoking Lemma \ref{lemma:univ_spam_algo}, and using \eqref{eq:thm_gauss_uinv_spam_1}, it follows that if 
\begin{equation} \label{eq:thm_gauss_uinv_spam_2}
\frac{\sqrt{2}C_2 (1+\varepsilon)\sigma}{\sqrt{N}} < \frac{\idenconst_1}{3}
\end{equation}
holds, then the stated choice of $\epsilon$ and $m$ ensures exact recovery. 
Finally, \eqref{eq:univ_gauss_1} ensures that \eqref{eq:thm_gauss_uinv_spam_2} holds and 
this completes the proof.
\end{proof}

\subsection{Bivariate case} \label{subsec:final_res_biv}
In the bivariate case from Section \ref{sec:spam_biv}, we use Algorithm \ref{algo:biv_spam} 
for recovering the supports $\bivsupp,\univsupp$ and with the Lemma \ref{lemma:biv_spam_algo}
giving sufficient conditions for their exact recovery.
Instantiating $\csalg_2$ with (P2) in \eqref{eq:sparserec_prog_ordtwo}, and $\csalg_1$ 
with (P1) in \eqref{eq:l1min_quadconstr} gives then the results below. 
%
\paragraph{Bounded noise model.} 
\jvyb{The} following theorem shows that if $\noisebd$ is sufficiently small in the bounded noise model described before,
then Algorithm \ref{algo:biv_spam} recovers $\bivsupp$ and $\univsupp$ exactly provided the 
parameters $m_i,n_i,\epsilon_i$, $i=1,2$ are chosen in a suitable way. 
\begin{theorem} \label{thm:main_biv_arbnoise}
For the bounded noise model with the noise uniformly bounded by $\noisebd$,
consider Algorithm \ref{algo:biv_spam} with $\csalg_2$ instantiated with 
$(P2)$ with $\arbnoisebd_2 = 4\noisebd n_2$ in \eqref{eq:sparserec_prog_ordtwo},
and $\csalg_1$ realized by $(P1)$ with $\arbnoisebd_1 = 2\noisebd \sqrt{n_1}$ in \eqref{eq:l1min_quadconstr}, respectively.
Let $\twohashfam$ be a $(d,2)$ hash family. 
If  
\begin{align*}
\noisebd &< \min\set{\frac{\idenconst_2}{12 C_4},\frac{\idenconst_1}{6 C_2}}, \ m_2 \geq \sqrt{2}\left(\frac{3L}{\idenconst_2}\right)^{1/\alpha}, \ 
n_2 \geq \tilde{c}_3 \abs{\bivsupp} \log(d^2/\abs{\bivsupp}),  \\
m_1 &\geq (3L/\idenconst_1)^{1/\alpha} \quad \text{and} \quad n_1 \geq \tilde{c}_1 
\abs{\univsupp} \log \biggl(\frac{d-\abs{\est{\bivsuppvar}}}{\abs{\univsupp}}\biggr)
\end{align*}
are satisfied, then 
$\est{\bivsupp} = \bivsupp$ and $\est{\univsupp} = \univsupp$ 
with probability at least $1 - \exp(-\tilde{c}_4 n_2) - 2\exp(-\tilde{c}_2 n_1)$. 
Here, the constants $C_4,\tilde{c}_3,\tilde{c}_4 > 0$ are 
from Theorem \ref{thm:arbnoise_bilin_rec}, while $\tilde{c}_1, \tilde{c}_2, C_2 > 0$ are as defined in Theorem \ref{thm:main_univ_arbnoise}. 
The total number of queries made is 
\begin{equation*}
4(2m_2 + 1)^2 n_2 \abs{\twohashfam} + 2(2m_1 + 1) n_1 
= \Omega\left(\jvyb{\bar{c}_2}\abs{\bivsupp} \log\left(\frac{d^2}{\abs{\bivsupp}}\right) \abs{\twohashfam} 
+ \jvyb{\bar{c}_1}\abs{\univsupp} \log \left(\frac{d-\abs{\est{\bivsuppvar}}}{\abs{\univsupp}}\right)\right).
\end{equation*}
\jvyb{Here $\bar{c}_i > 1$ depends on $L,\idenconst_i,\alpha$ with $\bar{c}_1$ as in Theorem \ref{thm:main_univ_arbnoise}.}
\end{theorem}
\begin{proof}
We focus on the proof of the exact recovery of $\bivsupp$. Once $\bivsupp$ is recovered exactly, the model 
reduces to a univariate SPAM on the variable set $\calP = [d] \setminus \bivsuppvar$, with $\univsupp \subset \calP$. 
Thereafter, the proof of the exact recovery of $\univsupp$ is identical to the proof of Theorem \ref{thm:main_univ_arbnoise}. 

Since $\abs{\noise_{i,1}}, \abs{\noise_{i,2}},\abs{\noise_{i,3}},\abs{\noise_{i,4}} \leq \noisebd$ in \eqref{eq:lin_sys_bivspam} for $i\in[n_2]$, 
we obtain $\|\noisevec\|_1 \leq n_2 \|\noisevec\|_{\infty} \leq 4 n_2 \noisebd$. Therefore 
we set $\csalg_2 = (P2)$ with $\arbnoisebd = 4\noisebd n_2$.

Now, as a consequence of Theorem \ref{thm:arbnoise_bilin_rec}, there exist constants  
$\tilde c_3, \tilde c_4, C_4 > 0$ such that if $n_2 \ge \tilde c_3 \abs{\bivsupp} \log(d^2/\abs{\bivsupp})$, 
then with probability at least $1-\exp(-\tilde c_4 n_2)$, 
\begin{equation*}
\norm{\est{\origsparsvec}(\vecx;\calA) - \origsparsvec(\vecx;\calA)}_{\infty} \ \leq \ 
\norm{\est{\origsparsvec}(\vecx;\calA) - \origsparsvec(\vecx;\calA)}_{2}\ \leq\, 4C_4\noisebd \quad \text{for all}\ \vecx \in \bigcup_{\hashfn \in \twohashfam} \baseset(\hashfn). 
\end{equation*}
This holds since $\origsparsvec(\vecx;\calA)$ is always at most $\abs{\bivsupp}$ sparse. 
Thus, with probability at least $1-\exp(-\tilde c_4 n_2)$, $\csalg_2$ is $\epsilon_2 = 4 C_4 \noisebd$ 
accurate for each $\hashfn \in \twohashfam$, $\vecx \in \baseset(\hashfn)$. 
Now invoking Lemma \ref{lemma:biv_spam_algo} reveals that if 
\begin{equation*}
4 C_4 \noisebd < \frac{\idenconst_2}{3} \Leftrightarrow \noisebd < \frac{\idenconst_2}{12 C_4}
\end{equation*}
holds, then the stated choice of $\epsilon_2$ and $m_2$ ensures $\est{\bivsupp} = \bivsupp$. 

\jvyb{In order to derive the lower bound on the success probability of identifying $\univsupp,\bivsupp$, we will 
make use of the following simple union bound inequality. For events $\calA,\calB$, it holds that 
\begin{equation} \label{eq:useful_union_bd_biv}
\prob(\calA \cup\calB) 
= \prob(\calA \cup \set{\calA^{c} \cap \calB}) 
= \prob(\calA) + \prob(\calA^{c} \cap \calB) 
\leq \prob(\calA) + \prob(\calB \mid \calA^{c}).
\end{equation}
Hence with $\calA = \set{\est \bivsupp \neq \bivsupp}$, $\calB = \set{\est \univsupp \neq \univsupp}$,
we readily arrive at the lower bound on the success probability.}

Lastly, the bound on the total number of queries follows from Lemma \ref{lemma:biv_spam_algo} 
by plugging in the stated bounds on $m_i,n_i$; $i=1,2$.
\end{proof}

\paragraph{Gaussian noise model.} Next, we consider the Gaussian noise model with the noise samples 
i.i.d. Gaussian ($\sim \calN(0,\sigma^2)$) across queries. Similarly to Theorem \ref{thm:main_univ_gauss_noise}, we show 
that if the noise variance $\sigma^2$ is sufficiently small, then 
Algorithm \ref{algo:biv_spam} recovers $\bivsupp$ and  $\univsupp$ for a careful choice of $m_i,n_i$, $i=1,2$.
We again reduce the variance via re-sampling each query either $N_2$ times (during the estimation of $\bivsupp$) 
or $N_1$ times (during the estimation of $\univsupp$), and averaging the values.
\begin{theorem} \label{thm:main_biv_gauss_noise}
For the Gaussian noise model with i.i.d noise samples $\sim \calN(0,\sigma^2)$, say we resample each query 
$N_2$ times (during estimation of $\bivsupp$) or $N_1$ times (during estimation of $\univsupp$), and average the values. 
Consider Algorithm \ref{algo:biv_spam}, where $\csalg_2$ and $\csalg_1$ are instantiated with 
$(P2)$ with $\arbnoisebd_2=2\sigma(1+\varepsilon) n_2/\sqrt{N_2}$ in \eqref{eq:sparserec_prog_ordtwo}, 
and $(P1)$ with $\arbnoisebd_1=\sqrt{2}(1+\varepsilon)\sigma \sqrt{n_1/N_1}$ in \eqref{eq:l1min_quadconstr}, respectively,
for some $\varepsilon \in (0,1)$. Let $\twohashfam$ be a $(d,2)$ hash family. If 
\begin{align*}
N_2 &\geq \left\lfloor \frac{72C_4^2 (1+\varepsilon)^2 \sigma^2}{\pi \idenconst_2^2} \right\rfloor + 1, \ m_2 \geq \sqrt{2}\left(\frac{3L}{\idenconst_2}\right)^{1/\alpha}, \ n_2 \geq \max\set{\tilde{c}_3 \abs{\bivsupp} \log\left(\frac{d^2}{\abs{\bivsupp}}\right), \frac{2\log[(2m_2+1)^2 e \abs{\twohashfam}]}{\tilde c_5 \varepsilon^2}} \nonumber \\ 
N_1 &\geq \left\lfloor \frac{18C_2^2 (1+\varepsilon)^2 \sigma^2}{\idenconst_1^2} \right\rfloor + 1, \ m_1 \geq \left(\frac{3L}{\idenconst_1}\right)^{1/\alpha}, \ \  
n_1 \geq \max\set{\tilde{c}_1  \abs{\univsupp} \log \left(\frac{d-\abs{\est{\bivsuppvar}}}{\abs{\univsupp}}\right), \frac{2\log(2m_1+1)}{c_3 \varepsilon^2}}
\end{align*}
hold, 
then $\est{\bivsupp} = \bivsupp$ and $\est{\univsupp} = \univsupp$  with probability at least 
\begin{equation*}
1 - \exp(-\tilde{c}_4 n_2) - \exp\left(-\frac{\tilde c_5 n_2 \varepsilon^2}{2}\right) - 2\exp(-\tilde{c}_2 n_1) - 2\exp\left(-\frac{c_3 \varepsilon^2 n_1}{2}\right). 
\end{equation*}
The total number of queries made is 
\begin{equation*}
4(2m_2 + 1)^2 n_2 N_2 \abs{\twohashfam} + 2(2m_1 + 1) n_1 N_1 
= \Omega\left(\jvyb{\bar{c}_2'}\abs{\bivsupp} \log\left(\frac{d^2}{\abs{\bivsupp}}\right) \abs{\twohashfam} 
+ \jvyb{\bar{c}_1'}\abs{\univsupp} \log \left(\frac{d-\abs{\est{\bivsuppvar}}}{\abs{\univsupp}}\right)\right).
\end{equation*}
\jvyb{Here $\bar{c}_i' > 1$ depends on $L,\idenconst_i,\alpha,\sigma$ with $\bar{c}_1'$ as in 
Theorem \ref{thm:main_univ_gauss_noise}.}
\end{theorem}
%
\begin{proof}
As in the bounded noise model, we only prove the exact recovery of $\bivsupp$. First, we
note that in \eqref{eq:lin_sys_bivspam}, as a consequence of resampling each query $N_2$ times and averaging,
$\noise_{i,1}, \noise_{i,2}, \noise_{i,3}, \noise_{i,4} \sim \calN(0,\frac{\sigma^2}{N_2})$, $i\in[n_2]$ are independent. 
Therefore $\noise_i \sim \calN(0,\frac{4\sigma^2}{N_2}), i\in[n_2]$ are also independent and we set $\csalg_2 = (P2)$ 
with $\arbnoisebd_2 = 2\sigma(1+\varepsilon)n_2/\sqrt{N_2}$.

From Corollary \ref{cor:gauss_noise_bilin_rec}, we know that there exist constants $\tilde c_3, \tilde c_4, \tilde c_5, C_4 > 0$ 
such that if $n_2 \geq  \tilde{c}_3 \abs{\bivsupp} \log\bigl(d^2/\abs{\bivsupp}\bigr)$, then 
for any given $\hashfn \in \twohashfam$, $\vecx \in \baseset(\hashfn)$, we have for the stated choice of $\arbnoisebd_2$ that 
\begin{equation} \label{eq:mainthm_bivgauss_temp1}
\norm{\est{\origsparsvec}(\vecx;\calA) - \origsparsvec(\vecx;\calA)}_{\infty} \ \leq \ 
\norm{\est{\origsparsvec}(\vecx;\calA) - \origsparsvec(\vecx;\calA)}_{1}\  \leq\, C_4 \sqrt{\frac{2}{\pi}}(1+\varepsilon)\cdot \frac{2\sigma}{\sqrt{N_2}}=:\epsilon_2, 
\end{equation}
with probability at least $1 - \exp(- \tilde c_4 n_2) - e \cdot \exp\left(-\tilde c_5 \varepsilon^2 n_2\right)$. 
This is true since $\origsparsvec(\vecx;\calA)$ is always at most $\abs{\bivsupp}$ sparse. 
Therefore, by taking the union bound, \eqref{eq:mainthm_bivgauss_temp1} holds uniformly for all 
$\hashfn \in \twohashfam$, $\vecx \in \baseset(\hashfn)$, with probability at least 
\begin{align}
\notag&1 - \exp(- \tilde c_4 n_2) - (2m_2 + 1)^2\abs{\twohashfam} e \cdot \exp\left(-\tilde c_5 \varepsilon^2 n_2\right) \\ 
\notag= &1 - \exp(- \tilde c_4 n_2) -  \exp\left(\log[(2m_2 + 1)^2\abs{\twohashfam} e] - \tilde c_5 \varepsilon^2 n_2\right) \\ 
\label{eq:prob_1}\geq &1 - \exp(- \tilde c_4 n_2) - \exp\left(-\frac{\tilde c_5 n_2 \varepsilon^2}{2}\right)
\end{align}
if $n_2 \geq \frac{2\log[(2m_2+1)^2 e \abs{\twohashfam}]}{\tilde c_5 \varepsilon^2}$ holds.
We conclude that $\csalg_2$ is $\epsilon_2$-accurate for each $\hashfn \in \twohashfam$, $\vecx \in \baseset(\hashfn)$
with probability at least \eqref{eq:prob_1}.
 
By the stated choice of $N_2$, we have $\epsilon_2< \frac{\idenconst_2}{3}$
and using Lemma \ref{lemma:biv_spam_algo} and the condition on $m_2$, 
we obtain $\est{\bivsupp} = \bivsupp$. 
\jvyb{The lower bound on the success probability of identifying $\univsupp,\bivsupp$ follows via 
the same argument as in the proof of Theorem \ref{thm:main_biv_arbnoise}.}
Lastly, the bound on the total number of queries 
follows from the expression in Lemma \ref{lemma:biv_spam_algo} (taking the resampling into account) 
by plugging in the stated bounds on $m_i,n_i, N_i$; $i=1,2$. 
\end{proof}

\begin{remark} \label{rem:samp_compl_fin_biv}
\jvyb{As discussed in Section \ref{sec:spam_biv}, we can construct $\twohashfam$ via a simple 
randomized method (in time linear in output size) (eg., \cite[Section 5]{Devore2011}) where 
$\abs{\twohashfam} = O(\log d)$, with probability at least $1-d^{-\Omega(1)}$. Plugging this into 
Theorem \ref{thm:main_biv_arbnoise} leads to a (worst case) sample complexity of}
\begin{equation*}
\jvyb{\Omega\left(\bar{c}_2\abs{\bivsupp} \log(d^2/\abs{\bivsupp}) \log d + \bar{c}_1\abs{\univsupp} \log (d/\abs{\univsupp})\right)}
\end{equation*}
\jvyb{in the setting of arbitrary bounded noise.} 
\end{remark} 
\subsection{Multivariate case} \label{subsec:final_res_multiv}
Finally, we consider the most general multivariate setting of Section \ref{sec:spam_multiv}.
Based on Lemma \ref{lemma:multiv_spam_algo}, we analyze Algorithm \ref{algo:multiv_spam}
recovering the supports $\totsupp_i$; $i \in [\order]$.
The recovery routines $\csalg_i$ are realized by (P1) in \eqref{eq:l1min_quadconstr} for $i=1$ and for each $3 \leq i \leq \order$
and $\csalg_2$ is instantiated with (P2) in \eqref{eq:sparserec_prog_ordtwo}.
This leads to the results below. 
%
%
\paragraph{Bounded noise model.} This is the same noise model as described in the preceding subsections. The following theorem shows that if 
$\noisebd$ is sufficiently small, then Algorithm \ref{algo:multiv_spam} recovers $\totsupp_i$ exactly for all $i \in [\order]$ provided the 
parameters $(m_i,n_i)_{i=1}^{\order}$ are well chosen.
\begin{theorem} \label{thm:main_multiv_arbnoise}
For the bounded noise model with noise uniformly bounded by $\noisebd$, 
consider Algorithm \ref{algo:multiv_spam} with 
\begin{itemize}
\item[(a)] $\csalg_i$ instantiated with $(P1)$ with $\arbnoisebd_i = 2^i \noisebd \sqrt{n_i}$ in \eqref{eq:l1min_quadconstr} for $i \in \set{3,\dots,\order}$,  
\item[(b)] $\csalg_2$ instantiated with $(P2)$ with $\arbnoisebd_2 = 4\noisebd n_2$ in \eqref{eq:sparserec_prog_ordtwo}, 
\item[(c)] $\csalg_1$ instantiated with $(P1)$ with $\arbnoisebd_1 = 2\noisebd \sqrt{n_1}$ in \eqref{eq:l1min_quadconstr}, respectively. 
\end{itemize}
Let $\hashfam_i^{\calP_i}$ be a $(\calP_i,i)$ hash family for each $2\le i \le \order$
and let $\calP_i$ denote the set $\calP \subseteq [d]$ at the beginning of $i^{\rm th}$ iteration (with $\calP_{\order} = [d]$). If 
%
\begin{align*}
\noisebd &< \min\Bigl\{\min_{i \in \{3,\dots,\order\}}\frac{\idenconst_i}{2^i 3 C_2}, \frac{\idenconst_2}{12 C_4},\frac{\idenconst_1}{6 C_2}\Bigr\}, \quad
m_i \geq \biggl(\frac{3L (\sqrt{i})^{\alpha}}{\idenconst_i}\biggr)^{1/\alpha}; \quad i \in [\order], \\ 
n_i &\geq \tilde{c}_6 \abs{\totsupp_i}^2 \log {\abs{\calP_i} \choose i}; \quad i \in \set{3,\dots,\order}, \\
n_2 &\geq \tilde{c}_3 \abs{\bivsupp} \log(\abs{\calP_2}^2/\abs{\bivsupp}) \quad \text{and} \quad n_1 \geq \tilde{c}_1 \abs{\univsupp} \log (\abs{\calP_1}/\abs{\univsupp})
\end{align*}
are satisfied, then 
$\est{\totsupp_i} = \totsupp_i$ for all $i \in [\order]$ with probability at least 
%
\begin{equation*}
1 - \sum_{i = 3}^{\order} \exp\left(-\frac{\tilde{c}_7 n_i}{\abs{\totsupp_i}^2}\right) - \exp(-\tilde{c}_4 n_2) - 2\exp(-\tilde{c}_2 n_1). 
\end{equation*}
Here, the constants $C_4,\tilde{c}_3,\tilde{c}_4 > 0$ are 
from Theorem \ref{thm:arbnoise_bilin_rec}, while $\tilde{c}_1, \tilde{c}_2, C_2 > 0$ are as defined in Theorem \ref{thm:main_univ_arbnoise}. 
The constants $\tilde{c}_6, \tilde{c}_7 > 0$ depend only on the constants $c_6,c_7 > 0$ defined in Theorem \ref{thm:arbnoise_multlin_rec}. 
The total number of queries made is 
\begin{align*}
&\jvyb{\sum_{i=1}^{\order} 2^i (2m_i + 1)^i n_i \abs{\hashfam_i^{\calP_i}}} \\
&= \Omega\left(\sum_{i=3}^{\order}\Bigl[c_i^i \jvyb{i} \abs{\totsupp_i}^2 \log(\abs{\calP_i}) \abs{\hashfam_i^{\calP_i}}\Bigr] 
+ \jvyb{\bar{c}_2}\abs{\bivsupp} \log\left(\frac{\abs{\calP_2}^2}{\abs{\bivsupp}}\right) \abs{\twohashfam} 
+ \jvyb{\bar{c}_1}\abs{\univsupp} \log \left(\frac{\abs{\calP_1}}{\abs{\univsupp}}\right)\right)
\end{align*}
where \jvyb{for $i=3,\dots,\order$,} each $c_i > 1$ depends on $\idenconst_i$,$L,\alpha,\jvyb{i}$, \jvyb{and $\bar{c}_1,\bar{c}_2 > 1$ 
are as in Theorem \ref{thm:main_biv_arbnoise}}.
\end{theorem}
%
%
\begin{proof}
Say we are at the beginning of $i^{\rm th}$ iteration with $3 \leq i \leq \order$ and $\est{\totsupp}_l = \totsupp_l$ holds true for each $l > i$.
Hence, the model has reduced to an order $i$ sparse additive model on the set $\calP_i \subseteq [d]$, with 
$\totsupp_i^{(1)},\totsupp_{i-1}^{(1)},\dots,\totsupp_1 \subset \calP_i$.

From \eqref{eq:lin_sys_multiv_spam}, we see for the noise vector $\noisevec \in \matR^{n_i}$ that 
\begin{equation}\label{eq:multi_noise1}
\noise_s = \sum_{z=1}^{2^i} (-1)^{\digit(z-1)} \noise_{s,z} \quad\text{for all}\ s\in[n_i].
\end{equation}
Since $\abs{\noise_{s,z}} \leq  \noisebd$, this implies $\|\noisevec\|_{\infty} \leq 2^i \noisebd$ and thus 
$\|\noisevec\|_2 \leq 2^i\noisebd\sqrt{n_i}$. This bound holds uniformly for each $\vecx$ at which the linear system 
is formed. So we now instantiate $\csalg_i$ with $(P1)$ in \eqref{eq:l1min_quadconstr}, with $\arbnoisebd_i = 2^i\noisebd\sqrt{n_i}$. 

As a consequence of part $1$ of Theorem \ref{thm:arbnoise_multlin_rec}, there exists constants $\tilde{c}_6,\tilde{c}_7 > 0$ 
depending on $c_6,c_7 > 0$ (as defined in Theorem \ref{thm:arbnoise_multlin_rec}) so that if 
$n_i \geq \tilde{c}_6 \abs{\totsupp_i}^2 \log {\abs{\calP_i} \choose i}$, then with probability at least 
$1-\exp\left(-\frac{\tilde{c}_7 n_i}{\abs{\totsupp_i}^2}\right)$, the 
matrix $\matB \in \matR^{n_i \times {\abs{\calP_i} \choose i}}$ satisfies $\ell_2/\ell_2$ RIP with 
$\riptt_{2\abs{\totsupp_i}} < \sqrt{2}-1$. Conditioning on this event, it follows from Theorem \ref{thm:arbnoise_multlin_rec} that 
%
\begin{equation*}
\norm{\est{\origsparsvec}(\vecx;\calA) - \underbrace{\origsparsvec(\vecx;\calA)}_{\abs{\totsupp_i} \ \text{sparse}}}_{\infty} \leq\ 
\norm{\est{\origsparsvec}(\vecx;\calA) - \origsparsvec(\vecx;\calA)}_{2}\  \leq 2^i\noisebd C_2=:\epsilon_i\ \text{for all}\
\vecx \in\bigcup_{\hashfn \in \hashfam_{i}^{\calP_i}} \baseset(\hashfn). 
\end{equation*}
%
Thus, with probability at least $1-\exp(-\frac{\tilde{c}_7 n_i}{\abs{\totsupp_i}^2})$, $\csalg_i$ is $\epsilon_i$-accurate 
for each $\hashfn \in \hashfam_{i}^{\calP_i}$, $\vecx \in \baseset(\hashfn)$.
The assumption on $\noisebd$ ensures that $\epsilon_i<D_i/3$
and by Lemma \ref{lemma:multiv_spam_algo} it follows that 
the stated choice of $(m_i,\epsilon_i)$ ensures exact recovery of $\totsupp_i$. 

Hence if 
$$
\noisebd < \min_{i \in \set{3,\dots,\order}}\frac{\idenconst_i}{2^i 3 C_2}
$$ 
is satisfied and $m_i,n_i,\epsilon_i$ satisfy their 
stated bounds (for $3 \leq i \leq \order$), 
\jvyb{then we can lower bound the probability that $\est{\totsupp_i} = \totsupp_i$ 
holds for all $i=3,\dots,\order$ via the following simple generalization of \eqref{eq:useful_union_bd_biv}. 
For events $\calA_3,\dots,\calA_{\order}$ it holds that
\begin{equation} \label{eq:useful_union_bd_multiv}
 \prob(\cup_{i=3}^{\order} \calA_i) 
\leq \prob(\calA_{\order}) + \prob(\calA_{\order-1} \mid \calA_{\order}^c) 
+ \prob(\calA_{\order-2} \mid \calA_{\order}^c \cap \calA^{c}_{\order-1}) 
+ \cdots + \prob(\calA_{3} \mid \cap_{i=4}^{\order} \calA_{i}^c).
\end{equation}
Therefore plugging $\calA_i = \set{\est{\totsupp_i} \neq \totsupp_i}$ in \eqref{eq:useful_union_bd_multiv}, 
we readily have that}
$\est{\totsupp_i} = \totsupp_i$ holds for all $3 \leq i \leq \order$ with probability at least 
\begin{equation*}
1 - \sum_{i = 3}^{\order} \exp\left(-\frac{\tilde{c}_7 n_i}{\abs{\totsupp_i}^2}\right).
\end{equation*}
Once $\totsupp_i$ are identified exactly for all $3 \leq i \leq \order$, 
we are left with a bivariate SPAM on the set $\calP_2$, with $\univsupp, \totsupp_2^{(1)} \subset \calP_2$.
Therefore by invoking Theorem \ref{thm:main_biv_arbnoise}, we see that if furthermore
$\noisebd < \min\set{\frac{\idenconst_2}{12 C_4},\frac{\idenconst_1}{6 C_2}}$ holds,
and $m_1,n_1,m_2,n_2$ satisfy their respective stated conditions, then the stated instantiations of $\csalg_1,\csalg_2$ (along with the 
stated choices of $\arbnoisebd_1,\arbnoisebd_2$) ensures $\est{\univsupp} = \univsupp$ and $\est{\bivsupp} = \bivsupp$ with probability at least 
$1 - \exp(-\tilde{c}_4 n_2) - 2\exp(-\tilde{c}_2 n_1)$. 
\jvyb{The stated lower bound on the success probability of identifying each $\totsupp_i$ ($i=1,\dots,\order$) 
follows readily from \eqref{eq:useful_union_bd_biv} by plugging 
$\calA = \cap_{i=3}^{\order} \set{\est{\totsupp_i} \neq \totsupp_i}$ and 
$\calB = \cap_{i=1}^{2} \set{\est{\totsupp_i} \neq \totsupp_i}$. }
This completes the proof for exact recovery of $\totsupp_i$'s for all $i \in [\order]$.

Finally, the stated sample complexity bound for the total number of queries made by Algorithm \ref{algo:multiv_spam} 
follows in a straightforward manner by plugging in  
\begin{equation*}
m_i = \Omega\left(\jvyb{3^{1/\alpha} \sqrt{i}} L^{1/\alpha}D_i^{-1/\alpha} \right), \quad i \in [\order]; \quad n_i = \Omega(i \abs{\totsupp_i}^2 \log \abs{\calP_i}), \quad \ 3 \leq i \leq \order
\end{equation*}
along with the complexity bounds for $n_1,n_2$ into the expression for total number of samples from Lemma \ref{lemma:multiv_spam_algo}.
This completes the proof.
\end{proof}

\paragraph{Gaussian noise model.} In the Gaussian noise model with noise samples 
i.i.d. Gaussian ($\sim \calN(0,\sigma^2)$) across queries, we again
reduce the variance via re-sampling each query $N_i$ times (during the estimation of $\totsupp_i$)
for every $i\in[\order]$ and averaging the values.
We show that if the noise variance $\sigma^2$ is sufficiently small, then 
Algorithm \ref{algo:multiv_spam} recovers $\totsupp_i$ exactly for each $i \in [\order]$, 
provided the parameters $m_i,n_i$; $i \in [\order]$, are well chosen.
\begin{theorem} \label{thm:main_multiv_gauss_noise}
For the Gaussian noise model with i.i.d. noise samples $\sim \calN(0,\sigma^2)$, consider Algorithm \ref{algo:multiv_spam} wherein 
we resample each query $N_i$ times during estimation of $\totsupp_i$ and average the values. 
Let\begin{itemize}
\item[(a)] $\csalg_i = (P1)$ with $\arbnoisebd_i = 2^{i/2}(1+\varepsilon)\sigma \sqrt{n_i/N_i}$ in \eqref{eq:l1min_quadconstr} for $3 \leq i \leq \order$,
\item[(b)] $\csalg_2 = (P2)$ with $\arbnoisebd_2 = 2(1+\varepsilon)\sigma n_2/\sqrt{N_2}$ in \eqref{eq:sparserec_prog_ordtwo} and
\item[(c)] $\csalg_1 = (P1)$ with $\arbnoisebd_1 = \sqrt{2}(1+\varepsilon)\sigma \sqrt{n_1/N_1}$ in \eqref{eq:l1min_quadconstr}, respectively,
\end{itemize}
for some $\varepsilon \in (0,1)$.
Let $\hashfam_i^{\calP_i}$ be a $(\calP_i,i)$ hash family for each $2\le i \le \order$, and 
denote $\calP_i$ to be the set $\calP \subseteq [d]$ at the beginning of $i^{\rm th}$ iteration with $\calP_{\order} = [d]$. If
\begin{align*}
N_i &\geq \left\lfloor \frac{9 C_2^2 (1+\varepsilon)^2 2^i \sigma^2}{\idenconst_i^2} \right\rfloor + 1, \quad 
n_i \geq \max\set{\tilde{c}_6 \abs{\totsupp_i}^2 \log {\abs{\calP_i} \choose i}, \frac{2\log[(2m_i+1)^i \abs{\hashfam_i^{\calP_i}}]}{c_3 \varepsilon^2}}; 
\quad 3 \leq i \leq \order \\
m_i &\geq \left(\frac{3L (\sqrt{i})^{\alpha}}{\idenconst_i}\right)^{1/\alpha}; \quad i \in [\order], \\
N_2 &\geq \left\lfloor \frac{72C_4^2 (1+\varepsilon)^2 \sigma^2}{\pi \idenconst_2^2} \right\rfloor + 1, \ 
n_2 \geq \max\set{\tilde{c}_3 \abs{\bivsupp} \log\left(\frac{\abs{\calP_2}^2}{\abs{\bivsupp}}\right), \frac{2\log[(2m_2+1)^2 e \abs{\hashfam_2^{\calP_2}}]}{\tilde c_5 \varepsilon^2}}, \\ 
N_1 &\geq \left\lfloor \frac{18C_2^2 (1+\varepsilon)^2 \sigma^2}{\idenconst_1^2} \right\rfloor + 1, \  
n_1 \geq \max\set{\tilde{c}_1  \abs{\univsupp} \log \left(\frac{\abs{\calP_1}}{\abs{\univsupp}}\right), \frac{2\log(2m_1+1)}{c_3 \varepsilon^2}}
\end{align*}
hold, then 
$\est{\totsupp_i} = \totsupp_i$ for all $i \in [\order]$ with probability at least 
%
\begin{align}
1 &- \sum_{i = 3}^{\order} \Bigl[\exp\left(-\frac{\tilde{c}_7 n_i}{\abs{\totsupp_i}^2}\right) + 2\exp\left(-\frac{c_3\varepsilon^2 n_i}{2}\right)\Bigr] \nonumber\\
&- \exp(-\tilde{c}_4 n_2) - \exp\left(-\frac{\tilde c_5 n_2 \varepsilon^2}{2}\right) - 2\exp(-\tilde{c}_2 n_1) - 2\exp\left(-\frac{c_3 \varepsilon^2 n_1}{2}\right). 
\end{align}
The constants $C_4,\tilde{c}_3,\tilde{c}_4, \tilde{c}_1, \tilde{c}_2, C_2, \tilde{c}_6, \tilde{c}_7 > 0$ are 
as explained in Theorem \ref{thm:main_multiv_arbnoise}, while $\tilde c_5, c_3 > 0$ come from Corollaries 
\ref{cor:gauss_noise_bilin_rec}, \ref{cor:gauss_noise_lin_rec} respectively. The total number of queries made is
\begin{align*}
&\sum_{i=1}^{\order} N_i 2^i (2m_i + 1)^i n_i \abs{\hashfam_i^{\calP_i}} \\
&= 
\Omega\left(\sum_{i=3}^{\order}\Bigl[\jvyb{\bar{c}_i' \bar{c}_i^i i} \abs{\totsupp_i}^2 \log(\abs{\calP_i}) \abs{\hashfam_i^{\calP_i}}\Bigr] 
+ \jvyb{\bar{c}_2'}\abs{\bivsupp} \log\left(\frac{\abs{\calP_2}^2}{\abs{\bivsupp}}\right) \abs{\hashfam_2^{\calP_2}} 
+ \jvyb{\bar{c}_1'}\abs{\univsupp} \log \left(\frac{\abs{\calP_1}}{\abs{\univsupp}}\right)\right)
\end{align*}
where \jvyb{for $i=3,\dots,\order$, $\bar{c}_i' > 1$ depends on $\sigma,\idenconst_i$, and $\bar{c}_i > 1$ depends 
on $L,\idenconst_i,\alpha,i$. Moreover, $\bar{c}_1',\bar{c}_2' > 1$ are as in 
Theorem \ref{thm:main_biv_gauss_noise}}.
\end{theorem}
\begin{proof}
Again, in the beginning of $i^{\rm th}$ iteration $3 \leq i \leq \order$ with $\est{\totsupp}_l = \totsupp_l$ for each $l > i$,
the model has reduced to an order $i$ sparse additive model on the set $\calP_i \subseteq [d]$, with 
$\totsupp_i^{(1)},\totsupp_{i-1}^{(1)},\dots,\totsupp_1 \subset \calP_i$. 

The noise vector is again given by \eqref{eq:multi_noise1}
As a consequence of resampling each query point $N_i$ times and averaging, we get
$\noise_{s,z} \sim \calN(0,\frac{\sigma^2}{N_i})$ i.i.d. for all $s,z$ and, therefore, $\noise_s \sim \calN(0,\frac{2^i\sigma^2}{N_i})$ i.i.d.
for each $s$. From part $1$ of Theorem \ref{thm:arbnoise_multlin_rec}, we know that if 
$n_i \geq \tilde{c}_6 \abs{\totsupp_i}^2 \log {\abs{\calP_i} \choose i}$, then with probability at least 
$1-\exp\left(-\frac{\tilde{c}_7 n_i}{\abs{\totsupp_i}^2}\right)$ (with $\tilde{c}_6,\tilde{c}_7$ depending only on $c_6,c_7$), 
the matrix $\matB \in \matR^{n_i \times {\abs{\calP_i} \choose i}}$ satisfies $\ell_2/\ell_2$ RIP with 
$\riptt_{2\abs{\totsupp_i}} < \sqrt{2}-1$. Let us condition on this event. 
Then by setting $\csalg_i = (P1)$ with $\arbnoisebd_i = (1+\varepsilon)\left(\frac{2^{i/2}\sigma \sqrt{n_i}}{\sqrt{N_i}}\right)$, 
and invoking Corollary \ref{cor:gauss_noise_lin_rec}, it follows for any given $\hashfn \in \hashfam_{i}^{\calP_i}$, $\vecx \in \baseset(\hashfn)$ 
that 
\begin{equation} \label{eq:multiv_tog_gauss_1}
\norm{\est{\origsparsvec}(\vecx;\calA) - \underbrace{\origsparsvec(\vecx;\calA)}_{\abs{\totsupp_i} \ \text{sparse}}}_{\infty} \ \leq \ 
\norm{\est{\origsparsvec}(\vecx;\calA) - \origsparsvec(\vecx;\calA)}_{2}\ \leq\, 2^{i/2}C_2(1+\varepsilon)\sigma \sqrt{n_i/N_i}=:\epsilon_i
\end{equation} 
with probability at least $1 - 2\exp(-c_3 \varepsilon^2 n_i)$. By the union bound, it follows that 
\eqref{eq:multiv_tog_gauss_1} holds for all $\hashfn \in \hashfam_{i}^{\calP_i}$, $\vecx \in \baseset(\hashfn)$, with probability at least
\begin{align*}
&1 - 2(2m_i+1)^i\abs{\hashfam_{i}^{\calP_i}}\exp(-c_3 \varepsilon^2 n_i) = 1 - 2\exp[\log[(2m_i+1)^i\abs{\hashfam_{i}^{\calP_i}}] - c_3 \varepsilon^2 n_i] \\
&\qquad\geq 1 - 2\exp\left(-\frac{c_3 \varepsilon^2 n_i}{2}\right)
\end{align*}
if $n_i \geq \frac{2\log[(2m_i+1)^i|\hashfam_i^{\calP_i}|]}{c_3 \varepsilon^2}$ holds. This gives us the stated condition on $n_i$ for $3\leq i \leq \order$. 
Hence for the aforementioned choice of $n_i$, $\csalg_i$ is $\epsilon_i$-accurate for each 
$\hashfn \in \hashfam_{i}^{\calP_i}$, $\vecx \in \baseset(\hashfn)$, with probability at least 
$1 - 2\exp\left(-\frac{c_3 \varepsilon^2 n_i}{2}\right) - \exp\left(-\frac{\tilde{c}_7 n_i}{\abs{\totsupp_i}^2}\right)$.
By the condition on $N_i$, we obtain $\epsilon_i<D_i/3$ and from Lemma \ref{lemma:multiv_spam_algo}, it follows that
%
%
for the stated choice of $m_i$ and $\epsilon_i$, we have $\est{\totsupp_i} = \totsupp_i$. 
Thus we conclude by the union bound \jvyb{in \eqref{eq:useful_union_bd_multiv}} that $\est{\totsupp_i} = \totsupp_i$ 
holds for all $3\leq i \leq \order$, \jvyb{with high probability}, for the stated choice of $m_i,n_i,\epsilon_i,N_i$.

Finally, say $\est{\totsupp_i} = \totsupp_i$ holds for all $3\leq i \leq \order$. Then we are 
left with a bivariate SPAM on the set $\calP_2$, with $\univsupp,\totsupp_2^{(1)} \subset \calP_2$.  
Thereafter, we only need to invoke Theorem \ref{thm:main_biv_gauss_noise}, which guarantees that  
$\est{\totsupp_2} = \totsupp_2$ and $\est{\totsupp_1} = \totsupp_1$ \jvyb{holds with high probability} 
for the stated choices of $\csalg_i,\arbnoisebd_i,m_i,n_i,\epsilon_i,N_i$; $i=1,2$. 
\jvyb{The lower bound on the success probability of identifying each 
$\totsupp_i$ ($i=1,\dots,\order$) then follows via the same argument as in the proof of 
Theorem \ref{thm:main_multiv_arbnoise}.}
This completes the proof for the exact recovery of $\totsupp_i$'s. 

Since each query is resampled $N_i$ times, with $N_i = \jvyb{1 + \Omega(2^i \sigma^2/\idenconst_i^2)}$, we obtain the stated sample complexity bound 
by proceeding as in the proof of Theorem \ref{thm:main_multiv_arbnoise}. 
\end{proof}

\begin{remark} \label{rem:samp_compl_fin_multiv}
\jvyb{As discussed in Section \ref{sec:spam_biv}, we can construct $\thashfam$ (for $t \geq 2$) via a simple 
randomized method (in time linear in output size) (eg., \cite[Section 5]{Devore2011}) where 
$\abs{\thashfam} = O(t e^t \log d)$, with probability at least $1-d^{-\Omega(t)}$. Plugging this into 
Theorem \ref{thm:main_multiv_arbnoise} leads to a (worst case) sample complexity of
\begin{equation*}
\Omega\left(\sum_{i=3}^{\order}\Bigl[c_i^i e^i i^{2} \abs{\totsupp_i}^2 (\log(\abs{\calP_i}))^2 \Bigr] 
+ \bar{c}_2 \abs{\bivsupp} \log\left(\frac{\abs{\calP_2}^2}{\abs{\bivsupp}}\right) \log(\abs{\calP_2}) 
+ \bar{c}_1 \abs{\univsupp} \log \left(\frac{\abs{\calP_1}}{\abs{\univsupp}}\right)\right)
\end{equation*}
in the setting of arbitrary bounded noise. Since $\abs{\calP_i} = O(d)$, this leads to the expression 
in \eqref{eq:main_sam_compl_res}.}
\end{remark} 
\section{Discussion} \label{sec:disc_conc_rems}
%
%
%
\jvyb{We start by providing a brief summary of our results with remarks concerning certain aspects of our algorithm.}
We then discuss how the components $\phi$ can be identified, and also compare our results with closely related work. 
Finally, we discuss an alternative approach described in \cite{Tyagi_spamint_long16} in more detail. 
%
\jvyb{\paragraph{Summary of our results.} Let us recall that in the setting where the 
queries are corrupted with arbitrary bounded noise, Algorithm \ref{algo:multiv_spam} 
succeeds with high probability in identifying each $\totsupp_i$ for $i=1,\dots,\order$ provided the 
noise level is sufficiently small, and makes 
\begin{equation*} 
\Omega\Biggl(\sum_{i=3}^{\order}\Bigl[\underbrace{c_i^i \jvyb{i^{2}} \abs{\totsupp_i}^2 \log^2 d}_{\text{Identifying } \totsupp_i} \Bigr] 
+ \underbrace{c_2 \abs{\bivsupp} \log\left(\frac{d^2}{\abs{\bivsupp}}\right) \log d}_{\text{Identifying } \totsupp_2}
+  \underbrace{c_1 \abs{\univsupp} \log \left(\frac{d}{\abs{\univsupp}}\right)}_{\text{Identifying } \totsupp_1}\Biggr)
\end{equation*} 
queries. This is the same expression in \eqref{eq:main_sam_compl_res} and is a consequence of 
Theorem \ref{thm:main_multiv_arbnoise} (see Remark \ref{rem:samp_compl_fin_multiv}). 
For each $3 \leq i \leq \order$, the term $c_i^i \jvyb{i^{2}} \abs{\totsupp_i}^2 \log^2 d$ represents the sample complexity 
of identifying $\totsupp_i$. In particular, $c_i^i$ 
arises from the size of the $i$-dimensional grid $\baseset(\hashfn)$ for a given hash function 
$\hashfn$ in Algorithm \ref{algo:multiv_spam}, with $c_i$ depending on the smoothness parameters $L,\alpha,D_i$
and scaling as $\sqrt{i}$ with $i$.
The term $i \abs{\totsupp_i}^2 \log d = \Theta(\abs{\totsupp_i}^2 \log{d\choose i})$ 
arises from the sample complexity of estimating a $i^{th}$ order sparse multilinear function in $d$ variables (with $\abs{\totsupp_i}$ 
terms) and follows from Theorem \ref{thm:arbnoise_multlin_rec}. Finally, the term $i e^i \log d$ arises from 
the size of a $(d,i)$ hash family (see Remark \ref{rem:samp_compl_fin_multiv}) where the $e^i$ factor is subsumed by 
$c_i$. The term $c_2 \abs{\bivsupp} \log\left(\frac{d^2}{\abs{\bivsupp}}\right) \log d$ is the sample complexity 
for identification of $\bivsupp$. Here, $c_2$ arises from the size of a two-dimensional grid in $[-1,1]^2$; 
$\abs{\bivsupp} \log\left(\frac{d^2}{\abs{\bivsupp}}\right)$ is the sample complexity of estimating a sparse 
bilinear function (see Theorem \ref{thm:arbnoise_bilin_rec}), and $\log d$ arises from the size of a $(d,2)$ hash 
family (see Remark \ref{rem:samp_compl_fin_biv}). Finally, 
the term $c_1 \abs{\univsupp} \log \left(\frac{d}{\abs{\univsupp}}\right)$ 
is the sample complexity for identification of $\univsupp$ 
where $c_1$ arises from the size of a grid in $[-1,1]$ and  
$\abs{\univsupp} \log \left(\frac{d}{\abs{\univsupp}}\right)$ is the sample complexity 
of estimating a sparse linear function (see Theorem \ref{thm:arbnoise_lin_rec}).} 

\jvyb{The following points are worth noting for our algorithms.}
\begin{itemize}
\item \jvyb{In general, we do not need to know the exact values
of the smoothness parameters $\alpha, L, D_i$ for $i = 1,2,...,\order$ . It suffices to use 
lower bounds for $\alpha,D_i$ and an upper bound for $L$. Similarly, it is also possible to work with just
the upper bound estimates of $\order$ and $\abs{\univsupp},\dots,\abs{\totsupp_{\order}}$.}

\item \jvyb{The computational cost of Algorithm \ref{algo:multiv_spam} is typically dominated by $\csalg$. 
At iteration $i = \order$, for a given hash function $\hashfn$ and base point $\vecx \in \baseset(\hashfn)$,  
the computational complexity of $\csalg$ is at most polynomial in $(n_{\order},d^{\order})$ 
(recall Remark \ref{rem:comp_cost_sparse_rec}). Since there are $O(m_{\order}^{\order} \abs{\mathcal{H}_{\order}^{d}})$ 
base points, the total cost (over $i = 1,\dots,\order$) is at most polynomial in the number of queries and $d^{\order}$. }
\end{itemize}

%
%
%
%
\paragraph{Identification of $\phi_{\vecj}$'s.} 
Once the sets $\univsupp, \bivsupp,\dots,\totsupp_{\order}$ are known, then one can identify each component in the representation 
\eqref{eq:mod_unique_multivar} by querying $f$ along the corresponding canonical subspaces. Indeed, for a given $1 \leq p \leq \order$, 
and $1 \leq r \leq p$, let us see how we can identify the $r$-variate component $\phi_{\vecj}$ for a given $\vecj \in \totsupp_{p}^{(r)}$. 
Consider any $\vecx = (x_1 \dots x_d)^T \in [-1,1]^d$ which is supported on $\totsupp_{p}^{(1)}$, i.e., $x_l = 0$ if $l \not\in \totsupp_{p}^{(1)}$. 
We then have from \eqref{eq:mod_unique_multivar} that
\begin{align*} 
f(\vecx) = \modmean + \sum_{u \in \totsupp_{p}^{(1)}} \phi_{u}(x_{u}) 
+ \sum_{\vecu \in \totsupp_{p}^{(2)}} \phi_{\vecu}(x_{\vecu}) \quad
+ \dots 
+ \sum_{\vecu \in \totsupp_{p}^{(p-1)}} \phi_{\vecu}(x_{\vecu}) 
+ \sum_{\vecu \in \totsupp_{p}} \phi_{\vecu}(x_{\vecu}).
\end{align*}
In particular, it follows from \eqref{eq:anova_comp_exp} that
\begin{align} \label{eq:discuss_obtain_phij}
\sum_{\veci \subseteq \vecj}(-1)^{|\vecj|-|\veci|} f(\setproj_{\veci}(\vecx)) = \phi_{\vecj}(\vecx_{\vecj}),
\end{align}
with $\setproj_{\veci}(\vecx)$ denoting the projection of $\vecx$ on the set of variables $\veci$. 
Hence by querying $f$ at the $2^{\abs{\vecj}}$ points $\set{\setproj_{\veci}(\vecx): \veci \subseteq \vecj}$, we can 
obtain the sample $\phi_{\vecj}(\vecx_{\vecj})$ via \eqref{eq:discuss_obtain_phij}. Consequently, by choosing 
$\vecx_j$ from a regular grid on $[-1,1]^{\vecj}$, we can estimate $\phi_{\vecj}$ from its corresponding samples 
via standard quasi interpolants (in the noiseless case) or tools from non parametric regression (in the noisy case).

\paragraph{SPAMs.} To begin with, we note that our work generalizes the recent results of Tyagi et al. \cite{Tyagi_aistats16, Tyagi_spamint_long16}
in two fundamental ways. Firstly, we provide an algorithm for the general case where $\order \geq 2$ is possible, the results in 
\cite{Tyagi_aistats16, Tyagi_spamint_long16} are for the case $\order = 2$. Secondly, our results only require $f$ to be H\"older continuous 
and not continuously differentiable as is the case in \cite{Tyagi_aistats16, Tyagi_spamint_long16}. We also mention that our sampling bounds 
for the case $\order = 2$ are linear in the sparsity $\abs{\univsupp} + \abs{\bivsupp}$ even when the noise samples are i.i.d. Gaussian. 
However the algorithms in \cite{Tyagi_aistats16, Tyagi_spamint_long16} have super-linear dependence on the sparsity in this noise model.
This is unavoidable in \cite{Tyagi_aistats16, Tyagi_spamint_long16} due to the localized nature of the sampling scheme, 
wherein finite difference operations are used to obtain linear measurements of the sparse gradient and Hessian of $f$. 
In the presence of noise, this essentially leads to the noise level getting scaled up by the step size parameter, 
and thus reducing the variance of noise necessarily leads to a resampling factor which is super linear in sparsity.  

Dalalyan et al. \cite{Dala2014} recently studied models of the form \eqref{eq:gen_spam_eq} with a 
\jvyb{set $\totsupp_{r_0}$ of $r_0$-wise interaction terms.} 
They considered the Gaussian white noise model, which while not 
the same as the usual regression setup, is known to be asymptotically equivalent to the same.
They derived non-asymptotic $L_2$ error rates in expectation for an estimator, with $f$ lying in a Sobolev 
space, and showed the rates to be minimax optimal. However, they do not consider the problem of identification 
of the interaction terms. Moreover, as noted in \cite{Dala2014}, the computational cost of their method
typically scales exponentially in 
\jvyb{$\abs{\totsupp_{r_0}}, r_0, d.$} Yang et al. \cite{Yang2015} also studied models of the 
form \eqref{eq:gen_spam_eq} in a Bayesian setting, wherein they place a Gaussian  prior (GP) on $f$ and 
can carry out inference via the posterior probability distribution. They derive an estimator
and provide error rates in the empirical $L_2$ norm for H\"older smooth $f$, but do not address the 
problem of identifying the interaction terms. 
%
\paragraph{Functions with few active variables.} There has been a fair amount of work in the literature on functions which intrinsically depend on a small subset of $k \ll d$ variables \cite{Devore2011, karin2011, Comming2011, Comming2012}. To our knowledge, this model was first considered in \cite{Devore2011}, and in fact, our idea of using a family of hash functions is essentially motivated from \cite{Devore2011} wherein such a family was used to construct the query points.  A prototypical result in \cite{Devore2011} is an algorithm that identifies the set of active variables with \jvyb{$(m+1)^k \abs{\calH_{k}^d} + k \log d$ queries, with $m > 0$} 
being the number of points on a uniform grid along a coordinate. 
\jvyb{The randomized construction of $\calH_{k}^d$ yields $\abs{\calH_k^d} = O(k e^k \log d)$ (see \cite[Section 5]{Devore2011}) 
which results in a sample complexity of $m^k e^k k \log d$.} The exponential dependence on $k$ is unavoidable 
in the worst case, and indeed our bounds are also exponential in $\order$ (see \eqref{eq:main_sam_compl_res}). 
\begin{itemize}
\item When $\order \geq k$, \eqref{eq:gen_spam_eq} is clearly a generalization of this model. 

\item In general, the model \eqref{eq:gen_spam_eq} is also a function of few active variables (those part of $\totsupp_i$'s); more precisely, at most 
$\sum_{i=1}^{\order} i \abs{\totsupp_i}$ variables. However, using a method that is designed generically for learning intrinsically $k$-variate functions would typically have sample complexity scaling exponentially with $\sum_{i=1}^{\order} i \abs{\totsupp_i}$. This is clearly suboptimal; our bounds in general depend at most polynomially on the size of $\totsupp_i$'s. This dependence is actually linear for the case $\order = 2$.
\end{itemize}
\paragraph{An alternative approach.}
Next, we discuss an alternative approach for learning the model \eqref{eq:mod_unique_multivar} which 
was mentioned already in \cite{Tyagi_spamint_long16}. It is based on a truncated expansion in a bounded orthonormal system.
For simplicity, we assume that $f$ takes the form
\begin{equation}\label{eq:mod_multiv_simple}
f(\vecx)=\sum_{\vecj\in\totsupp_{\order}}\phi_{\vecj}(\vecx_{\vecj}),\quad \vecx\in[-1,1]^d.
\end{equation}

Let $\set{\psi_k}_{k \in \mathbb{Z}}$ be an orthonormal basis in $L_2([-1,1])$ with respect to the normalized Lebesgue measure on $[-1,1]$.
We assume that $\psi_0 \equiv 1$
and we define $\set{\psi_{\veci}}_{\veci \in \mathbb{Z}^d}$ to be the tensor product orthonormal basis in $L_2([-1,1]^d)$, where 
\begin{align*}
\psi_{\veci}(\vecx) = \bigotimes_{l=1}^d \psi_{i_l} (x_l),\quad \vecx\in[-1,1]^d.
\end{align*}
For the components of \eqref{eq:mod_multiv_simple} we obtain (for each $\vecj \in \totsupp_{\order}$) the decomposition
\begin{equation} \label{eq:orth_basis_rep_temp1}
\phi_{\vecj}(\vecx_{\vecj}) = \sum_{\veci \in \matZ^d} a_{\vecj, \veci} \psi_{\veci}(\vecx_{\vecj})=
\sum_{\veci \in \matZ^d:{\operatorname{supp\,}}\veci\subseteq \vecj} a_{\vecj, \veci} \psi_{\veci}(\vecx_{\vecj}).
\end{equation}
In the last identity, we used that $\int_{-1}^1\psi_{k}(t)dt=0$ for every $k\in\matZ\setminus\{0\}$ and, therefore,
$a_{\vecj,\veci}=\langle \phi_{\vecj},\psi_{\veci}\rangle=0$ if $i_l\not=0$ for some $l\not\in\vecj.$

Using the smoothness of $\phi_{\vecj}$, we can truncate \eqref{eq:orth_basis_rep_temp1} at level $N\in\matN$ (which we will determine later) and obtain
\begin{equation}\label{eq:orth_basis_rep_temp3}
\phi_{\vecj}(\vecx_{\vecj}) = \sum_{\substack{\veci \in \matZ^d:\|\veci\|_\infty\le N\\{\operatorname{supp\,}}\veci\subseteq \vecj}} a_{\vecj, \veci} \psi_{\veci}(\vecx_{\vecj})
+r_{\vecj}(\vecx_j)\quad\text{for all}\quad 
\vecj 
\in \totsupp_{\order}.
\end{equation}
Summing up over $\vecj\in \totsupp_{\order}$ we arrive at
\begin{equation}\label{eq:orth_basis_rep_temp4}
f(\vecx)=\sum_{\vecj\in{\totsupp}_{\order}}
\biggl(\sum_{\substack{\veci \in \matZ^d:\|\veci\|_\infty\le N\\{\operatorname{supp\,}}\veci\subseteq \vecj}} a_{\vecj, \veci} \psi_{\veci}(\vecx_{\vecj})\biggr)
+r(\vecx)\quad\text{with}\quad r(\vecx)=\sum_{\vecj\in{\totsupp}_{\order}}r_{\vecj}(\vecx_j).
\end{equation}
The worst-case error of the uniform approximation of H\"older continuous functions  with exponent $\alpha>0$ 
(i.e., functions from the unit ball of $C^{\alpha}$) is bounded
from below by the Kolmogorov numbers of the embedding of $C^{\alpha}$ into $L_{\infty}$, cf. \cite{Vyb08},
$$
\|r_{\vecj}(\vecx_\vecj)\|_\infty \approx [(2N)^{\order}]^{-\alpha/\order}=(2N)^{-\alpha}\quad\text{and}\quad \|r(\vecx)\|_\infty\approx |{\totsupp}_{\order}|(2N)^{-\alpha}.
$$
Thus for any $\epsilon \in (0,1)$, we typically require $N \gtrsim \Bigl(\frac{\abs{\totsupp_{\order}}}{\epsilon}\Bigr)^{1/\alpha}$ to 
ensure $\|r(\vecx)\|_\infty \lesssim \epsilon$.

Furthermore, the number of degrees of freedom \jvyb{$D$ in \eqref{eq:orth_basis_rep_temp4} is lower bounded} by
$$
D\ge {d\choose \order}(2N)^{\order}.
$$
If the basis functions $\psi_{\veci}$ are uniformly bounded (i.e. they form a Bounded Orthonormal System (BOS)),
it was shown in \cite[Theorem 12.31]{Foucart13} or \cite[Theorem 4.4]{RauhutStruct2010}, that one can recover
$\veca=(a_{\vecj,\veci})$ in \eqref{eq:orth_basis_rep_temp4} from $m\gtrsim |{\totsupp}_{\order}|(2N)^{\order}\log^4(D)$ random samples
of $f$ by $\ell_1$-minimization. Plugging in our estimates on $D$ and $N$, we arrive at $m\gtrsim |{\totsupp}_{\order}|^{1+\order/\alpha}\cdot \log^4(d)$.
This bound is always superlinear in $|{\totsupp}_{\order}|$ and (if $\alpha\in(0,1]$) with the power of dependence at least $1+\order.$

\bibliographystyle{plain}
\bibliography{SPAMgen_main}
\appendix
\section{Proofs for Section \ref{sec:problem}} \label{sec:app_proofs_problem}

\begin{proof}[Proof of Proposition \ref{lem:anova}] For $U\subseteq[d]$, we define
\begin{equation*}
(P_Uf)(\vecx_U)=f\Bigl(\sum_{j\in U}x_j\vece_j\Bigr)\quad \text{and}\quad f_U(\vecx_U)=\sum_{V\subseteq U}(-1)^{|U|-|V|}(P_V f)(\vecx_V).
\end{equation*}
Here, $\{\vece_1,\dots,\vece_d\}$ is the canonical basis of $\matR^d$.

By its definition, $f_U$ is a continuous function. Furthermore,
\begin{align*}
\sum_{U\subseteq[d]}f_U(\vecx_U)&=\sum_{U\subseteq[d]}\sum_{V\subseteq U}(-1)^{|U|-|V|}(P_V f)(\vecx_V)
=\sum_{V\subseteq [d]}(P_V f)(\vecx_V)\sum_{U\supseteq V}(-1)^{|U|-|V|}\\
&=\sum_{V\subseteq [d]}(P_V f)(\vecx_V)\sum_{W\subseteq [d]\setminus V}(-1)^{|W|}.
\end{align*}
\jvyb{The last sum can be rewritten as
$$
\sum_{W\subseteq [d]\setminus V}(-1)^{|W|}=\sum_{k=0}^{d-|V|}(-1)^k{d-|V|\choose k},
$$
which is equal to $(1-1)^{d-|V|}=0$ for all $V\subseteq [d]$ except $V=[d]$. This leads to \eqref{eq:anova1}.}

If $x_{j}=0$ for some $j\in U$, we get
\begin{align*}
f_U(\vecx_U)&=\sum_{V\subseteq U}(-1)^{|U|-|V|}(P_V f)(\vecx_V)\\
&=\sum_{V\subseteq U\setminus\{j\}}\Bigl[(-1)^{|U|-|V|}(P_V f)(\vecx_V)+(-1)^{|U|-|V\cup\{j\}|}(P_{V\cup\{j\}} f)(\vecx_{V\cup\{j\}})\Bigr]=0
\end{align*}
as all the terms in the last sum are equal to zero.

Finally, the uniqueness follows by induction. The statement is obvious for $d=1$. Let now $d>1$
and let us assume that a given function $f\in C([-1,1]^d)$ allows a decomposition
$$
f(\vecx)=\sum_{U\subseteq[d]}f_U(\vecx_U),\quad \vecx\in[-1,1]^d,
$$
which satisfies the properties a)-c) of Proposition \ref{lem:anova}.

Let $W$ be a proper subset of $[d]$ and put for $\vecx_W\in[-1,1]^{|W|}$
\begin{align*}
g_W(\vecx_W)=f\Bigl(\sum_{j\in W}x_j\vece_j\Bigr).
\end{align*}
Then $g_W\in C([-1,1]^{|W|})$ and using c) of Proposition \ref{lem:anova} we obtain

$$
g_W(\vecx_W)=\sum_{U\subseteq[d]}f_U\Bigl(\Bigl(\sum_{j\in W}x_j\vece_j\Bigr)_U\Bigr)
=\sum_{U\subseteq[d]}f_U\Bigl(\sum_{j\in W}x_j(\vece_j)_U\Bigr)=\sum_{U\subseteq W}f_U(\vecx_U).
$$
This decomposition of $g_W$ satisfies a)-c) of Proposition \ref{lem:anova} with $|W|\le d-1$. By the induction assumption,
this decomposition is therefore unique. We conclude, that
$f_U$ is uniquely determined for all $U\subset [d]$. Finally, also
$$
f_{[d]}(\vecx)=f(\vecx)-\sum_{U\subset[d]}f_U(\vecx_U)
$$
is uniquely determined.
\end{proof}

\begin{proof}[Proof of Proposition \ref{prop:mod_unique_bivar}]

Let $f$ be given by \eqref{eq:mod_unique_bivar} and let us assume that it satisfies all the assumptions of Proposition \ref{prop:mod_unique_bivar}.
We show that \eqref{eq:mod_unique_bivar} coincides with its Anchored-ANOVA decomposition as described in Proposition \ref{lem:anova} and \eqref{eq:mod_unique_bivar}
is therefore unique.

Let $U\subseteq[d]$ with $|U|\ge 3$. Then
\begin{align*}
f_U(\vecx_U)&=\sum_{V\subseteq U}(-1)^{|U|-|V|}(P_V f)(\vecx_V)\\
&=\sum_{V\subseteq U}(-1)^{|U|-|V|}\Bigl\{\mu+\sum_{j\in \univsupp\cup\bivsuppvar}\phi_j((x_V)_j)+\sum_{\vecj=(j_1,j_2)\in\bivsupp}\phi_\vecj((x_V)_{j_1},(x_V)_{j_2})\Bigr\}\\
&=\mu\sum_{V\subseteq U}(-1)^{|U|-|V|}+\sum_{j\in \univsupp\cup\bivsuppvar}\phi_j(x_j)\sum_{\substack{V\subseteq U\\ j\in V}}(-1)^{|U|-|V|}\\
&+\sum_{j\in \univsupp\cup\bivsuppvar}\phi_j(0)\sum_{\substack{V\subseteq U\\ j\not\in V}}(-1)^{|U|-|V|}
+\sum_{\vecj=(j_1,j_2)\in \bivsupp}\sum_{V\subseteq U}(-1)^{|U|-|V|}\phi_\vecj((x_V)_{j_1},(x_V)_{j_2}).
\end{align*}
It is easy to see, that the first three terms are zero. Finally, the last term can be split into a sum of four terms depending on if $j_1\in V$ or $j_2\in V$, i.e. terms of the kind
$$
\sum_{\vecj=(j_1,j_2)\in{\bivsupp}}\phi_\vecj(x_{j_1},x_{j_2})\sum_{\substack{V\subseteq U\\ j_1,j_2\in V}}(-1)^{|U|-|V|},
$$
which also vanish. Therefore, $f_U(\vecx_U)=0.$

If $U=\{j_1,j_2\}$ with $1\le j_1<j_2\le d$, then $\emptyset, \{j_1\}, \{j_2\}$ and $\{j_1,j_2\}$ are the only subsets of $U$ and we obtain
\begin{align*}
f_U(\vecx_U)&=\sum_{V\subseteq U}(-1)^{|U|-|V|}(P_V f)(\vecx_V)\\
&=(P_\emptyset f)(0)-(P_{j_1} f)(x_{j_1})-(P_{j_2} f)(x_{j_2})+(P_{\{j_1,j_2\}} f)(x_{j_1},x_{j_2})\\
&=f(0)-f(x_{j_1}\vece_{j_1})-f(x_{j_2}\vece_{j_2})+f(x_{j_1}\vece_{j_1}+x_{j_2}\vece_{j_2})\\
&=\sum_{j\in \univsupp\cup\bivsuppvar}[\phi_j(0)-\phi_j((x_{j_1}\vece_{j_1})_j)-\phi_j((x_{j_2}\vece_{j_2})_j)+\phi_j((x_{j_1}\vece_{j_1}+x_{j_2}\vece_{j_2})_j)]\\
&+\sum_{\vecj\in\bivsupp}[\phi_\vecj(0)-\phi_\vecj((x_{j_1}\vece_{j_1})_\vecj)-\phi_\vecj((x_{j_2}\vece_{j_2})_\vecj)+\phi_\vecj((x_{j_1}\vece_{j_1}+x_{j_2}\vece_{j_2})_\vecj)].
\end{align*}
The first sum vanishes for all $j\in \univsupp\cup\bivsuppvar$, which can be easily observed
by considering the options $j\not\in\{j_1,j_2\}, j=j_1$, or $j=j_2.$
If $(j_1,j_2)\not\in \bivsupp$, then also the second sum vanishes. If $(j_1,j_2)\in\bivsupp$,
then
$$
f_{j_1,j_2}(x_{j_1},x_{j_2})=\phi_{j_1,j_2}(x_{j_1},x_{j_2})-\phi_{j_1,j_2}(x_{j_1},0)-\phi_{j_1,j_2}(0,x_{j_2})+\phi_{j_1,j_2}(0,0)=\phi_{j_1,j_2}(x_{j_1},x_{j_2}).
$$

Finally, if $l\in[d]$ then $\emptyset$ and $\{l\}$ are the only subsets of $\{l\}$. If furthermore 
$l\not\in \univsupp\cup\bivsuppvar$, we get
\begin{align*}
f_{\{l\}}(x_l)&=-f(0)+f(x_l\vece_l)\\
&=-\Bigl(\mu+\sum_{j\in \univsupp\cup\bivsuppvar}\phi_j(0)+\sum_{\vecj\in \bivsupp}\phi_\vecj(0)\Bigr)+
\Bigl(\mu+\sum_{j\in \univsupp\cup\bivsuppvar}\phi_j(0)+\sum_{\vecj\in\bivsupp}\phi_\vecj(0)\Bigr)=0.
\end{align*}
Similarly, $f_{\{l\}}(x_l)=\phi_l(x_l)$ if $l\in \univsupp\cup\bivsuppvar$.
\end{proof}

\section{Some standard concentration results} \label{app:std_conc_results}
First, we recall that the sub-Gaussian norm of a random variable $X$ is defined as
\begin{align*}
\norm{X}_{\psi_2}=\sup_{p\ge 1}p^{-1/2}(\expec |X|^p)^{1/p}
\end{align*}
and $X$ is called sub-Gaussian random variable if $\norm{X}_{\psi_2}$ is finite. Similarly, the sub-exponential
norm of a random variable is the quantity
\begin{align*}
\norm{X}_{\psi_1}=\sup_{p\ge 1}p^{-1}(\expec |X|^p)^{1/p}
\end{align*}
and $X$ is called sub-exponential if $\norm{X}_{\psi_1}$ is finite.
Next, we recall the following concentration results for sums of i.i.d. sub-Gaussian and
sub-exponential random variables.
\begin{proposition}\cite[Proposition 5.16]{vershynin2012} \label{prop:subexp_conc}
Let $X_1,\dots,X_N$ be independent centered sub-exponential random variables, and $K = \max_i \norm{X_i}_{\psi_1}$.
Then for every $\veca = (a_1,\dots,a_N) \in \matR^N$ and every $t \geq 0$, we have
$$ \prob \left( \Bigl|\sum_i a_i X_i\Bigr| \geq t \right) \leq 2 \exp\left[-c\min\set{\frac{t^2}{K^2\norm{\veca}_2^2},\frac{t}{K \norm{\veca}_{\infty}}}\right],$$
where $c > 0$ is an absolute constant.
\end{proposition}
\begin{proposition}\cite[Proposition 5.10]{vershynin2012} \label{prop:subgauss_conc}
Let $X_1,\dots,X_N$ be independent centered sub-Gaussian random variables, and $K = \max_i \norm{X_i}_{\psi_2}$.
Then for every $\veca = (a_1,\dots,a_N) \in \matR^N$ and every $t \geq 0$, we have
$$ \prob \left( \Bigl|\sum_i a_i X_i\Bigr| \geq t \right) \leq e \cdot \exp\left[- \frac{c t^2}{K^2\norm{\veca}_2^2}\right],$$
where $c > 0$ is an absolute constant.
\end{proposition}
%
%

Let $\noisevec = (\noise_1 \ \noise_2 \ \dots \ \noise_n)^T$, where $\noise_i \sim \calN(0,\sigma^2)$ are i.i.d. for each $i$.
The Proposition below is a standard concentration result, stating that $\norm{\noisevec}_2 = \Theta(\sigma \sqrt{n})$ and
$\norm{\noisevec}_1 = \Theta(\sigma n)$, with high probability. We provide proofs for completeness.
%
\begin{proposition} \label{prop:gaussian_norm_bd}
Let $\noisevec = (\noise_1 \ \noise_2 \ \dots \ \noise_n)^T$ where $\noise_i \sim \calN(0,\sigma^2)$ i.i.d for each $i$. 
Then, there exists constants $c_1, c_2 > 0$ so that for any $\epsilon \in (0,1)$, we have:
\begin{enumerate}
\item $\prob\left(\norm{\noisevec}_2 \in \left[(1-\epsilon)\sigma\sqrt{n}, (1+\epsilon)\sigma\sqrt{n} \ \right] \right) 
\geq 1-2\exp\left(- c_1 \epsilon^2 n \right)$, and \label{prop:gauss_l2norm_bd}

\item $\prob\left(\norm{\noisevec}_1 \in \left[(1-\epsilon) n \sigma \sqrt{\frac{2}{\pi}}, (1+\epsilon)n \sigma \sqrt{\frac{2}{\pi}}\right]\right) \geq 
1- e \cdot \exp\left(-\frac{2 c_2 \epsilon^2 n}{\pi}\right)$. \label{prop:gauss_l1norm_bd}
\end{enumerate}
\end{proposition}
%
%
\begin{proof}
\begin{enumerate}
\item Note that $\noise_i$ are i.i.d. sub-Gaussian random variables with $\norm{\noise_i}_{\psi_2} \leq C_1 \sigma$. 
Hence $\noise_i^2$ are i.i.d. sub-exponential\footnote{A random variable $X$ is sub-Gaussian iff $X^2$ is sub-exponential. 
Moreover, $\norm{X}_{\psi_2}^2 \leq \norm{X^2}_{\psi_1} \leq 2\norm{X}_{\psi_2}^2$ (cf. \cite[Lemma 5.14]{vershynin2012}).} with 
$$\norm{\noise_i^2}_{\psi_1} \leq 2 \norm{\noise_i}_{\psi_2}^2 \leq C_2 \sigma^2$$
for some constant $C_2 > 0$. This implies that $\noise_i^2 - \expec[\noise_i^2]$ are i.i.d. sub-exponential with 
$$\norm{\noise_i^2 - \expec[\noise_i^2]}_{\psi_1} \leq 2\norm{\noise_i^2}_{\psi_1} \leq C_3\sigma^2$$ 
for some constant $C_3 > 0$. Using Proposition \ref{prop:subexp_conc} with $t=n\epsilon\sigma^2$ for $0<\epsilon<1$,
we obtain for some constant $c_1 > 0$
$$\prob\left(\Bigl|\sum_{i=1}^{n} (\noise_i^2 - \expec[\noise_i^2])\Bigr| \leq n \epsilon \sigma^2 \right) \geq 
1-2\exp\left[-c_1 \min\set{\epsilon^2, \epsilon} n \right].$$
%
Together with some standard manipulations this completes the proof.

\item Note that $\abs{\noise_i} - \expec[\abs{\noise_i}]$ is sub-Gaussian with
$$\norm{\abs{\noise_i} - \expec[\abs{\noise_i}]}_{\psi_2} \leq 2 \norm{\noise_i}_{\psi_2} \leq C \sigma,$$
for some constant $C > 0$. Using Proposition \ref{prop:subgauss_conc} with $t = n \epsilon \expec[\abs{\noise_1}]$ for $\epsilon > 0$, 
we hence obtain
$$\prob\left(\Bigl|\sum_{i=1}^{n} (\abs{\noise_i} - \expec[\abs{\noise_i}])\Bigr| \leq n \epsilon \expec[\abs{\noise_1}] \right) 
\geq 1-e \cdot \exp\left[- \frac{c_2 n \epsilon^2 (\expec[\abs{\noise_1}])^2}{\sigma^2} \right]$$
for some constant $c_2 > 0$. Finally, observing that $\expec[\abs{\noise_1}] = \sigma \sqrt{\frac{2}{\pi}}$ (cf., \cite{winkel14}) 
completes the proof.
\end{enumerate}
\end{proof}

\section{Proof of Theorem \ref{thm:arbnoise_bilin_rec}} \label{app:proof_thm_sparse_bilin}
The proof of part ($1$) follows in the same manner as \cite[Proposition 1]{Chen15}, with minor 
differences in calculation at certain parts. For completeness, we outline the main steps below.

Invoking the Hanson-Wright inequality for quadratic forms \cite{Hanson1971}, we get for some constant $c > 0$ and all $t>0$
\begin{equation}\label{eq:Hanson-Wright1}
\prob(\abs{\sampsgnvec^T\matA \sampsgnvec - \expec[\sampsgnvec^T\matA \sampsgnvec]} > t) 
\leq 2\exp\biggl[-c \min \set{\frac{t^2}{K^4 \norm{\matA}_F^2}, \frac{t}{K^2\norm{\matA}}}\biggr],
\end{equation}
where $\norm{\sampsgn_i}_{\psi_2} \leq K$. For a \randvar random variable $\sampsgn_i$, $K = 1$ and
\begin{align*}
\expec[\sampsgnvec^T\matA\sampsgnvec]=\expec\sum_{i\not=j}\sampsgn_i\sampsgn_jA_{ij}=0.
\end{align*}
Using $\norm{\matA}\le\norm{\matA}_F$, \eqref{eq:Hanson-Wright1} implies for every $t>0$
\begin{equation}\label{eq:Hanson-Wright2}
\prob(\abs{\sampsgnvec^T\matA \sampsgnvec} > t) 
\leq 2\exp\biggl[-c \min \set{\frac{t^2}{\norm{\matA}_F^2}, \frac{t}{\norm{\matA}_F}}\biggr].
\end{equation}

We will now find upper and lower bounds on $\expec[\abs{\sampsgnvec^T\matA \sampsgnvec}]$. 
The upper bound is easy since \eqref{eq:Hanson-Wright2} implies
\begin{equation} \label{eq:up_bd_expec}
\expec[\abs{\sampsgnvec^T\matA \sampsgnvec}]=\int_0^\infty \prob(\abs{\sampsgnvec^T\matA \sampsgnvec}>t) dt\leq 
c'\norm{\matA}_F.
\end{equation}
In order to find the lower bound, we have via repeated application of Cauchy Schwartz inequality 
$$
\bigl(\expec[\abs{\sampsgnvec^T\matA \sampsgnvec}^2]\bigr)^2\le \expec[\abs{\sampsgnvec^T\matA \sampsgnvec}]\cdot\expec[\abs{\sampsgnvec^T\matA \sampsgnvec}^3]
\le \expec[\abs{\sampsgnvec^T\matA \sampsgnvec}]\cdot\bigl(\expec[\abs{\sampsgnvec^T\matA \sampsgnvec}^2]\bigr)^{1/2} \cdot
\bigl(\expec[\abs{\sampsgnvec^T\matA \sampsgnvec}^4]\bigr)^{1/2}
$$
and
$$\expec[\abs{\sampsgnvec^T\matA \sampsgnvec}] \geq 
\sqrt{\frac{(\expec[\abs{\sampsgnvec^T\matA \sampsgnvec}^2])^3}{\expec[\abs{\sampsgnvec^T\matA \sampsgnvec}^4]}}.$$
Since $\sampsgnvec$ consists of i.i.d. \randvar random variables, we obtain
$$
\expec[\abs{\sampsgnvec^T\matA \sampsgnvec}^2] = 2\norm{\matA}_F^2 = \norm{\veca}_2^2.
$$ 
Moreover, an argument similar to \eqref{eq:up_bd_expec} gives $\expec[\abs{\sampsgnvec^T\matA \sampsgnvec}^4]\le c''\norm{\matA}_F^4$.
Hence, 
\begin{equation} \label{eq:low_bd_expec}
    \expec[\abs{\sampsgnvec^T\matA \sampsgnvec}] \geq \sqrt{\frac{8\norm{\matA}_F^6}{c'' \norm{\matA}_F^4}} = \tilde c \norm{\matA}_F.
\end{equation}
Eqs. \eqref{eq:up_bd_expec}, \eqref{eq:low_bd_expec} give us upper and lower bounds for 
$\expec[\abs{\sampsgnvec^T\matA \sampsgnvec}]$. As a last step, we consider the 
zero mean random variables $X_1, \dots, X_n$, where 
$X_i = \abs{\sampsgnvec_i^T\matA \sampsgnvec_i} - \expec{\abs{\sampsgnvec_i^T\matA \sampsgnvec_i}}$.
A simple modification of \eqref{eq:up_bd_expec} shows that they are sub-exponential with $\norm{X_i}_{\psi_1}\le c\norm{\matA}_F.$
We can therefore apply a standard concentration bound (see \cite[Proposition 5.16]{vershynin2012}) to bound the deviation
$$
\biggl|\frac{1}{n}\norm{\matB\veca}_1 - \frac{1}{n}\expec[\norm{\matB\veca}_1]\biggr|=\frac{1}{n}\biggl|\sum_{i=1}^nX_i\biggr|.
$$
A straightforward  calculation then yields the statement of part (1) of the Theorem.

Part (2) follows from standard arguments 
based on $\epsilon$-nets, detailed for instance in \cite{Baraniuk2008_simple}.

The proof of part (3) copies that of \cite[Theorem 3]{Chen15}, which again is inspired by \cite{Candes08restr}. We sketch the 
main steps below for completeness. Denoting $\est{\veca} = \veca + \vech$, we have by feasibility of $\veca$ that
$\frac{1}{n}\norm{\matB\vech}_1 \leq \frac{2\arbnoisebd}{n}$. Denoting $\Omega_0$ to be set of indices corresponding to 
the $\totsparsity$ largest entries of $\veca$, we get ${\veca} = \veca_{\Omega_0} + \veca_{\Omega_0^c}$. For a suitable positive integer 
$K$, we define $\Omega_1$ as the set of indices of $K$ largest entries of $\vech_{\Omega_0^{c}}$, $\Omega_2$ as the set of indices of $K$ largest entries of 
$\vech$ on $(\Omega_0 \cup \Omega_1)^{c}$ and so on. Following this argument of \cite{Candes08restr} gives the proof.
%
%
%
%
%
%
%
%

The proof of part (4) follows by choosing $K = 4 \bigl(\frac{4 c_2}{c_1}\bigr)^2\totsparsity$. Indeed, \eqref{eq:gammas} gives
$$
\frac{1-\riptolb_{\totsparsity+K}}{\sqrt{2}} - (1+\riptoub_{K})\sqrt{\frac{k}{K}}\ge \frac{c_1}{2\sqrt{2}}-2c_2\cdot\frac{c_1}{8c_2}=\frac{(\sqrt{2}-1)c_1}{4}=\beta>0
$$
if $n>c_3'(k+K)\log(d^2/(k+K))$, 
with probability at least $1-e^{-C_4 n}$ for some constant $C_4 > 0$.

\end{document}